\documentclass[a4paper, 11pt]{article}

\usepackage[
            includefoot,  
            marginparwidth=0in,     
            marginparsep=0in,       
            margin=1.45in,               
            includemp]{geometry}

\usepackage{bbm}
\usepackage{float}
\usepackage{graphicx}                  
\usepackage{amssymb}
\usepackage{amsfonts}
\usepackage{xcolor}

\RequirePackage{amsmath}
\RequirePackage{amsthm}

\RequirePackage{latexsym}
\RequirePackage{amsfonts}
\RequirePackage{amsmath}
\RequirePackage{amssymb}
\RequirePackage{amsthm}

\RequirePackage[backrefs]{amsrefs}

\newtheoremstyle{montheoreme}
  {}
  {}
  {\itshape}
  {}
  {\bf}
  {.}
  {.5em}
  {}
\newtheoremstyle{maremarque}
  {}
  {}
  {}
  {}
  {\bf}
  {.}
  {.5em}
  {}
\usepackage{fancyhdr}
\theoremstyle{montheoreme}

\newtheorem{thm}{Theorem}[section]
\newtheorem{defn}[thm]{Definition}
\newtheorem{prop}[thm]{Proposition}
\newtheorem{lem}[thm]{Lemma}

\theoremstyle{maremarque}
\newtheorem*{rmq}{Remark}{}
\newtheorem*{nota}{Notation}

\DeclareMathOperator{\Hess}{Hess}
\DeclareMathOperator{\osc}{osc}

\DeclareMathOperator{\cov}{cov}
\DeclareMathOperator{\Ent}{Ent}

\newcommand{\R}{\mathbb{R}}
\newcommand{\N}{\mathbb{N}}
\newcommand{\T}{\mathbb{T}}

\makeatletter

\@addtoreset{equation}{chapter}
\makeatother

\makeatletter

\@addtoreset{equation}{section}
\makeatother

\begin{document}

\title{Hydrodynamic limit for conservative spin systems with super-quadratic, partially inhomogeneous single-site potential.}
\author{Max Fathi \& Georg Menz  \thanks{LPMA, University Paris 6, France, max.fathi@etu.upmc.fr.\newline  Stanford University, gmenz@stanford.edu.} }
\date{\today}

\maketitle

\begin{abstract}
  We consider an interacting unbounded spin system, with conservation of the mean spin. We derive quantitative rates of convergence to the hydrodynamic limit provided the single-site potential is a bounded perturbation of a strictly convex function with polynomial growth, and with an additional random inhomogeneous linear term. This additional linear term models the impact of a random chemical potential. The argument adapts the two-scale approach of Grunewald, Otto, Villani and Westdickenberg from the quadratic to the general case. The main ingredient is the derivation of a covariance estimate that is uniform in the system size. We also show that this covariance estimate can be used to change the iterative argument of [MO] for deducing the optimal scaling LSI for the canonical ensemble into a two-scale argument in the sense of [GOVW]. We also prove the LSI for canonical ensembles with an inhomogeneous linear term.  \bigskip

{\begin{footnotesize}
\noindent\emph{MSC:} Primary 60K35; secondary 60J25; 82B21 \newline
\emph{Keywords:} Hydrodynamic limit; Kawasaki dynamics; Spin system; Coarse-graining
  \end{footnotesize}
}

\end{abstract}

\tableofcontents

\vspace{1cm}

\section{Introduction and Background}

In this work, we are interested in generalizing the results of [GOVW] to a larger setting. There are two topics we are interested in : logarithmic Sobolev inequalities and quantitative rates of convergence to the hydrodynamic limit for continuous spin systems on a lattice. We use the two-scale approach introduced in [GOVW], which is based on coarse-graining arguments.

\subsection{Logarithmic Sobolev inequalities}

\begin{defn}
A probability measure $\mu$ on a Riemannian manifold $X$ is said to satisfy a logarithmic Sobolev inequality with parameter $\rho > 0$ if, for all smooth, non-negative functions $f : X \rightarrow \R_+$, we have
$$\Ent_{\mu}(f) \leq \frac{1}{\rho}\int{\frac{|\nabla f|^2}{f}\mu}.$$
\end{defn}

This inequality can be very useful in the study of diffusion processes. It was originally introduced by Gross in [Gr] to study hypercontractivity phenomenon for diffusion processes. For example, it is well known that such an inequality implies exponentially fast convergence to equilibrium for the diffusion equation naturally associated with $\mu$ (namely $\Delta - \nabla V$, when $\mu = \exp(-V)$). Due to tensorization properties (see Criterion I in Appendix B), it also behaves well in large dimensions. For example, it has been proven in [Go] that it is equivalent to dimension-free Gaussian concentration. A well-written introduction to this topic can be found in [L], and we recall the three most common criterion that are used to derive such inequalities in Appendix B. 

In [GOVW], a new criteria, based on a decomposition of the system in two scales, is devised. It is then applied to prove the LSI for the canonical ensemble with a bounded perturbation of a quadratic Ginzburg-Landau type potential. Due to a technical restriction on a covariance estimate they used, their criteria could not be directly applied to the case of a superquadratic Ginzburg-Landau type potential. To overcome this problem, the two-scale strategy was modified into an iterative approach in [MO]. In this paper, we will devise a new covariance estimate (Proposition 2.4), which will allow us to apply the original two-scale approach. We also obtain the LSI when we add an inhomogeneous linear term to the Hamiltonian, which generalizes to the superquadratic case results of [M] and [LN].

\subsection{Hydrodynamic limits}

Our second theme of interest is the derivation of scaling limits for stochastic interacting spin systems. Typically, such results consist in proving that, under a suitable time-space scaling, random systems with a large number of particles behave like a deterministic object, given as the solution of a partial differential equation.

There are two main general methods in the literature to prove hydrodynamic limit results for stochastic systems. The first, often called the entropy method, was introduced in [GPV], and is based on martingale methods and entropy estimates. The second, called the relative entropy method, was devised in [Y], and is based on proving a Gronwall-type inequality for the relative entropy with respect to a well-chosen time-varying reference state. It yields stronger results than the entropy method, and is in general simpler, but relies on stronger assumptions, namely convergence of the initial data in the sense of relative entropy rather than of the macroscopic observables. Both methods are reviewed in [KL].

In [GOVW], a new method is presented, based on estimates in Wasserstein distances. Its assumptions as well as its results, are at an intermediate level between the two classical methods, and uses a logarithmic Sobolev inequality. It has the additional interest of giving explicit quantitative rates of convergence. This method was then applied to Kawasaki dynamics with an equilibrium measure that is a bounded perturbation of a quadratic potential. Here we will generalize their results to the case where the single-site potential is a bounded perturbation of a uniformly convex function with polynomial growth, and with an additional inhomogeneous linear term, with coefficients given by bounded iid random variables. The additional linear term corresponds to the effect of a random chemical potential. It seems that convergence to the hydrodynamic limit when the Hamiltonian includes such a random chemical potential is entirely new.

The plan of the paper is as follows: In Section 2, we present the framework and our main results. Section 3 is devoted to the new covariance estimate. Section 4 presents the proof of the logarithmic Sobolev inequality, and Section 5 contains the proofs of our results on hydrodynamic limits.

\vspace{1cm}

{\Large\textbf{Notations}}

\begin{itemize}

\item $\mathcal{L}^k$ is the k-dimensional Lebesgue measure

\item $\nabla$ is the gradient, $\Hess$ stands for Hessian, $| \cdot |$ is the norm and $\langle \cdot, \cdot \rangle$ is an inner product. If necessary, a subscript will indicate the space on which these are taken.

\item $A^t$ is the transpose of the operator $A$.

\item $\Ent_{\mu}(f) = \int{f (\log f) \mu} - \left(\int{f \mu} \right) \log \left( \int{f \mu} \right)$ is the entropy of the positive function $f$ with respect to the measure $\mu$.

\item $\osc{\Psi} = \sup_X \Psi - \inf_X \Psi$ is the oscillation of the function $\Psi$.

\item $Z$ is a constant enforcing unit mass for a probability measure.

\item $C$ is a positive constant, which may vary from line to line, or even within a line.

\item $id_X$ is the identity map $X \rightarrow X$.

\item LSI is an abbreviation for logarithmic Sobolev inequality.

\item $||f||_{\infty, E} = \underset{x \in E}{\sup} \hspace{1mm} |f(x)|$ is the supremum norm of $f$ over the set $E$.

\item $a \lesssim b$ means that there is a constant $C$ such that $a \leq Cb$.

\end{itemize}

\section{Framework and Main Results}

\subsection{Microscopic and macroscopic scales}

We consider a (periodic) lattice spin system of N continuous variables governed by a Ginzburg-Landau type potential $\psi : \R \rightarrow \R$ and an additional inhomogeneous linear term given by a family of real numbers $a =(a_i)_{1 \leq i \leq N}$. The grand canonical measure on $\R^N$ has density

\begin{equation} 
\frac{d\mu_{N,a}}{d\mathcal{L}^N}(x) = \frac{1}{Z}\exp \left( - \underset{i = 1}{\stackrel{N}{\sum}} \hspace{1mm} \psi(x_i) + a_i x_i \right).
\end{equation}

We shall assume that the potential $\psi$ is of class $C^2$ and is a bounded perturbation of a strictly convex potential, that is
\begin{equation} \label{assumption_potential}
\psi(x) = \psi_c(x) + \delta\psi(x); \hspace{3mm} \psi_c'' \geq \lambda > 0; \hspace{3mm} ||\delta\psi||_{C^2} < +\infty.
\end{equation}
We shall also assume that 
\begin{equation} \label{assumption_bound_a}
\sup |a_i| \leq L < +\infty.
\end{equation}

To simplify notations, we define the Hamiltonian 

\begin{equation}
H(x) := \underset{i = 1}{\stackrel{N}{\sum}} \hspace{1mm} \psi(x_i) + a_i x_i - \log Z,
\end{equation}
so that $\mu(dx) = \exp(-H(x))dx$.

We will add to the situation a constraint of fixed mean spin. Our phase state space will be
 \begin{equation*}
    \label{e_definition_of_X_N_M}
   X_{N,m} := \left\{  x \in \R^N, \ \frac{1}{N} \sum_{i=1}^N x_i =m   \right\}.
 \end{equation*}
 
where $m$ is an arbitrary real number. This space is a hyperplane of $\R^N$ with a fixed mean constraint. We endow this space with the $\ell^2$ inner product
\begin{equation}
\langle x, \tilde{x} \rangle_X = \underset{i = 1}{\stackrel{N}{\sum}} \hspace{1mm} x_i \tilde{x}_i.
\end{equation}

For a given $m \in \R$, we consider the restriction $\mu_{N,a,m}$ of the grand canonical measure to $X_{N,m}$, that is
\begin{equation} \label{def_grand_can}
\frac{d\mu_{N,a,m}}{d\mathcal{L}^{N-1}}(x) = \frac{1}{Z}1_{(1/N)\sum x_i = m}\exp \left( - \underset{i = 1}{\stackrel{N}{\sum}} \hspace{1mm} \psi(x_i) + a_i x_i \right).
\end{equation}

This measure is called the canonical ensemble. It gives the distribution of the random variables $x_i$ conditioned on the event that their mean value is given by $m$.

We now consider two integers $M$ and $K$ such that $N = KM$. To each microscopic profile $x \in X$, we associate a macroscopic profile $y \in \R^M$, given by

\begin{equation}
y_i := \frac{1}{K}\underset{j = (i - 1)K + 1}{\stackrel{iK}{\sum}} \hspace{1mm} x_j.
\end{equation}

We denote by $P_{N,K}$ the operator that associates the macroscopic profile $y$ to the microscopic profile $x$. The macroscopic profiles live in the space

\begin{equation}
Y_{M,m} := \left\{y \in \R^M , \ \frac{1}{M} \sum_{i=1}^N y_i =m \right\},
\end{equation}

which we endow with the $L^2$ inner product

\begin{equation}
\langle y, \tilde{y} \rangle_Y := \frac{1}{M}\sum y_i \tilde{y}_i.
\end{equation}

\begin{rmq}
Note that we are \emph{not} using the same scalar product for both the microscopic and the macroscopic spaces.
\end{rmq}

If we denote by $X_{N,K,y}$ the space of microscopic profiles such that $Px = y$, we have the factorization

\begin{equation} \label{facto_fluct}
X_{N,K,y} = \underset{i = 1}{\stackrel{M}{\bigotimes}} \hspace{1mm} X_{M,y_i},
\end{equation}
and $P_{N,K}$ factorizes on the fibers $X_{N,K,y}$.

An important feature of the coarse-graining operator $P_{N,K}$ is the relation
\begin{equation}
PNP^t = id_Y,
\end{equation}
so that the symmetric operator $P^tNP$ is the orthogonal projection of $X$ to $(\operatorname{Ker } P)^{\perp}$. This induces a decomposition of $X$ into macroscopic variables in $Y \approx (\operatorname{Ker } P)^{\perp}$ and microscopic fluctuations in $\operatorname{Ker } P$, as $x = P^tNPx + (\text{id} - P^tNP)x$. This decomposition of variables induces a decomposition of measures and a decomposition of gradients :

\vspace{2mm}

(i) a decomposition of measures : We denote by $\bar{\mu}$ the image of the measure $\mu$ by $P$. This measure is the distribution of the macroscopic variable when the microscopic variable is distributed according to $\mu$. Let $\mu(dx|y)$ be the distribution of the microscopic variable conditioned on the event $\{ Px = y\}$. For each $y$, $\mu(dx|y)$ is a probability measure, and we have

\begin{equation}
\mu(dx) = \mu(dx|y)\bar{\mu}(dy).
\end{equation}

We can define a macroscopic Hamiltonian associated to $\bar{\mu}$ by

\begin{equation}
\bar{H}(y) := -\frac{1}{N} \log \left(\frac{d\bar{\mu}}{dy} \right),
\end{equation}

so that $\bar{\mu}(dy) = \exp(-N\bar{H}(y))dy$. We can explicitly compute $\bar{H}$ as 

\begin{align}
\bar{H}(y) &= -\frac{1}{N}\left(\log \int_{\{Px = y\}}{\exp(-H(x))dx} - \log \bar{Z} \right) \notag \\
&= -\frac{1}{N} \underset{i = 1}{\stackrel{M}{\sum}} \hspace{1mm}  \log \int_{X_{K,y_i}}{\exp \left(-\underset{j = (i-1)K+1}{\stackrel{iK}{\sum}} \psi(x_j) + a_j x_j \right)dx} - \frac{1}{N}\log \bar{Z}. \notag
\end{align}

If we define 

\begin{equation}
\psi_{K,a_1..a_K}(m) := -\frac{1}{K}\log \int_{X_{K,m}}{\exp(-\sum a_i x_i + \psi(x_i))dx},
\end{equation}

we get

\begin{equation} \label{equn_bar_h_psi}
\bar{H}(y) = \frac{1}{M}\underset{i = 1}{\stackrel{M}{\sum}} \hspace{1mm}  \psi_{K,a_{(i-1)K+1},..a_{iK}}(y_i) - \frac{1}{N}\log \bar{Z}.
\end{equation}

Also, the fluctuations measure $\mu(dx|y)$ on $X_{N,K,y}$ can be decomposed as a product measure along the factorization (\ref{facto_fluct})

\begin{equation} \label{factorization_measure_fluct}
\mu_{N,a,m}(dx|y) = \underset{i = 1}{\stackrel{M}{\bigotimes}} \hspace{1mm} \mu_{K,(a_j)_{(i-1)K+1 \leq j \leq iK}, y_i}
\end{equation}

(ii) a decomposition of gradients : If $f$ is a smooth function on $X$, we can decompose its gradient into a macroscopic gradient and a fluctuation gradient:

\begin{equation}
\nabla^{\text{macro}} f(x) = NP^tP\nabla f(x), \hspace{5mm} \nabla^{\text{fluct}} f(x) = (id_X - NP^tP)\nabla f(x).
\end{equation} 

This decomposition makes sense, since $(id_X - NP^tP)$ is the orthogonal projection onto $\text{ker } P$, which is the tangent space to the fiber $\{Px = y\}$.

An important use of the introduction of a macroscopic scale will be that, for large enough $K$, $\bar{H}$ will be uniformly convex : 

\begin{prop}\label{p_convexity_coarse_grained_Hamiltonian}
Assume that \eqref{assumption_potential} and \eqref{assumption_bound_a} hold. There exists an integer $K_0$ and a positive real number $\lambda$ such that, for any $K \geq K_0$ and $M \in \N$, we have for any $y$, $\tilde{y}$ $\in Y$,
$$\langle \Hess \bar{H}(y) \tilde{y}, \tilde{y} \rangle_Y \geq \lambda |\tilde{y}|_Y^2.$$
\end{prop}

This result is a consequence of a local Cram\`er theorem, which we prove in Appendix A.

\subsection{A new covariance estimate}

To explain our results, let us first define the covariance of two functions.

\begin{defn}
The covariance of a non-negative function $f$ and a vector-valued function $g$ with respect to a measure $\mu$ is given by
$$\cov_{\mu}(f,g) = \int{gf \hspace{1mm} d\mu} - \left(\int{f \hspace{1mm} d\mu}\right) \left(\int{g \hspace{1mm} d\mu} \right).$$
\end{defn}

In [GOVW], the following covariance estimate to bound some crucial terms which appeared because of the interaction between the microscopic and macroscopic scales.

\begin{prop}
Let $\mu$ be a probability measure on $X$ satisfying $LSI(\rho)$ for some $\rho > 0$. Then for any two Lipschitz functions $f : X \rightarrow \R_+$ and $g : X \rightarrow \R$, we have
$$|\cov_{\mu}(f,g) \leq ||\nabla g||_{L^{\infty}} \sqrt{\frac{2}{\rho}\int{f d\mu} \Ent_{\mu}(f)}.$$
\end{prop}

To be usefully applied, this proposition requires $g$ to have a bounded gradient. Since this result was used in [GOVW] with $g = \nabla H$, this restricted the method of [GOVW] to the case where the function $\psi$ they used as a single-site potential has bounded Hessian, and therefore it didn't apply to the superquadratic case. For the logarithmic Sobolev inequality, this problem was circumvented in [MO] through the use of an iterated coarse-graining, but this strategy cannot be transported very well for hydrodynamic limits, where the use of a finite number of scales is crucial. Therefore we devised a new covariance estimate, which applies well to our problem :

\begin{prop}\label{p_crucial_covariance_two_scale}
  Assume that the single-site potential $\psi$ is perturbed strictly convex in the sense that there is a splitting $\psi= \psi_c + \delta \psi$ such that
\begin{equation}
  \label{e_perturbation_convex}
  \psi_c'' \gtrsim 1 \qquad  \mbox{and} \qquad |\delta \psi| + |\delta \psi ' | \lesssim 1. 
\end{equation}
Then the canonical ensemble $\mu_{K,m}$ satisfies for any nonnegative function $f \geq0$ with $\int f d \mu_{K,m} =1$ the following covariance estimate: 
\begin{equation*}
 \left| \cov_{\mu_{K,a,m} } \left(f, \frac{1}{K} \sum_{i=1}^K  \psi'(x_i) \right)  \right|^2 \leq C_0 \hspace{1mm} K \hspace{1mm} \left( \int \frac{\sum_{i=1}^{K-1} \left| \frac{d }{dx_i} f \right|^2}{f}   \ d \mu_{K,a,m} \right)
\end{equation*}
for some constant $C_0$ which only depends on $\psi$ and $L$.
\end{prop}

Note that, in the case of a quadratic potential, this covariance bound is suboptimal, since it has an extra $K$ factor. However, when used in the proof of the hydrodynamic limit, this extra factor doesn't change the optimal order of magnitude of the quantitative bounds (in $N^{-1/2}$) which is obtained when taking the best possible choices for $K$ and $M$.

\subsection{Logarithmic Sobolev inequality}

One of our main results is to prove that, when the potential $\psi$ is a bounded perturbation of a strictly convex function, the (generalized) canonical ensemble with inhomogeneous linear part, as defined by (\ref{def_grand_can}), satisfies a LSI with a parameter $\rho$ that doesn't depend on the dimension $N$, the mean $m$ or the $a_i$. 

\begin{thm} \label{lsi}
Assume that $\psi$ is a bounded perturbation of a strictly convex potential, as in (\ref{assumption_potential}). Let $a_1,..,a_N$ be a collection of $N$ real numbers, such that
$$\underset{i}{\sup} \hspace{1mm} |a_i| \leq L.$$
The measure $\mu_{N,a,m}$ satisfies LSI($\rho$) for some constant $\rho > 0$ that is independent of the dimension $N$, the mean $m$ and the sequence $(a_i)$. However, it does depend on the bound $L$.
\end{thm} 

This result constitutes an extension of the main result of [MO] (which generalized a result of [LPY] with a different method), where they didn't add an inhomogeneous linear part. This result doesn't rule out a dependence of $\rho$ on $L$, although it doesn't in the case of grand canonical measures without a conservation law. The case where $\psi$ is a bounded perturbation of a quadratic potential was already addressed in [M] by using the two-scale approach. It should be noted that the assumption of boundedness of the $a_i$ isn't used in [M], and the constant obtained still doesn't depend at all on the $a_i$. Therefore it is possible that our assumption of boundedness on the $a_i$ is not optimal. However, the local Cram\'er theorem we use to obtain convexity of the coarse-grained Hamiltonian does not hold for unbounded $a_i$, so it seems our method cannot be extended to cover that case.

Another approach to this type of result was also proposed in [BM], but does not obtain the result of [MO] in full generality. It consists in using curvature lower bounds and Gaussian concentration estimates to recover the LSI.

The logarithmic Sobolev inequality for Kawasaki dynamics (which is weaker than the LSI we prove here, see [Ch]) when the $a_i$ are bounded iid random variables and for a quadratic single-site potential was also addressed in [LN].

\subsection{Hydrodynamic limit for Kawasaki dynamics}

We shall now consider the (reversible) stochastic dynamics on $X_{N,m}$ described by the time-evolution

\begin{equation} \label{micro_evolution}
\frac{\partial}{\partial t}(f\mu) = \nabla \cdot \left( A\nabla f \mu_{N,a,m} \right),
\end{equation}

which is to be understood in a weak sense : for any smooth test function $\xi$, we have

$$\frac{d}{dt} \int{\xi(x)f(t,x)\mu_{N,a,m}(dx)} = -\int{\nabla \xi(x) \cdot A\nabla f(t,x) \mu_{N,a,m}(dx)}.$$

In this equation, the operator $A$ is an $N \times N$ matrix defined by

\begin{equation}
A_{i,j} := N^2(-\delta_{i,j-1} + 2\delta_{i,j} - \delta_{i,j+1}).
\end{equation}

In this definition, $\delta$ is the Kronecker symbol and by convention $N+1 = 1$. $A$ is a symmetric, positive operator when restricted to $X_{N,0}$. We start with an initial condition $f_0$ such that $f_0\mu$ is a probability measure. Then $f(t,\cdot)\mu$ represents the law at time $t$ of the solution of the stochastic differential equation
$$dX_t = -A\nabla H(X_t)dt + \sqrt{2A}dB_t$$
with initial value distributed according to $f_0\mu$, and where $B_t$ is a standard Brownian motion on $\R^N$. 

Our first result is a quantitative bound on how close a typical realization of the diffusion is to a deterministic vector, whose time-evolution will be given by the following equation :

\begin{equation} \label{macro_evol}
\frac{d\eta}{dt} = -\bar{A}\nabla \bar{H}_{N,K,a}(\eta(t))
\end{equation}

where the (positive, symmetric) operator $\bar{A}$ on $Y$ is defined through

\begin{equation}
\bar{A}^{-1} = PA^{-1}NP^t.
\end{equation}

Note that $\eta$ depends on the $a_i$.

Our main abstract result on the hydrodynamic limit is the following. 

\begin{thm} \label{thm_hydro_bounds}
Let us consider a solution of (\ref{micro_evolution}) in dimension $N = KM$, and a collection $a_1,..,a_N$ of real numbers in $[-L,L]$. Let $\eta$ be a solution of (\ref{macro_evol}) with initial condition $\eta_0$, and
\begin{equation} 
\Theta(t) := \frac{1}{2N} \int{\langle A^{-1}(x - NP^t\eta(t)), (x - NP^t\eta(t)) \rangle f(t,x)\mu_{N,a,m}(dx)},
\end{equation}

Assume that 

(i) $\Ent_{\mu_{N,a,m}}(f_0) \leq C_1N$;

(ii) $\bar{H}(\eta_0) \leq C_2$ for some $C_2 > 0$;

(iii) $\inf \bar{H} \geq -\beta$ for some $\beta > 0$;

(iv) $\int{|x|^2\mu_{N,a,m}(dx)} \leq \alpha N$ for some $\alpha > 0$;

(v) The coarse-grained Hamiltonian $\bar{H}_{N,K,a}$ is $\lambda-$uniformly convex.

Then we have, for any $T > 0$,
\begin{align} \label{estimate_micro_macro}
\max &\left\{\underset{0 \leq t \leq T}{\sup} \hspace{1mm} \Theta(t), \frac{\lambda}{2}\int_0^T{|y - \eta(t)|_Y^2\bar{f}(t,y)\bar{\mu}(dy)} \right\} \notag \\
&\leq \Theta(0) + T\frac{M}{N} + \frac{\gamma C_0 C_1 K}{2\lambda M^2} + \frac{\gamma^{1/2}}{M}\left( 2\alpha + \frac{2C_1}{\rho} \right)^{1/2}(C_1 + C_2 + \beta)^{1/2}, 
\end{align}

where $C_0$ is the constant in the covariance estimate of Proposition \ref{p_crucial_covariance_two_scale}.
\end{thm}
The statement of Theorem~\ref{thm_hydro_bounds} can be understood as analog of [GOVW, Theorem 8]. Comparing the estimate~\eqref{estimate_micro_macro} to [GOVW, Theorem 8] one sees that instead of the term with scaling~$\frac{1}{M^2}$ one gets a term with scaling~$\frac{K}{M^2}$, which is a relict of the fact that our covariance estimate has the factor~$K$ in contrast to the covariance estimate used in[GOVW].

Our aim will then be to prove a hydrodynamic limit theorem by applying Theorem~\ref{thm_hydro_bounds}. 
To do this, we will need to embed the spaces $X_{N}$ into a single functional space. 
To simplify notations, we arbitrarily take  the mean $m$ to be zero. Everything works fine for any $m \in \R$.
With this aim in mind, we identify $X_{N,0}$ with the space $\bar{X}_N$ of piecewise constant functions on $\mathbb{T} = \R/\mathbb{Z}$

$$\bar{X}_N := \left\{\bar{x} : \mathbb{T} \longrightarrow \R; \hspace{2mm} \bar{x} \text{ is constant on } \left( \frac{j-1}{N}; \frac{j}{N} \right], \hspace{2mm} j = 1,..,N \right\}.$$

From now on, we will call "the step function associated to $x \in X_N$" the step function $\bar{x} \in \bar{X}_N$ with
\begin{equation}
\bar{x}(\theta) = x_j, \hspace{1cm} \theta \in \left( \frac{j-1}{N}; \frac{j}{N} \right].
\end{equation}

Similarly, the "vector associated to $\bar{x} \in \bar{X}_N$  will denote the vector $x \in X_N$ with

$$x_j = \bar{x}(j/N).$$

It turns out the $L^2$ norm is not well-adapted to our problem, since it is too sensitive to local fluctuations. 
Instead, we endow the spaces $\bar{X}_N$ with the $H^{-1}$ norm, which we define as follows : if $f : \mathbb{T} \rightarrow \R$ is a locally integrable function with zero mean, define

\begin{equation}
||f||_{H^{-1}}^2 := \int_{\T}{w(\theta)^2d\theta}, \hspace{5mm} w' = f, \hspace{3mm} \int_{\T}{w(\theta)d\theta} = 0.
\end{equation}

As an application of Theorem \ref{thm_hydro_bounds}, we consider a sequence of bounded iid random variables $(a_q)_{q \in \mathbb{Q}}$, 
and we define $a_{i,N} = a_{i/N}$ for $i \in \{1,..,N\}$. 
Given a realization of this sequence, we define $\mu_{N,a,0}(dx) = \exp(-H_N(x))dx$ a probability measure on $X_{N,m}$, where 
$$H_N(x) := \underset{i = 1}{\stackrel{N}{\sum}} \hspace{1mm} \psi(x_i) + a_{i,N}x_i.$$
We also consider a sequence $(F_N)$ of probability densities on $X_{N,0}$ 
such that for every $N$, $F_N$ is absolutely continuous with respect to the Lebesgue measure on $X_{N,0}$. 
We denote by $f_{0,N,a}$ the density of $F_N$ with respect to $\mu_{N,a,0}$.

We denote by $\mu_N^0$ the generalized canonical ensemble on $X_N$ where the linear part is equal to $0$, that is 

\begin{equation} \label{ref_measure_no_field}
\mu_N^0(dx) = \frac{1}{Z}1_{(1/N)\sum x_i = m}\exp \left(- \sum \psi(x_i) \right).
\end{equation}

We make the following extra assumption of polynomial growth on $\psi_c$, which we will need to obtain some a priori uniform $L^p$ estimates on $\eta$.

\begin{equation} \label{poly_growth}
c_1(1 + |x|^{p-2}) \leq \psi_c(x) \leq c_2(1 + |x|^{p-2}), \hspace{5mm} p \geq 2.
\end{equation}

This assumption is probably not the most general possible, but covers the physically relevant situations. It should be possible to extend it to an assumption of subexponential growth, but this would require manipulating some Orlicz norm associated to $\psi_c$ rather than the $L^p$ norm we shall use in Section 5 in the energy estimates. \medskip

\begin{thm} \label{thm_lim_hydro}
Assume $\psi$ satisfies (\ref{assumption_potential}) and (\ref{poly_growth}). Given a realization of the random variables $(a_q)$, let $f = f(t,x)$ be a time-dependent probability density on $(X_N, \mu_{N,a})$, solving
$$\frac{\partial}{\partial t}(f \hspace{1mm} \mu_{N,a}) = \nabla \cdot (A \nabla f \mu_{N,a})$$
where $f(0,\cdot) = f_{0,a}$  is the density of a probability measure $F_N$ with respect to $\mu_{N,a}$. We assume $F_N$ to be absolutely continuous with respect to the Lebesgue measure, and to satisfy the bound
\begin{equation}\label{e_closness_initial_data}
\int_{X_N}{\left( \frac{dF_N}{d\mu_N^0} \right) \log \left( \frac{dF_N}{d\mu_N^0} \right) \mu_N^0} \leq CN 
\end{equation}
for some constant $C > 0$, where $\mu^0_N$ was defined in \eqref{ref_measure_no_field}. Assume that 
\begin{equation}
\underset{N \uparrow \infty}{\lim} \hspace{1mm} \int_{X_N}{||\bar{x} - \zeta_0||_{H^{-1}}^2 F_N(dx)} = 0
\end{equation}
for some $\zeta_0 \in L^p(\mathbb{T})$ (initial macroscopic profile) with $\int{\zeta_0 d\theta} = 0$. Then for any $T > 0$ we have, for almost every realization of the field of random variables $(a_q)$,
\begin{equation}
\underset{N \uparrow \infty}{\lim} \hspace{1mm} \underset{0 \leq t \leq T}{\sup} \hspace{1mm} \int_{X_N}{||\bar{x} - \zeta(t, \cdot)||_{H^{-1}}^2 f(t,x)\mu_{N,a}(dx)} = 0
\end{equation}
where $\zeta$ is the unique weak solution of the nonlinear parabolic equation
\begin{equation} \label{hydro_equn}
\frac{\partial \zeta}{\partial t} = \frac{\partial^2}{\partial \theta^2}\tilde{\varphi}'(\zeta)
\end{equation}
with initial condition $\zeta(0,\cdot) = \zeta_0$, and where $\tilde{\varphi}$ is defined by
\begin{equation} \label{average_hydro_potential}
\tilde{\varphi}(m) = \underset{\sigma \in \R}{\sup} \hspace{1mm} \left\{ \sigma m - \mathbb{E}^a\left[\log \int_{\R}{\exp \left((\sigma - a)x - \psi(x) \right)dx}\right] \right\}.
\end{equation}
\end{thm}

In this theorem, the notion of weak solution we have used is the following :

\begin{defn}
We will call $\zeta(t,\theta)$ a weak solution of (\ref{hydro_equn}) on $[0,T] \times \T$ if
\begin{equation}
\zeta \in L^{\infty}_t(L^p_{\theta}), \hspace{5mm} \frac{\partial \zeta}{\partial t} \in L^2_t(H^{-1}_{\theta}), \hspace{5mm} \tilde{\varphi}'(\zeta) \in L^{\infty}_t(L^q_{\theta});
\end{equation}
and
\begin{equation}
\left \langle \xi, \frac{\partial \zeta}{\partial t} \right\rangle_{H^{-1}} = - \int_{\T}{\xi \tilde{\varphi}'(\zeta)d\theta} \hspace{5mm} \text{for all } \xi \in L^2, \hspace{3mm} \text{for almost every } t \in [0,T].
\end{equation}
\end{defn}

\begin{rmq}
Our method also works (and the proof is much easier) in the case where the $a_i$ are of the form $a_i = a(i/N)$ for some continuous function $a : \mathbb{T} \rightarrow \R$. In such a case, the hydrodynamic equation is
$$\frac{\partial \zeta}{\partial t} = \frac{\partial^2}{\partial \theta^2}(\varphi'(\zeta) - a).$$
\end{rmq}

\section{Proof of Proposition \ref{p_crucial_covariance_two_scale}}

Our proof of the covariance estimate relies on the following result, which already was an important component of the proof of the LSI in [MO].

\begin{prop} [One-dimensional asymmetric Brascamp-Lieb inequality]
Let  $\nu(dx) = Z^{-1}\exp(-\psi(x))dx$ a probability measure on $\R$, where $\psi = \psi_c + \delta\psi$ is a bounded perturbation of a strictly convex potential, as in (\ref{assumption_potential}). Then for any functions $f$ and $g$, we have
$$|\cov_{\nu}(f,g)| \leq \exp(-3\osc \delta \psi) \hspace{1mm} \underset{x}{\sup} \hspace{1mm} \left| \frac{g'(x)}{\psi_c''(x)} \right| \int{|f'|d\nu}.$$
\end{prop}

We refer to [MO] for a proof of this result. This covariance estimate was generalized in [CCL] to cover other situations, but here this estimate will suffice. Using this inequality, we shall prove the following two lemmas, which will imply Proposition~\ref{p_crucial_covariance_two_scale}.

\begin{lem} \label{estimate_cov_convex}
Let $\mu_{N,a,m}(dx) = Z^{-1}\exp \left( - \sum \psi(x_i) + a_ix_i \right)dx$ be a probability measure on $X_{N,m}$, where the single-site potential $\psi$ is strictly convex and the $a_i$ are some real numbers. Then we have, for any non-negative function $f$ such that $\int{f d\mu_{N,a,m}} = 1$, the following covariance estimate:
\begin{equation}
\left| \cov_{\mu_{N,a,m}} \left( f, \frac{1}{N}\underset{i =1}{\stackrel{N}{\sum}} \hspace{1mm} \psi'(x_i) \right) \right|^2 \lesssim  K\int{\frac{\sum \left| \frac{d}{dx_i} f \right|^2}{f}\mu_{N,a,m}}
\end{equation}
\end{lem}

\begin{lem}\label{p_cov_perturbed}
  Assume that the single-site potential $\psi$ is perturbed strictly convex in the sense of~\eqref{e_perturbation_convex}. Let $\tilde K \in \mathbb{N}$ be an arbitrary integer. Then the canonical ensemble $\mu_{\tilde K,m}$ satisfies for any nonnegative function $f \geq0$ with $\int f d \mu_{\tilde K,a,m} =1$ the following covariance estimate: 
\begin{equation*}
 \left| \cov_{\mu_{\tilde K,a,m}} \left(f, \frac{1}{\tilde K} \sum_{i=1}^{\tilde K }  \psi'(x_i) \right)  \right|^2 \lesssim C(\tilde K)  \int \frac{\sum_{i=1}^{\tilde K } \left|\frac{d}{dx_i} f \right|^2}{f}   \ d \mu_{\tilde K,a,m} .
\end{equation*}
\end{lem}

Using these two lemmas and a two-scale decomposition of the variance, we can prove Proposition \ref{p_crucial_covariance_two_scale} 

\begin{proof}[Proof of Proposition \ref{p_crucial_covariance_two_scale}]
We decompose the system of $K$ spins into $\tilde M$ blocks, each containing $\tilde K$ many spins i.e. $K = \tilde K \tilde M$. We disintegrate $\mu_{K,a,m}$ according to the coarse-graining operator 
\begin{equation*}
P x = \left( \frac{1}{\tilde K} \sum_{i \in B(1)} x_i \; ,\;  \ldots \; , \; \frac{1}{ \tilde K} \sum_{i \in B(\tilde M)} x_ i \right).
\end{equation*}
Here, the index set $B(l)$ of the $l$-th block is given by $B(l)= \left\{ (l-1) \tilde K +1 , l \tilde K  \  \right\}$,$l \in \left\{ 1, \ldots , \tilde M\right\}$.
Then the canonical ensemble $\mu_{K,m}$ can be written as
\begin{equation*}
  \mu_{K,m} = \mu(dx| Px=y) \ \bar \mu (dy),
\end{equation*}
where $\mu(dx| Px=y)$ are the conditional measures and $\bar \mu (dy)$ is the marginal. This disintegration implies the following decomposition of the covariance
\begin{multline}
  \cov_{\mu_{K,m} } \left(f, \frac{1}{K} \sum_{i=1}^K  \psi'(x_i) \right) =  \int \cov_{ \mu (dx| Px=y) } \left(f, \frac{1}{K} \sum_{i=1}^K  \psi'(x_i) \right) \bar \mu (dy) \; \\
+ \; \cov_{\bar \mu} \left(\int f \mu (dx| Px=y), \frac{1}{K} \int \sum_{i=1}^K  \psi'(x_i)  \mu (dx| Px=y)\right),
\end{multline}\smallskip
and therefore 

\begin{align}  \label{e_decomp_cov_micro_macro}
\left| \right.  \cov_{\mu_{K,m} }& \left(f, \frac{1}{K} \sum_{i=1}^K  \psi'(x_i) \right) \left. \right|^2 \notag \\
&\leq 2\left( \int \cov_{ \mu (dx| Px=y) } \left(f, \frac{1}{K} \sum_{i=1}^K  \psi'(x_i) \right) \bar \mu (dy) \right)^2 \notag \\
& \hspace{5mm} + \left| \cov_{\bar \mu} \left(\int f \mu (dx| Px=y), \frac{1}{K} \int \sum_{i=1}^K  \psi'(x_i)  \mu (dx| Px=y)\right) \right|^2 \notag \\
&\leq 2 \int \left(\cov_{ \mu (dx| Px=y) } \left(f, \frac{1}{K} \sum_{i=1}^K  \psi'(x_i) \right)\right)^2 \bar \mu (dy)  \notag \\
& \hspace{5mm} + 2\left| \cov_{\bar \mu} \left(\int f \mu (dx| Px=y), \frac{1}{K} \int \sum_{i=1}^K  \psi'(x_i)  \mu (dx| Px=y)\right) \right|^2.
\end{align}

Let us now estimate the first term on the right hand side of~\eqref{e_decomp_cov_micro_macro}: \newline
Note that the conditional measures $\mu (dx | Px=y)$ have a product structure i.e.~one can write
\begin{equation*}
  \mu (dx | Px=y) = \bigotimes_{j=1}^{\tilde M} \mu_{\tilde K, y_j} (dx^j),
\end{equation*}
where the probability measures $\mu_{\tilde K, y_j}$ are defined as in \ref{def_grand_can}. Additionally, we used the notation $x^j = (x_i)_{i \in B(j)}$. The latter yields the following representation of the microscopic covariance
\begin{multline*}
  \cov_{\bar \mu (dx| Px=y) } \left(f, \frac{1}{K} \sum_{i=1}^K  \psi'(x_i) \right) \\
  = \frac{1}{\tilde M} \sum_{j=1}^{\tilde M}   \cov_{\mu_{\tilde K , y_j} } \left( \int f \otimes_{i \neq j} \mu_{\tilde K , y_i} (dx_i), \frac{1}{\tilde K} \sum_{i=1}^{\tilde K}  \psi'(x_i) \right).
 \end{multline*}

On the right hand side we apply Lemma~\ref{p_cov_perturbed}. This leads to the estimate
\begin{align}
    \cov_{\bar \mu (dx| Px=y) }& \left(f, \frac{1}{K} \sum_{i=1}^K  \psi'(x_i) \right)^2 \notag \\
&  \leq C(\tilde K)  \ \frac{1}{\tilde M} \sum_{j=1}^{\tilde M} \int \frac{\sum_{i \in B(j)} \left| \frac{d}{dx_i} f \right |^2}{f} \ \mu(dx| Px=y) \notag \\
& =  C(\tilde K) \ \frac{1}{\tilde M} \int \frac{\sum_{i=1}^K   \left| \frac{d}{dx_i} f \right |^2}{f} \ \mu(dx| Px=y).
\label{e_estimate_micro_cov}
\end{align} \smallskip

Let us now estimate the second term on the right hand side of~\eqref{e_decomp_cov_micro_macro}: \newline 
We take a closer look at the structure of the marginal $\bar \mu$. It turns out that the Hamiltonian $\bar H$ of 
\begin{equation*}
  \bar \mu (dy) = \frac{1}{Z }  \ \exp \left( - \bar H (y) \right) \mathcal{H}_{ \lfloor \left\{ \frac{1}{\tilde M} \sum_{j=1}^{\tilde M} y_j = m  \right\}      \ }^{\tilde M-1} (dx)
\end{equation*}
is given by a sum of single-site potentials $\psi_{\tilde K,j}$ which depend on the site $j$ because of the inhomogeneous linear part.
\begin{equation*}
  \bar H (y) = \sum_{j=1}^{\tilde M} \tilde K \psi_{\tilde K,j} (y_j),
\end{equation*}
where the single-site potential $\psi_{\tilde K,j}$ is given by
\begin{equation*}
  \psi_{\tilde K} (y_j) = - \frac{1}{\tilde K} \log \int \exp \left( - \sum_{i=1}^{\tilde K} \psi (x_i) + a_{(j-1)\tilde K + i}  \right) \mathcal H (dx)_{ \lfloor \left\{ \frac{1}{\tilde K} \sum_{i=1}^{\tilde K} x_i = y_j  \right\}  }^{\tilde K -1} .
\end{equation*}
Without loss of generality we may assume that $\tilde K \gg 1$. Hence by the generalized local Cram\'er theorem (see Appendix A), the function $\psi_{\tilde K}$ is uniformly strictly convex. A direct calculation yields that the first derivative of $\psi_{\tilde K,j}$ is given by
\begin{equation}\label{e_deriv_psi_K}
  \psi_{\tilde K}' (y_j) = \frac{1}{\tilde K} \int \sum_{i=1}^{\tilde K} \psi' (x_i) + a_{(j-1)\tilde K + i} \mu_{\tilde K , y_j} (dx).
\end{equation}
Let us now estimate the macroscopic covariance term. Since the constants $a_i$ disappear in the covariance, rewriting the macroscopic covariance using~\eqref{e_deriv_psi_K} yields 
\begin{multline*}
   \cov_{\bar \mu} \left(\int f \mu (dx| Px=y), \frac{1}{K} \int \sum_{i=1}^K  \psi'(x_i)  \mu (dx| Px=y)\right)  \\
  =   \frac{1}{\tilde K \tilde M} \sum_{j=1}^{ \tilde M} \cov_{\bar \mu} \left(\int f \mu (dx| Px=y) , \tilde K \psi_{\tilde K}' (y_j)  \right) \\
  =   \cov_{\bar \mu} \left(\int f \mu (dx| Px=y), \frac{1}{\tilde K \tilde M}\sum_{j=1}^{ \tilde M} \tilde K \psi_{\tilde K}' (y_j)  \right).
\end{multline*}
Therefore an application of Lemma~\ref{estimate_cov_convex} yields the estimate
\begin{multline}\label{e_macro_cov_step_1}
  \left|  \cov_{\bar \mu} \left(\int f \mu (dx| Px=y), \frac{1}{K} \int \sum_{i=1}^K  \psi'(x_i)  \mu (dx| Px=y)\right)  \right|^2 \\
 \leq C \tilde{M}  \int \frac{\sum_{j=1}^{\tilde M} \left| \frac{d}{dy_j} \int f \mu (dx | Px=y) \right|^2}{\int{f(x)\mu(dx|y)}} \bar \mu (dy).
\end{multline}
Because 
\begin{multline*}
  \frac{d}{dy_j} \int f \mu (dx | Px =y) \\ = \int \sum_{i \in B(j)} \frac{d}{d x_i  } f \mu (dx | Px= y) - \cov_{\mu (dx | Px=y)} \left(f, \sum_{i \in B(j)} \psi' (x_i)  \right) ,
\end{multline*}
we get the estimate
\begin{align*}
 & \left| \frac{d}{dy_j} \int f \mu (dx | Px =y) \right|\\
 & \qquad \lesssim \tilde K \int \sum_{i \in B(j) } \left| \frac{d}{d x_i} f  \right| \mu (dx | Px =y) \\
& \qquad + \left|  \cov_{\mu (dx | Px=y)} \left(f, \sum_{i \in B(j)} \psi' (x_i)  \right) \right|.
\end{align*}
The covariance terms can be estimated by an application of Lemma~\ref{p_cov_perturbed} yielding the overall inequality
\begin{multline}
  \left| \frac{d}{dy_j} \int f \mu (dx | Px =y)\right|^2 \\ 
\lesssim  C (\tilde K ) \left(\int{f(x)\mu(dx|y)}\right)\int \frac{\sum_{i \in B(j) } \left| \frac{d}{d x_i} f \right|^2}{f} \mu (dx | Px =y) . \label{e_macro_mean_difference}
\end{multline}
Now combining the estimates~\eqref{e_macro_cov_step_1} and~\eqref{e_macro_mean_difference} implies
\begin{align}
&  \left|  \cov_{\bar \mu} \left(\int f \mu (dx| Px=y), \frac{1}{K} \int \sum_{i=1}^K  \psi_c'(x_i)  \mu (dx| Px=y)\right)  \right|^2 \notag \\
& \qquad \lesssim C(\tilde K) \ \tilde M \int \frac{\sum_{i =1}^{K} \left| \frac{d}{d x_i} f  \right|^2}{f} \mu_{K,m} (dx).
  \label{e_estimate_macro_cov}
\end{align}\smallskip
Now, it is only left to insert the estimate~\eqref{e_estimate_micro_cov} for the conditional covariance and the estimate~\eqref{e_estimate_macro_cov} for the macroscopic covariance in the decomposition~\eqref{e_decomp_cov_micro_macro}. Using $K = \tilde K \tilde M$ this yields
\begin{equation}
\left|    \cov_{\mu_{K,m} } \left(f, \frac{1}{K} \sum_{i=1}^K  \psi'(x_i) \right) \right|^2 \lesssim C(\tilde K ) K  \int \frac{\sum_{i =1}^{K} \left| \frac{d}{d x_i} f  \right|^2}{f} \mu_{K,m} (dx) 
\end{equation}

\end{proof}

\begin{proof} [Proof of Lemma \ref{estimate_cov_convex}]

We will apply a recursive approach. In the first step, we will estimate the term 
\begin{equation} \label{e_cov_recur_start}
  \cov_{\mu_{K,m}} \left( f, \frac{1}{K} \psi' (x_n) \right)
\end{equation}
for a fixed site $n \in \left\{ 1, \ldots , K\right\}$. Let us choose another arbitrary site $\ell \in \left\{1, \ldots K \right\}$ and disintegrate the canonical ensemble $\mu_{K,m}$ by fixing the spin $x_{\ell}$ i.e.
\begin{equation}
  \label{e_dis_mu_x_l}
  \mu_{K,a,m} (dx) = \mu (d \bar x_{\ell} | x_{\ell}) \ \bar \mu (dx_{\ell}),
\end{equation}
where we used the notation
\begin{equation*}
  \bar x_l = (x_1, \ldots, x_{{\ell}-1}, x_{{\ell}+1}, \ldots, x_K).  
\end{equation*}
Then, the marginal $\bar \mu (dx_{\ell})$ is given by 
\begin{equation*}
  \bar \mu (d x_l) =  \frac{1}{Z} \ \exp \left( - \psi (x_{\ell}) - a_{\ell} - (K-1) \psi_{K-1, \bar{a}_{\ell}} \left( \frac{K}{K-1} m - \frac{1}{K-1} x_{\ell} \right) \right),
\end{equation*}
where the function $\psi_{K-1, \bar{a}_{\ell}}$ is given by
\begin{equation*}
  \psi_{K-1} (z) = - \frac{1}{K-1} \log  \int \exp \left( -\sum_{i=1}^{K-1} \psi (x_i) - \langle \bar{a}_{\ell}, x \rangle \right) \mathcal{H}_{\lfloor \left\{\frac{1}{K-1} \sum_{i=1}^{K-1} x_i = z \right\}}^{K-2} (dx).
\end{equation*}
Note that the single-site potential $\psi_{K-1, \bar{a}_{\ell}} $ inherits the uniform strict convexity from the single-site potential $\psi$ by a standard argument using the (symmetric) Brascamp-Lieb inequality. The disintegration~\eqref{e_dis_mu_x_l} of $\mu_{K,a,m}$ yields the following representation of the covariance
\begin{equation*}
  \cov_{\mu_{K,m}} \left( f, \frac{1}{K} \psi' (x_n) \right)= \cov_{\bar \mu(dx_{\ell})} \left( \int f \mu (d \bar x_{\ell}| x_{\ell}), \frac{1}{K} \psi' (x_n) \right).
\end{equation*}
Because $\bar \mu (d x_{\ell})$ is an one-dimensional measure with strictly convex Hamiltonian, an application of the one-dimensional asymmetric Brascamp-Lieb inequality yields
\begin{multline}
  \cov_{\mu_{K,m}} \left( f, \frac{1}{K} \psi' (x_n) \right) \\
\leq \frac{1}{K} \ \underbrace{\left| \frac{\psi''(x_{\ell})}{\psi''(x_l) + \frac{1}{K-1} \psi_K'' \left( \frac{K}{K-1} m - \frac{1}{K-1} x_{\ell} \right)} \right|}_{\leq 1} \ \int \left| \frac{d}{dx_{\ell}} \int f \mu (d \bar x_{\ell} | x_{\ell}) \right| \bar \mu (dx_{\ell})  . \label{e_iter_ass_BL_1}
\end{multline}
Direct calculation reveals that
\begin{align*}
&  \frac{d}{dx_l}  \int f \mu (d \bar x_{\ell} | x_{\ell}) \\
 & =  \int \left( \frac{d}{d x_{\ell}} f - \frac{1}{K-1} \sum_{i=1}^{K-1} \frac{d}{dx_i} f \right) \mu (d \bar x_{\ell} | x_{\ell}) - \frac{1}{K-1} \cov_{\mu (d \bar x_{\ell} | x_{\ell})} \left( f, \sum_{i \neq l, \ i =1}^{K} \psi' (x_i) \right) \\
 & =  \frac{1}{K-1}   \int \left( \sum_{i=1}^K \frac{d}{dx_{\ell}}  f -  \frac{d}{d x_i} f \right) \mu (d \bar x_{\ell} | x_{\ell})  -  \frac{1}{K-1} \cov_{\mu (d \bar x_{\ell} | x_{\ell})} \left( f, \sum_{i \neq l, \ i= 1 }^{K} \psi' (x_i) \right) \\
\end{align*}
Using the last identity one can directly deduce from~\eqref{e_iter_ass_BL_1} that
\begin{multline}
    \cov_{\mu_{K,m}} \left( f, \frac{1}{K} \psi' (x_n) \right) \leq \frac{1}{K - 1}  \int \sum_{i=1}^K \left| \frac{d}{dx_{\ell}}  f -  \frac{d}{d x_i} f  \right| \mu_{K,m} (dx) \\
 + \int \left| \frac{1}{K-1} \cov_{\mu (d \bar x_{\ell} | x_{\ell})} \left( f, \sum_{i \neq l, \ i=1  }^{K} \psi' (x_i) \right) \right| \bar \mu (dx_{\ell}). \label{e_iter_ass_BL_2}
\end{multline}
One sees that on the right hand side, one has to estimate a covariance term 
\begin{equation*}
  \frac{1}{K-1} \cov_{\mu (d \bar x_{\ell} | x_{\ell})} \left( f, \sum_{i \neq l, \ i=1  }^{K} \psi' (x_i) \right)
\end{equation*}
that has the same structure as the covariance term we started with (cf. \eqref{e_cov_recur_start}). The reason is that the conditional measure $\mu(d \bar x_{\ell} | x_{\ell}) $ has the structure of a canonical ensemble i.e.
\begin{equation*}
  \mu(d \bar x_{\ell} | x_{\ell}) = \mu _{K-1, \frac{K}{K-1}m - \frac{1}{K-1} x_{\ell}} (d \bar x_{\ell}).
\end{equation*}
Therefore one can apply the estimate~\eqref{e_iter_ass_BL_2} recursively, until there is no covariance term left. On the right hand side of the covariance estimate, there only occur terms of the form
\begin{equation*}
 \int  \left|  \frac{d}{dx_i} f - \frac{d}{dx_j} f \right| \mu_{K,m}  
\end{equation*}
for some indexes $1 \leq i < j \leq K$. Therefore, one only has to determine the prefactors in front of these terms. By symmetry of the system this prefactor is independent of the particular choice $i$ and $j$ and can be determined by combinatorics. Indeed, we will need the following observation that follows from the recursive formula~\eqref{e_iter_ass_BL_2}: \smallskip

The term  $\int  \left|  \frac{d}{dx_i} f - \frac{d}{dx_j} f \right| \mu_{K,m}$ only occurs, if one conditions on the spin value $x_i$ and the spin value $x_j$ is free or if one conditions on the spin value $x_j$ and the spin value $x_i$ is free. \smallskip

Let us start to determine the prefactors. In the first step, there are $K$ free spins. Let us condition on an arbitrary spin $x_l$, $l \in \left\{ 1, \ldots, K \right\}$. Then the probability that one conditions on the spin $x_i$ or $x_j$ is given by $\frac{2}{K}$. In this case, one gets by~\eqref{e_iter_ass_BL_2} the prefactor 
\begin{equation*}
  \frac{2}{K} \frac{1}{K-1}.
\end{equation*}
Note that in the remaining conditional steps, the term $\int  \left|  \frac{d}{dx_i} f - \frac{d}{dx_j} f \right| \mu_{K,m}$ does not occur anymore. \newline
So, let us consider the remaining cases. We assume that in the first step one conditioned on a spin $x_l$ with $i \neq l \neq j$ and in the second step one conditions on the spin $x_i$ or $x_j$. It follows from~\eqref{e_iter_ass_BL_2} that the resulting prefactor is given by
\begin{equation*}
  \frac{K-2}{K} \frac{2}{K-1} \frac{1}{K-2}.
\end{equation*}
 With this argument it becomes clear that the overall prefactor $c$ is given by
 \begin{align*}
   c &=   \frac{2}{K} \frac{1}{K-1} +  \frac{K-2}{K} \frac{2}{K-1} \frac{1}{K-2} +   \frac{K-2}{K} \frac{K-3}{K-1} \frac{2}{K-2} \frac{1}{K-3} + \ldots\\
   & = \frac{2}{K} \frac{1}{K-1} \sum_{i=1}^{K-1} 1 = \frac{2}{K}. 
 \end{align*}
Overall, this deduces the covariance estimate
\begin{align}
  \cov_{\mu_{K,m}} \left( f, \frac{1}{K} \psi' (x_i) \right) & \lesssim \frac{1}{K} \int \sum_{i, j =1}^K \left|\frac{d}{dx_i} f - \frac{d}{dx_j} \right| \mu_{K,m} \label{e_cruc_cov_pure} 
\end{align}
and therefore
\begin{align}
\left| \right. \cov_{\mu_{K,m}} &\left( f, \frac{1}{K} \underset{i = 1}{\stackrel{K}{\sum}} \psi' (x_i) \right) \left. \right|^2 \lesssim \left( \int \sum_{i, j =1}^K \left|\frac{d}{dx_i} f - \frac{d}{dx_j} \right| \mu_{K,m} \right)^2 \notag \\
&\lesssim \left( \int \frac{\sum_{i, j =1}^K \left|\frac{d}{dx_i} f - \frac{d}{dx_j} \right|^2}{f} \mu_{K,m} \right) \left( \int{f  \mu_{K,m}} \right) \notag \\
&\lesssim \int \frac{\sum_{i, j =1}^K 2\left|\frac{d}{dx_i} f \right|^2  + 2 \left| \frac{d}{dx_j} \right|^2}{f} \mu_{K,m} \notag \\
&\lesssim K \int \frac{\sum_{i = 1}^K \left|\frac{d}{dx_i} f \right|^2}{f} \mu_{K,m}
\end{align}
\end{proof}

\begin{proof}[Proof of Lemma~\ref{p_cov_perturbed}]
  By the splitting $\psi= \psi_c + \delta \psi$, we can decompose the covariance term according to
  \begin{align} \label{eq4005}
    \cov_{\mu_{\tilde K,m} } &\left(f, \frac{1}{\tilde K} \sum_{i=1}^{\tilde K}  \psi'(x_i) \right) \notag \\ 
   & \hspace{2cm} = \cov_{\mu_{\tilde K,m} } \left(f, \frac{1}{\tilde K} \sum_{i=1}^{\tilde K}  \psi_c'(x_i) \right)  + \cov_{\mu_{\tilde K,m} } \left(f, \frac{1}{\tilde K} \sum_{i=1}^{\tilde K}  \delta \psi'(x_i) \right) .
  \end{align}
Because $| \delta \psi' | \lesssim 1$ by assumption, we can estimate the second term on the right hand side of the last equality in the same way as in [MO] by an application of the classical Csisz\'ar-Kullback-Pinsker inequality
$$||f\mu - \mu||_{TV} \leq \sqrt{2\Ent_{\mu}(f)}$$
and the LSI for $\mu_{\tilde K,m}$, which holds a priori with a bad constant that depends on $\tilde{K}$  (see Lemma \ref{lem_lsi_micro} in the next section). Indeed, 
\begin{align} \label{eq3999}
\cov_{\mu_{\tilde K,m} } &\left(f, \frac{1}{\tilde K} \sum_{i=1}^{\tilde K}  \delta \psi'(x_i) \right) \leq ||\delta \psi'||_{\infty} \int{|f(x) - 1|\mu_{\tilde K,m}(dx)} \notag \\
&= ||f\mu_{\tilde K,m} - \mu_{\tilde K,m}||_{TV} \notag \\
&\leq \sqrt{2\Ent_{\mu_{\tilde K,m}}(f)} \notag \\
&\leq C(\tilde{K}) \left(\int{\frac{ \sum \left| \frac{df}{dx_i} \right|^2}{f}\mu_{\tilde K,m}} \right)^{1/2}.
\end{align}
A proof of the Csisz\'ar-Kullback-Pinsker inequality can be found for example in [L]. Now, let us have a look at the first term on the right hand side of (\ref{eq4005}). 
This term can be estimated with the same strategy as for Lemma~\ref{estimate_cov_convex}. 
The only difference is that instead of the asymmetric Brascamp-Lieb inequality for strictly convex single-site potentials one has to apply the version for perturbed strictly convex single-site potentials. 
However, this does not matter, because we do not care about constants depending on $\tilde K$. 
There is still a little bit of work, because $\psi_K$ is not convex for small $K$. 
This is not a problem because $\psi_K$ can be represented as a sum of a strictly convex potential and a perturbation (see [MO]). 
What we get is
\begin{align} \label{eq4000}
\cov_{\mu_{\tilde K,m} } &\left(f, \frac{1}{\tilde K} \sum_{i=1}^{\tilde K}  \psi_c'(x_i) \right) \leq C(\tilde{K}) \int \sum_{i =1}^{\tilde{K}} \left|\frac{d}{dx_i} f \right| \mu_{\tilde{K},m}, \notag \\
&\leq C(\tilde{K}) \left(\int{ \frac{\sum_{i =1}^{\tilde{K}} \left|\frac{d}{dx_i} f \right|^2}{f}\mu_{\tilde{K},m}} \right)^{1/2} \left(\int{f \mu_{\tilde{K},m}} \right)^{1/2}.
\end{align}

Inserting (\ref{eq3999}) and (\ref{eq4000}) into (\ref{eq4005}), we get 
\begin{align}
\cov_{\mu_{\tilde K,m} } &\left(f, \frac{1}{\tilde K} \sum_{i=1}^{\tilde K}  \psi'(x_i) \right)^2  \notag \\
&\leq  2\cov_{\mu_{\tilde K,m} } \left(f, \frac{1}{\tilde K} \sum_{i=1}^{\tilde K}  \psi_c'(x_i) \right)^2  + 2\cov_{\mu_{\tilde K,m} } \left(f, \frac{1}{\tilde K} \sum_{i=1}^{\tilde K}  \delta \psi'(x_i) \right)^2 \notag \\
&\leq C(\tilde{K}) \left(\int{ \frac{\sum_{i =1}^{\tilde{K}} \left|\frac{d}{dx_i} f \right|^2}{f}\mu_{\tilde{K},m}} \right)^{1/2} \left(\int{f \mu_{\tilde{K},m}} \right).
\end{align}
\end{proof}

\section{Proof of Theorem \ref{lsi}}

As announced, to prove this result we will use the two-scale approach that was introduced in [GOVW] and the covariance estimate of Proposition \ref{p_crucial_covariance_two_scale}. 
First, we will prove that the logarithmic Sobolev inequalities are satisfied at macroscopic scale and at the scale of fluctuations around a fixed macroscopic profile, 
before showing how we can combine the two LSIs to obtain a full LSI at microscopic scale.

\begin{lem} \label{lem_lsi_micro}
The measure $\mu$ satisfies LSI($\rho e^{-cN}$), with constants $c$ and $\rho$ independant of $N$, $m$ and the sequence $(a_i)$. As a consequence, $\mu(dx|y)$ satisfies LSI($\rho_0$) for some constant $\rho_0$ that depends on $K$, but not on $y$.
\end{lem}

\begin{proof}
The first part is a consequence of a combination of the Bakry-Emery theorem and the Holley-Stroock perturbation Lemma (see Appendix B).

Recall (\ref{factorization_measure_fluct}) : 
$$\mu_{N,a,m}(dx|y) = \underset{i = 1}{\stackrel{M}{\bigotimes}} \hspace{1mm} \mu_{K,(a_j)_{(i-1)K+1 \leq j \leq iK}, y_i},$$
so by the tensorization principle (Criterion I of Appendix B), $\mu(dx|y)$ satisfies LSI($\rho_0$) for some constant $\rho_0$ that depends on $K$, but not on $y$, since $\mu_{K,(a_j)_{(i-1)K+1 \leq j \leq iK}, y_i}$ satisfies LSI($\rho e^{-cK}$).
\end{proof}

Note that, since $\mu_{K,(a_j)_{(i-1)K+1 \leq j \leq iK}, y_i}$ has the same form as $\mu$,  Theorem \ref{lsi} will imply that a posteriori $\rho_0$ does not actually depend on $K$.

\begin{lem} \label{lem_lsi_macro}
There exists a positive constant $\lambda$ and an integer $K_0$ such that for all $K \geq K_0$, the measure $\bar{\mu}$ satisfies LSI($\lambda N$), independently of the mean $m$ and the sequence $(a_i)$.
\end{lem}

\begin{proof}
By the generalized local Cram\`er theorem of Appendix A, for $K$ large enough, $\bar{H}(y) = \frac{1}{M} \sum \psi_{K,i,a}(y_i)$ is $\lambda$-uniformly convex, with constant $\lambda$ independent of $K$, $m$ and the sequence $(a_i)$. An application of the Bakry-Emery Theorem then yields the desired result.
\end{proof}

To get the full logarithmic Sobolev inequality at microscopic scale, we rely on the following decomposition of the entropy :

\begin{equation} \label{dec_ent}
\Ent_{\mu}(f) = \Ent_{\bar{\mu}}(\bar{f}) + \int_Y{\Ent_{\mu(dx|y)}(f)\bar{\mu}(dy)}.
\end{equation}

This identity separates the entropy into the sum of a macroscopic entropy and an entropy at the scale of fluctuations. 
We can then use the fact that we have a LSI at both scales to bound these two terms separately. 
Indeed, by using Lemma \ref{lem_lsi_micro}, we have, for a decomposition with large but fixed $K$,

\begin{equation} \label{aleph1}
\Ent_{\mu(dx|y)}(f) \leq \frac{1}{\rho_0}\int_X{\frac{|(id_X - P^tNP)\nabla f|^2}{f}\mu(dx|y)}.
\end{equation}

In the same way, at macroscopic scale, we have

\begin{equation} \label{aleph2}
\Ent_{\bar{\mu}}(\bar{f}) \leq \frac{1}{\lambda N}\int_Y{\frac{|\nabla \bar{f}|^2}{\bar{f}}\bar{\mu}}.
\end{equation}

To use this second estimate, we need to relate $\nabla \bar{f}$ and $\nabla f$. This is the point of the following Lemma :

\begin{lem} \label{covlem}
For any  $f \in$ Lip$(X)$ and any $y \in Y$, 
\begin{equation}
\int{P\nabla f \mu(dx|y)} = \frac{1}{N} \nabla_Y \bar{f}(y) + P\cov_{\mu(dx|y)}(f, \nabla H).
\end{equation}
\end{lem}

The proof of this Lemma will be deferred to the end of this section. Using this result and the convexity of the function $(x, b) \rightarrow |x|^2/b$ with Jensen's inequality, we get

\begin{align} \label{aleph3}
\frac{|\nabla \bar{f}|^2}{\bar{f}} &= \frac{\left|N\int{P\nabla f \mu(dx|y)} + NP\cov_{\mu(dx|y)}(f, \nabla H) \right|^2}{\int{f(x)\mu(dx|y)}} \notag \\
&\leq 2N\frac{1}{\bar{f}(y)}\left|\int{P^tNP\nabla f \mu(dx|y)}\right|^2 + 2N\frac{1}{\bar{f}(y)} \left| NP^tP \cov(f, \nabla H) \right|^2 \notag \\
&\leq 2N \int{\frac{|P^tNP\nabla f|^2}{f}\mu(dx|y)} + 2N\frac{1}{\bar{f}(y)} \left| NP^tP \cov(f, \nabla H) \right|^2 
\end{align}

We can now use the covariance estimate of Proposition \ref{p_crucial_covariance_two_scale} to bound the second term on the right-hand side of this equation. 
\begin{align} \label{aleph4}
\left| NP^tP \cov(f, \nabla H) \right|^2 &= \underset{i = 1}{\stackrel{M}{\sum}} \hspace{1mm} \left| \cov_{\mu_{K,a,y_i} } \left(\int{f \underset{k \neq i} \bigotimes \mu_{K,a,y_k}}, \frac{1}{K} \sum_{j=(i-1)K+1}^{iK}  \psi'(x_j) \right)  \right|^2 \notag \\
&\leq CK\bar{f}(y)\int{\frac{|(id_X - P^tNP)\nabla f|^2}{f}\mu_{N,a,m}(dx|y)}.
\end{align}

Plugging (\ref{aleph1}), (\ref{aleph2}), (\ref{aleph3}) and (\ref{aleph4}) into (\ref{dec_ent}), we get
\begin{align}
\Ent_{\mu_{N,a,m}}(f) &\leq \frac{2}{\lambda}\int{\frac{|P^tNP\nabla f|^2}{f}\mu_{N,a,m}(dx)} \notag \\
&\hspace{1cm} + \left(\frac{\exp(cK)}{\rho} + \frac{2CK}{\lambda} \right)\int{\frac{|(id_X - P^tNP)\nabla f|^2}{f}\mu_{N,a,m}(dx)} \notag \\
&\lesssim (1 + K)\exp(cK) \int{\frac{|\nabla f|^2}{f}\mu_{N,a,m}(dx)}.
\end{align}
Since as long as $K$ is large enough for the coarse-grained Hamiltonian to be strictly convex, this bound holds, we just have to take $K$ large but fixed for Theorem \ref{lsi} to be proved.

We finish this section with the proof of Lemma \ref{covlem}, which is unchanged from [GOVW] Lemma 21.

\begin{proof}[Proof of Lemma \ref{covlem}]
By definition, we have

\begin{align}
\bar{f}(y) &= \int{f(x)\mu_{N,a,m}(dx|y)} \notag \\
&= \frac{1}{\int_{\{Px = 0\}}{\exp(-H(NP^ty + z))dz}}\int_{\{Px = 0\}}{f(NP^ty + z)\exp(-H(NP^ty + z))dz}. \notag
\end{align}

Thus, for any $\tilde{y} \in Y$, we have
\begin{align}
\nabla_Y \bar{f}(y) \cdot \tilde{y} &= N\int{\nabla f(x) \cdot P^t\tilde{y} \hspace{1mm} \mu(dx|y)} -N \int{f(x) \nabla H(x) \cdot P^t\tilde{y} \mu(dx|y)} \notag \\
&\hspace{3mm} - N\left(\int{f(x)\mu(dx|y)}\right)\left(\int{(-H(x)\cdot P^t\tilde{y})\mu(dx|y)} \right) \notag \\
&= N\left[ \int{P\nabla f(x)\mu(dx|y)} - \int{f(x) P\nabla H(x) \mu(dx|y)} \notag\right. \\
&\left. \hspace{3cm} + \left(\int{f(x)\mu(dx|y)}\right)\left( \int{P\nabla H(x) \mu(dx|y)} \right) \right] \cdot \tilde{y}, \notag 
\end{align}
which is what we were aiming for.
\end{proof}

\section{Proof of the hydrodynamic limit}~\label{s_hydro_limit}
This section is devoted to the proof of the hydrodynamic limit, namely the abstract quantitative estimate of Theorem~\ref{thm_hydro_bounds} and the concrete application to the Kawasaki dynamics stated in Theorem~\ref{thm_lim_hydro}. This section is organized in the following way:
\begin{itemize}
\item In Section~\ref{s_proof_hydro_bounds} we give the complete proof of Theorem \ref{thm_hydro_bounds}.
\item In Section~\ref{s_hydro_kawasaki} we give the complete proof of Theorem~\ref{thm_lim_hydro} up to two missing ingredients, namely Proposition~\ref{convergence_macro_hydro} and Proposition~\ref{lp_bounds_energy}.  Proposition~\ref{convergence_macro_hydro} contains the convergence of the deterministic macroscopic ODE given by~\eqref{macro_evol} to the nonlinear heat equation given by~\ref{hydro_equn} on the continuum. Proposition~\ref{lp_bounds_energy} contains polynomial bounds on the free energy.
\item In Section~\ref{s_aux_1_poly_bounds_energy}, we state and prove Proposition~\ref{lp_bounds_energy}.
\item In Section~\ref{s_aux_2_conv_free_energy}, we state and prove auxiliary results needed for Proposition~\ref{convergence_macro_hydro}.
\item And finally in Section~\ref{s_proof_of_conv_ODE_PDE}, we prove Proposition~\ref{convergence_macro_hydro}.  
\end{itemize}

\subsection{Proof of Theorem \ref{thm_hydro_bounds}}\label{s_proof_hydro_bounds}

The first element of the proof is the following lemma, which will allow to control fluctuations.  

\begin{lem} \label{pen_fluct}
There exists a constant $\gamma > 0$, which is independent of $N$ and $M$, such that for any $x \in X_{N,0}$, we have
\begin{equation} \label{pen_fluct1}
|(id_X - P^tNP)x|_X^2 \leq \frac{\gamma}{M^2}\langle Ax, x \rangle_X;
\end{equation}

\begin{equation} \label{pen_fluct2}
\langle A^{-1}(id_X - P^tNP)x, (id_X - P^tNP)x \rangle \leq \frac{\gamma}{M^2}|x|_X^2.
\end{equation}
\end{lem}
Of course, these statements are not new and we copy the argument of (54) in~[GOVW].

\begin{proof}
Bound (\ref{pen_fluct1}) is a consequence of the discrete Poincar\'e inequality applied to the blocks: there exists a universal constant $\gamma$ such that for any $N \in \N$ and real numbers $x_1, .. x_N$ with $\sum x_i = 0$, we have
$$\underset{i = 1}{\stackrel{N}{\sum}} \hspace{1mm} x_i^2 \leq \gamma N^2 \underset{i = 1}{\stackrel{N}{\sum}} \hspace{1mm} (x_{i+1} - x_i)^2.$$

Estimate (\ref{pen_fluct2}) follows from (\ref{pen_fluct1}) and the Cauchy-Schwartz inequality : 

\begin{align}
\langle A^{-1}&(id_X - P^tNP)x, (id_X - P^tNP)x \rangle = \langle x, (id_X - P^tNP)A^{-1}(id_X - P^tNP)x \rangle \notag \\
&\leq |x| |(id_X - P^tNP)A^{-1}(id_X - P^tNP)x| \notag \\
&\leq |x| \left(\frac{\gamma}{M^2}\langle AA^{-1}(id_X - P^tNP)x, A^{-1}(id_X - P^tNP)x \rangle \right)^{1/2} \notag \\
&= |x| \left(\frac{\gamma}{M^2}\langle (id_X - P^tNP)x, A^{-1}(id_X - P^tNP)x \rangle \right)^{1/2}. \notag
\end{align}

\end{proof}

We will also need the following energy and moment estimates, which are those of Proposition 24 in [GOVW]. For the sake of completness we give the proof of Proposition~\ref{energy} at the end of the section.

\begin{prop} \label{energy}
If $f(t,x)$ and $\eta(t)$ satisfy the assumptions of Theorem \ref{thm_hydro_bounds}, then for any $T < +\infty$ we have

\begin{equation} \label{energy1}
\Ent_{\mu}(f(T,\cdot)) + \int_0^T{\left(\int_X{\frac{\langle A \nabla f, \nabla f\rangle}{f}d\mu}\right)dt} = \Ent_{\mu}(f_0);
\end{equation}

\begin{equation} \label{energy2}
\bar{H}(\eta(T)) + \int_0^T{\left\langle \frac{d\eta}{dt}, \bar{A}^{-1}\frac{d\eta}{dt} \right\rangle dt} = \bar{H}(\eta(0));
\end{equation}

\begin{equation} \label{energy3}
\left(\int{|x|^2f(t,x)\mu(dx)} \right)^{1/2} \leq \left( \frac{2}{\rho}\Ent_{\mu}(f_0)\right)^{1/2} + \left(\int{|x|^2\mu(dx)} \right)^{1/2}
\end{equation}
\end{prop}

We will now prove estimate (\ref{estimate_micro_macro}) in three steps: first we will compute the time derivative of $\Theta$, 
then we will bound it, and finally integrate it in time to obtain our result. 
The calculations are exactly the same as in [GOVW], except for the use of the new covariance estimate of Proposition \ref{p_crucial_covariance_two_scale}

\textbf{Step 1} : Our aim is to obtain the exact formula 
\begin{align} \label{step1}
\frac{d}{dt}\Theta(t) &= -\int{(\nabla_Y \bar{H}(y) - \nabla_Y \bar{H}(\eta))\cdot(y - \eta)\bar{f}(y)\bar{\mu}(dy)} + \frac{M-1}{N} \notag \\
&- \int{(y - \eta)\cdot P\cov_{\mu(dx|y)}(f, \nabla H) \bar{\mu}(dy)} \notag \\
&- \int{ \frac{1}{N}(id - P^tNP)x \cdot \nabla f \mu(dx)} - \int{\frac{d\eta}{dt}\cdot PA^{-1}(id - P^tNP)x \hspace{1mm} f \mu(dx)}. 
\end{align}

Using the definition of the stochastic dynamic, of the coarse-grained evolution (\ref{macro_evol}) and the splitting $x = P^tNPx + (id_X - P^tNP)x$, we have

\begin{align}
\frac{d}{dt}\Theta(t) &= \frac{d}{dt}\frac{1}{2N}\int{\langle A^{-1}(x - NP^t\eta(t)), (x - NP^t\eta(t)) \rangle f(t,x)\mu_{N,a,m}(dx)} \notag \\
&\stackrel{(\ref{micro_evolution})}{=} -\int{\frac{1}{N}A^{-1}( x - NP^t\eta(t)) \cdot A\nabla f(t,x)\rangle \mu(dx)} - \int{P^t\frac{d\eta}{dt} \cdot A^{-1}(x - NP^t\eta) f\mu(dx)} \notag \\
&= -\int{P^t(Px - \eta)\cdot \nabla f \mu(dx)} - \int{PA^{-1}NP^t\frac{d\eta}{dt} \cdot (Px - \eta) f\mu(dx)} \notag \\
& \hspace{2mm} - \int{ \frac{1}{N}(id - P^tNP)x \cdot \nabla f \mu(dx)} - \int{\frac{d\eta}{dt}\cdot PA^{-1}(id - P^tNP)x \hspace{1mm} f \mu(dx)} \notag \\
&= -\int{(Px - \eta)\cdot P\nabla f \mu(dx)} - \int{\bar{A}^{-1}\frac{d\eta}{dt} \cdot (Px - \eta) f\mu(dx)} \notag \\
& \hspace{2mm} - \int{ \frac{1}{N}(id - P^tNP)x \cdot \nabla f \mu(dx)} - \int{\frac{d\eta}{dt}\cdot PA^{-1}(id - P^tNP)x \hspace{1mm} f \mu(dx)} \notag \\
&= -\int{(y - \eta)\cdot P\int{\nabla f \mu(dx|y)}\bar{\mu}(dy)} + \int{\nabla_Y \bar{H}(\eta) \cdot (y - \eta) \bar{f}(y)\bar{\mu}(dy)} \notag \\
& \hspace{2mm} - \int{ \frac{1}{N}(id - P^tNP)x \cdot \nabla f \mu(dx)} - \int{\frac{d\eta}{dt}\cdot PA^{-1}(id - P^tNP)x \hspace{1mm} f \mu(dx)}.
\end{align}

We keep the last three terms unchanged, and transform the first one according to Lemma \ref{covlem} 
\begin{align}
-\int&{(y - \eta)\cdot P\int{\nabla f \mu(dx|y)}\bar{\mu}(dy)} \notag \\
&= -\frac{1}{N}\int{(y - \eta)\cdot \nabla_Y \bar{f}(y) \bar{\mu}(dy)} - \int{(y - \eta)\cdot P\cov_{\mu(dx|y)}(f, \nabla H) \bar{\mu}(dy)}.
\end{align}

Using integration by parts, the first term on the right-hand side of this last equation becomes

\begin{align}
-\frac{1}{N}&\int{(y - \eta)\cdot \nabla_Y \bar{f}(y) \bar{\mu}(dy)} \notag \\
&= \frac{1}{N}\int{(\nabla_Y \cdot y)\bar{f}(y)\bar{\mu}(dy)} - \int{(y - \eta) \cdot \nabla_Y \bar{H}(y) \bar{\mu}(dy)} \notag \\
&= \frac{M-1}{N}  - \int{(y - \eta) \cdot \nabla_Y \bar{H}(y) \bar{\mu}(dy)}.
\end{align}

This final estimate allows us to conclude step 1.

\textbf{Step 2} : \textit{Upper bound on the time-derivative} Our aim here is to get the upper bound
\begin{align} \label{step2}
\frac{d}{dt}\Theta(t) & + \frac{\lambda}{2}\int{|y - \eta|_Y^2\bar{f}\bar{\mu}(dy)} \notag \\
&\leq \frac{M-1}{N}  + \frac{\gamma C_0K}{2\lambda NM^2}\int{\frac{\langle A\nabla f, \nabla f \rangle_X}{f}\mu(dx)} \notag \\
&+ \left( \frac{\gamma}{M^2}\int{\frac{1}{Nf}\langle A\nabla f, \nabla f \rangle \mu(dx)} \cdot \int{\frac{1}{N}|x|_X^2f(x)\mu(dx)} \right)^{1/2} \notag \\
&+ \left( \frac{d\eta}{dt} \cdot \bar{A}^{-1} \frac{d\eta}{dt} \right)^{1/2} \left(\int{\frac{\gamma}{NM^2}|x|_X^2 f(x)\mu(dx)} \right)^{1/2}
\end{align}

We will now individually bound each term in (\ref{step1}). The first term can be bounded by using the assumption that $K$ is large enough such that $\bar{H}$ is $\lambda-$ uniformly convex, which yields
\begin{equation}
-\int{(\nabla_Y \bar{H}(y) - \nabla_Y \bar{H}(\eta))\cdot(y - \eta)\bar{f}(y)\bar{\mu}(dy)} \leq -\lambda\int{|y - \eta|_Y^2 \bar{f}\bar{\mu}(dy)}.
\end{equation}

The third term can be treated in the same way as was done for estimate (\ref{aleph4}) in the proof of the LSI, and by using Lemma \ref{pen_fluct}: 
\begin{align}
\int&{(y - \eta)\cdot P\cov_{\mu(dx|y)}(f, \nabla H) \bar{\mu}(dy)} \notag \\
&\leq \left( \int{|y - \eta|_Y^2 \bar{f}(y)\bar{\mu}(dy)} \right)^{1/2} \left( \int{\frac{|P\cov_{\mu(dx|y)}(f, \nabla H)|^2}{\bar{f}(y)}\bar{\mu}(dy)} \right)^{1/2} \notag \\
&\leq \frac{\lambda}{2}\int{|y - \eta|_Y^2 \bar{f}(y)\bar{\mu}(dy)} + \frac{1}{2\lambda N} \int{\frac{|P^tNP\cov_{\mu(dx|y)}(f, \nabla H)|^2}{\bar{f}(y)}\bar{\mu}(dy)} \notag \\
&\leq \frac{\lambda}{2}\int{|y - \eta|_Y^2 \bar{f}(y)\bar{\mu}(dy)} + \frac{C_0K}{2\lambda N}\int{\frac{|(id_X - P^tNP)\nabla f|^2}{f} \mu(dx)} \notag \\
&\leq \frac{\lambda}{2}\int{|y - \eta|_Y^2 \bar{f}(y)\bar{\mu}(dy)} + \frac{\gamma C_0K}{2\lambda NM^2}\int{\frac{\langle A\nabla f, \nabla f \rangle_X}{f}\mu(dx)}.
\end{align}

The fourth term can be controlled using (\ref{pen_fluct2}) and the Cauchy-Schwartz inequality: 

\begin{align}
\left|\right.\int&{\frac{1}{N}(id - P^tNP)x \cdot \nabla f \mu(dx)}\left.\right| \notag \\
&\leq \left(\frac{1}{N}\int{\langle A^{-1}(id - P^tNP)x, (id - P^tNP)x \rangle f(x)\mu(dx)} \cdot \int{\frac{1}{Nf}\langle A\nabla f, \nabla f \rangle \mu(dx)} \right)^{1/2} \notag \\
&\leq \left( \frac{\gamma}{M^2}\int{\frac{1}{Nf}\langle A\nabla f, \nabla f \rangle \mu(dx)} \cdot \int{\frac{1}{N}|x|_X^2f(x)\mu(dx)} \right)^{1/2}
\end{align}

Similarly, we have for the fifth term
\begin{align}
\left|\right. \int&{\frac{d\eta}{dt}\cdot PA^{-1}(id - P^tNP)x \hspace{1mm} f \mu(dx)} \left.\right| \notag \\
&\leq \left(\int{P^t\frac{d\eta}{dt} \cdot A^{-1}NP^t\frac{d\eta}{dt} f{x}\mu(dx)} \right)^{1/2} \notag \\
&\hspace{1cm} \cdot \left(\int{\frac{1}{N}\langle A^{-1}(id - P^tNP)x, (id - P^tNP)x \rangle f(x)\mu(dx)} \right)^{1/2} \notag \\
&\leq \left( \frac{d\eta}{dt} \cdot \bar{A}^{-1} \frac{d\eta}{dt} \right)^{1/2} \left(\int{\frac{\gamma}{NM^2}|x|_X^2 f(x)\mu(dx)} \right)^{1/2}
\end{align}

Combining these estimates gives the desired upper bound. 

\textbf{Step 3} : \textit{Time integration and conclusion} Integrating (\ref{step2}) with respect to time yields
\begin{align} \label{bnd_integ}
\max & \left\{ \underset{0 \leq t \leq T}{\sup} \hspace{1mm} \Theta(t), \hspace{2mm} \frac{\lambda}{2}\int_0^T{\int_Y{|y - \eta(t)|_Y^2 \bar{f}\bar{\mu}(dy)}  \hspace{1mm}dt} \right\} \notag \\
&\leq \Theta(0) + T\frac{M-1}{N} + \frac{\gamma C_0K}{2\lambda NM^2}\int_0^T{\int{\frac{\langle A\nabla f, \nabla f \rangle_X}{f}\mu(dx)} \hspace{1mm}dt} \notag \\
&\hspace{3mm} + \frac{\gamma^{1/2}}{M}\int_0^T{ \left(\int{\frac{1}{N}|x|_X^2f(x)\mu(dx)} \right)^{1/2}} \notag \\
&\hspace{5mm} \times \left( \left( \int{\frac{1}{Nf}\langle A\nabla f, \nabla f \rangle \mu(dx)}\right)^{1/2} + \left(\frac{d\eta}{dt} \cdot \bar{A}^{-1} \frac{d\eta}{dt} \right)^{1/2} \right) \hspace{1mm} dt.
\end{align}

According to Proposition \ref{energy}, since the entropy is positive, and the initial entropy satisfies the bound of assumption (i), we have

\begin{equation} \label{bnd_ent_prod}
\int_0^T{\int{\frac{\langle A\nabla f, \nabla f \rangle_X}{f}\mu(dx)} \hspace{1mm}dt} \leq C_1N.
\end{equation}

The bounds of Proposition \ref{energy} also yield the bound

\begin{align} \label{y_a_bnd}
\int_0^T&{ \left(\int{\frac{1}{N}|x|_X^2f(x)\mu(dx)} \right)^{1/2}} \times \left( \left( \int{\frac{1}{Nf}\langle A\nabla f, \nabla f \rangle \mu(dx)}\right)^{1/2} + \left(\frac{d\eta}{dt} \cdot \bar{A}^{-1} \frac{d\eta}{dt} \right)^{1/2} \right) \hspace{1mm} dt \notag \\
&\leq \left(\int_0^T{ \int{\frac{1}{N}|x|_X^2f(t,x)\mu(dx)}dt} \right)^{1/2} \notag \\
& \hspace{5mm} \times \left( \left( \int_0^T{\int{\frac{1}{Nf}\langle A\nabla f, \nabla f \rangle \mu(dx)}dt}\right)^{1/2} + \left(\int_0^T{\frac{d\eta}{dt} \cdot \bar{A}^{-1} \frac{d\eta}{dt} dt} \right)^{1/2} \right) \notag \\
&\leq \left(2\int_0^T {\left(\int{\frac{1}{N}|x|_X^2\mu(dx)} + \frac{2}{N\rho}\Ent_{\mu}(f_0) \right)dt} \right)^{1/2} \notag \\
& \hspace{5mm} \times \left( \left( \frac{1}{N}\Ent_{\mu}(f_0)\right)^{1/2} + \left(\bar{H}(\eta_0) - \bar{H}(\eta(T)) \right)^{1/2} \right) \notag \\
&\leq \left( 2\alpha + \frac{2C_1}{\rho} \right)^{1/2} \left( C_1 + C_2 + \beta \right)^{1/2}
\end{align}

Plugging (\ref{bnd_ent_prod}) and (\ref{y_a_bnd}) into (\ref{bnd_integ}) ends the proof.

We finish this section with the proof of Proposition \ref{energy}. For the proof we will use the following lemma.  It was proven in [GOVW] using a semigroup argument. Here we give a proof based on Talagrand's inequality.

\begin{lem} \label{lem42}
Assume $\mu$ is a probability measure on an Euclidean space $X$ that satisfies LSI($\rho$) for some $\rho > 0$. Then for any probability density $f$ with respect to $\mu$, we have 
$$\left(\int{|x|^2f(x)\mu(dx)} \right)^{1/2} \leq \left(\int{|x|^2\mu(dx)} \right)^{1/2} + \left( \frac{2}{\rho} \Ent_{\mu}(f) \right)^{1/2}.$$
\end{lem}

\begin{proof}
Since $\mu$ satisfies LSI($\rho$), it also satisfies Talagrand's transport-entropy inequality
$$W_2^2(\mu, \nu) \leq \frac{2}{\rho} \Ent_{\mu}(\nu)$$
for any probability measure $\nu$, where $W_2$ is the Wasserstein distance
$$W_2^2(\mu, \nu) = \underset{\pi \in \Pi}{\inf} \hspace{1mm} \int{|x - y|^2 \pi(dx, dy)},$$
where $\Pi$ is the set of couplings of $\mu$ and $\nu$, that is the set of all probability measure $\pi$ on $X^2$ such that its first marginal is $\mu$ and its second marginal is $\nu$. Proofs of the fact that the LSI implies this transport-entropy inequality (a result which is known as the Otto-Villani theorem) can be found in [OV] and [Go].

Let $\pi$ be a coupling of $\mu$ and $f\mu$. We have, by the Cauchy-Schwartz inequality
$$\int{\langle x, y \rangle \pi(dx, dy)} \geq -\left(\int{|x|^2f(x)\mu(dx)} \right)^{1/2}\left(\int{|x|^2\mu(dx)} \right)^{1/2}$$
so that
\begin{align}
\int & |x - y|^2 \pi(dx, dy) = \int{|x|^2f(x)\mu(dx)} + \int{|x|^2\mu(dx)} - 2 \int{\langle x, y \rangle \pi(dx, dy)} \notag \\
&\geq \int{|x|^2f(x)\mu(dx)} + \int{|x|^2\mu(dx)} - 2 \left(\int{|x|^2f(x)\mu(dx)} \right)^{1/2}\left(\int{|x|^2\mu(dx)} \right)^{1/2} \notag \\
&= \left(\left(\int{|x|^2f(x)\mu(dx)} \right)^{1/2}-\left(\int{|x|^2\mu(dx)} \right)^{1/2} \right)^2. \notag
\end{align}
Taking the infimum over $\pi \in \Pi$ and applying the transport-entropy inequality above, we get
$$\left(\int{|x|^2f(x)\mu(dx)} \right)^{1/2}-\left(\int{|x|^2\mu(dx)}\right)^{1/2} \leq \left( \frac{2}{\rho} \Ent_{\mu}(f) \right)^{1/2}$$
which, after rearranging the terms, is the desired inequality.
\end{proof}

\begin{proof} [Proof of Proposition \ref{energy}]
To get (\ref{energy1}), we use the definition of the microscopic dynamics to get
$$\frac{d}{dt}\int{f(t,x) \log f(t,x) \mu(dx)} = -\int{\frac{\langle A \nabla f, \nabla f \rangle}{f} \mu}$$
and integrating yields 
$$\int_0^T{\int{\frac{\langle A \nabla f, \nabla f \rangle}{f} \mu} dt} = \Ent_{\mu}(f_0) - \Ent_{\mu}(f(T,\cdot))$$
which is what we want. 

In the same way, to get (\ref{energy2}), we use (\ref{macro_evol}) to get 
$$\frac{d}{dt}\bar{H}(\eta(t)) = -\langle \bar{A}\nabla \bar{H}(\eta(t)), \nabla \bar{H}(\eta(t)) \rangle_Y$$
and integrating this yields the desired result.

Finally, we get (\ref{energy3}) simply by applying Lemma \ref{lem42} and the fact that the relative entropy is decrasing, which means that by~(\ref{energy1}) it holds
\begin{align*}
 Ent_\mu (f(t)) \leq  Ent_\mu (f_0) .
\end{align*}
.
\end{proof}

\subsection{Proof of Theorem \ref{thm_lim_hydro}}\label{s_hydro_kawasaki}

We shall prove Theorem \ref{thm_lim_hydro} in two steps: first we shall use Theorem~\ref{thm_hydro_bounds} to show that our data is asymptotically close, 
in the sense of the $H^{-1}$ norm, to a sequence of deterministic macroscopic vectors given by (\ref{macro_evol}), 
and then we will show that the step functions associated to these vectors converge to the solution of the hydrodynamic equation (\ref{hydro_equn}).

To apply Theorem \ref{thm_hydro_bounds}, we consider a sequence $\{M_{\ell}, N_{\ell} \}_{\ell = 1}^{\infty}$ of integers with
\begin{equation}
M_{\ell} \uparrow +\infty, \hspace{1cm} N_{\ell} \uparrow +\infty, \hspace{1cm} K_{\ell} := \frac{N_{\ell}}{M_{\ell}} \uparrow +\infty.
\end{equation}

To simplify notations, we shall often not explicitly write the dependence of our various objects on $\ell$, and just write $N$, $M$ and $K$ for $N_{\ell}$, $M_{\ell}$ and $K_{\ell}$.

Let $\bar{\eta}_0^{\ell} \in \bar{Y}_M$ be a step function approximation of $\zeta_0$ with 
\begin{equation}
\underset{\ell \uparrow \infty}{\lim} ||\bar{\eta}_0^{\ell} - \zeta_0 ||_{L^2} \hspace{1mm} = 0.
\end{equation}
Since by assumption $\zeta_0$ lies in $L^p(\T)$, we can take a sequence such that
\begin{equation} \label{bound_eta}
\sup ||\bar{\eta}_0^{\ell}||_{L^p} \leq C.
\end{equation}

Let $\eta_0^{\ell}$ be the vector associated to $\bar{\eta}_0^{\ell}$, and consider the solutions $\eta^{\ell}$ of 

\begin{equation} \label{eqn_macro_seq}
\frac{d\eta^{\ell}}{dt} = -A\nabla_Y \bar{H}_{N,K,a}(\eta^{\ell}), \hspace{1cm} \eta^{\ell}(0) = \eta_0^{\ell}.
\end{equation}

Then we have the following result, which will be the key to pass from Theorem \ref{thm_hydro_bounds} to Theorem \ref{thm_lim_hydro}.

\begin{prop} \label{convergence_macro_hydro}
Under the notations above, and for almost every realization of the family of random variables $(a_q)$, the step functions $\bar{\eta}^{\ell}$ converge strongly in $L^{\infty}(H^{-1})$ to the unique weak solution of 
$$\frac{\partial \zeta}{\partial t} = \frac{\partial^2}{\partial \theta^2}\tilde{\varphi}(\zeta), \hspace{1cm} \zeta(0,\cdot) = \zeta_0.$$
\end{prop}

We shall defer the proof of this result to Section 5.5, and use it to prove our theorem. It will use several auxiliary results on the behavior of the coarse-grained Hamiltonian $\bar{H}$, that will be proven in Sections 5.3 and 5.4.

First, let us show that the assumptions of Theorem \ref{thm_hydro_bounds} hold for every $N$ and $M$ large enough, with uniform constants. 

Assumption (iv) is given by the following lemma:

\begin{lem} \label{lem43}
Assume that the real numbers $a_i$ satisfy $\sup |a_i| \leq L$. Then there exists a constant $\alpha$, which only depends on $\psi$, $L$ and $m$ such that
$$\int{|x|^2\mu_{N,a,m}(dx)} \leq \alpha N.$$
\end{lem}

\begin{proof}
Our proof of this result will be based on Lemma \ref{lem42}. We apply its result with reference measure 
\begin{align}\label{e_def_reference_measure}
  \mu_{N,m}^0 = \frac{1}{Z}\exp \left( - \sum \psi(x_i) \right) dx . 
\end{align}
Note that the measure~$\mu_{N,m}^0$  satisfies LSI($\rho$) and $f\mu_{N,m}^0 = \mu_{N,a,m}$, so that
$$\left(\int{|x|^2\mu_{N,a,m}(dx)} \right)^{1/2} \leq \left( \int{|x|^2\mu_{N,m}^0(dx)} \right)^{1/2}$$
.
$$ + \left( \frac{2}{\rho} \int{ \frac{1}{Z}\left(-\sum a_i x_i - \log Z \right)\exp \left(-\sum a_i x_i \right) \mu_{N,m}^0(dx)} \right)^{1/2}.$$
Using this and 
\begin{align*}
 0 &\leq \Ent_{\mu_{N,m}^0} \left( \frac{1}{Z}\exp(- \sum a_i x_i) \right) \\
& = \left( \frac{2}{\rho} \int{ \frac{1}{Z}\left(-\sum a_i x_i - \log Z \right)\exp \left(-\sum a_i x_i \right) \mu_{N,m}^0(dx)} \right)^{1/2} 
\end{align*}
we get 
\begin{align}
\int{|x|^2\mu_{N,a,m}(dx)} &\leq 2\int{|x|^2\mu_{N,m}^0(dx)} - \frac{4}{\rho}\int{ \left(-\sum a_i x_i \right)\mu_{N,a,m}(dx)} \notag \\
& \hspace{1cm} - \frac{4}{\rho}\log \left( \int{\exp \left(- \sum a_i x_i \right)\mu_{N,m}^0(dx)} \right) \notag \\
&\leq 2\int{|x|^2\mu_{N,m}^0(dx)} + \frac{4}{\rho}\left( \sum a_i^2 \right)^{1/2} \left(\int{|x|^2\mu_{N,a,m}(dx)} \right)^{1/2} \notag \\
& \hspace{1cm} - \frac{4}{\rho}\int{ \left( \sum a_i x_i \right)\mu_{N,m}^0(dx)} \notag \\
&\leq 2\int{|x|^2\mu_{N,m}^0(dx)} + \frac{4}{\rho} \sqrt{NL}\left(\int{|x|^2\mu_{N,a,m}(dx)} \right)^{1/2} \notag \\
& \hspace{1cm} + \frac{4}{\rho} \sqrt{NL}\left(\int{|x|^2\mu_{N,m}^0(dx)} \right)^{1/2}.
\end{align}
Applying the Poincar\'e inequality to $\mu_{N,m}^0(dx)$, we have
$$\int{x_i^2 \mu_{N,m}^0(dx)} \leq \frac{1}{\rho} \int{|\nabla x_i|^2 \mu_{N,m}^0(dx)} + \left( \int{x_i \mu_{N,m}^0(dx)} \right)^2 \leq \frac{1}{\rho} + m^2$$
since, for $\mu_{N,m}^0(dx)$, the coordinates are exchangeable and therefore all have the same mean $m$. This implies that
$$\int{|x|^2\mu_{N,m}^0(dx)} \leq \left( \frac{1}{\rho} + m^2 \right)N$$
and therefore
$$\int{|x|^2\mu_{N,a,m}(dx)} \leq CN + C\sqrt{N} \sqrt{\int{|x|^2\mu_{N,a,m}(dx)}}.$$
Using Young's inequality on the last term on the right hand side, it is then easy to show that 
$$\int{|x|^2\mu_{N,a,m}(dx)} \leq \alpha N$$
for some constant $\alpha$ which doesn't depend on $N$ or the $a_i$.
\end{proof}

Let us now turn to Assumption (i): Recalling our choice of $f_{0,a}$ (see Theorem~\ref{thm_lim_hydro}) and the definition~\eqref{e_def_reference_measure} of the measure~$\mu_N^0$, we have
\begin{align}
\Ent_{\mu_{N,a}}(f_{0,a}) &= \int{ \log \left( \frac{dF_N}{d\mu_{N,a}} \right) dF_N} \notag \\
&= \int{ \log \left( \frac{dF_N}{d\mu^0_{N}} \right) dF_N} + \int{ \log \left( \frac{d\mu^0_{N}}{d\mu_{N,a}} \right) dF_N} \notag \\
&\overset{\eqref{e_closness_initial_data}}{\leq} CN + \int{ \left(\sum a_i x_i \right)F_N(dx)} - \log \int{\exp \left( \sum a_i x_i \right)\mu_{N,a}(dx)} \notag \\
&\leq CN + \sqrt{\sum a_i^2} \sqrt{\int{|x|^2F_N(dx)}} - \int{ \left(\sum a_i x_i \right)\mu_{N,a}(dx)} \notag \\
&\leq CN + \sqrt{LN}\sqrt{\int{|x|^2F_N(dx)}} + \sqrt{LN} \sqrt{\int{|x|^2\mu_{N,a}(dx)}}.
\end{align}
We already know by Lemma \ref{lem43} that $\int{|x|^2\mu_{N,a}(dx)} \leq CN$. Moreover, we can use Lemma \ref{lem42}, the fact that $\Ent_{\mu_N^0}(F_N) \leq CN$ (see~\eqref{e_closness_initial_data}) and Lemma \ref{lem43} to show that $\int{|x|^2F_N(dx)} \leq CN$, and we then obtain assumption (i) of Theorem \ref{thm_hydro_bounds} with uniform constant $C_1$. 

Let us turn to the Assumption (ii) and (iii), which follow from Proposition~\ref{lp_bounds_energy}. Indeed, recall that 
$$\bar{H}_{N,K,a}(y) = \frac{1}{M}\underset{i = 1}{\stackrel{M}{\sum}} \hspace{1mm} \psi_{N,K,a}(y_i) + \frac{1}{N}\log Z.$$
Since 
$$\frac{1}{N}\log Z = \frac{1}{N}\log \int{\exp \left( - \sum \psi(x_i) + a_i x_i \right) dx},$$
it is not hard to deduce from (\ref{poly_growth}) and the boundedness of the $a_i$ that this quantity is bounded both from above and from below, so that, in order to prove (ii) and (iii), we can assume without loss of generality that $\bar{H}$ is given by 
$$\bar{H}_{N,K,a}(y) = \frac{1}{M}\underset{i = 1}{\stackrel{M}{\sum}} \hspace{1mm} \psi_{N,K,a}(y_i).$$
Assumption (ii) then follows from (\ref{bound_eta}) and part (iv) of Proposition \ref{lp_bounds_energy} from below. Assumption (iii) also follows from part (iv) of Proposition \ref{lp_bounds_energy}.

Finally, Assumption (v) follows directly from Proposition~\ref{p_convexity_coarse_grained_Hamiltonian}. \medskip

Before turning to the proof of Theorem~\ref{thm_lim_hydro}, we need two more ingredients: they are estimates comparing the $H^{-1}$ norm to the $A^{-1}$ norm on $X_N$. 

\begin{lem} \label{compare_norms}
There exists a constant $C < \infty$ such that for any $x \in X$, if $\bar{x}$ is the associated step function, then 

(i) 
$$\frac{1}{C}\langle \bar{x}, \bar{x} \rangle_{H^{-1}} \leq \frac{1}{N}\langle A^{-1}x, x \rangle_X \leq C \langle \bar{x}, \bar{x} \rangle_{H^{-1}};$$

(ii) If $x$ is bounded in $L^2$, then 
$$\left| \langle \bar{x}, \bar{x} \rangle_{H^{-1}} - \frac{1}{N}\langle A^{-1}x, x \rangle_X \right| \leq \frac{C}{N}.$$
\end{lem}
These estimates were already included in [GOVW], but for the sake of completness we reproduce the proof.
\begin{proof}
First, we can express the discrete norm as
$$\frac{1}{N}\langle A^{-1}x, x \rangle_X = \frac{1}{N} \underset{i = 1}{\stackrel{N}{\sum}} \hspace{1mm} F_i^2$$
.
$$\text{where } x_i = N(F_{i+1} - F_i) \hspace{3mm} \text{ and } \hspace{3mm} \underset{i = 1}{\stackrel{N}{\sum}} \hspace{1mm} F_i = 0.$$

To estimate $\langle \bar{x}, \bar{x} \rangle_{H^{-1}}$ through the $F_i$, recall that by definition
$$\langle \bar{x}, \bar{x} \rangle_{H^{-1}} = \int_{\T}{w(\theta)^2d\theta}, \hspace{5mm} \text{where } w' = f \hspace{3mm} \text{and } \int_{\T}{w(\theta)d\theta} = 0.$$
It is easy to check that we can define such a $w$ with
$$w(\theta) = F_i + N(F_{i+1} - F_i) \left( \theta - \frac{i-1}{N} \right) \hspace{3mm} \text{for } \theta \in \left[\frac{i-1}{N}, \frac{i}{N} \right).$$

Consequently, we have
\begin{align} \label{bounds_norms}
\langle \bar{x}, \bar{x} \rangle_{H^{-1}} &= \underset{i = 1}{\stackrel{N}{\sum}} \hspace{1mm}  \int_0^{N^{-1}}{(F_i + N(F_{i+1} - F_i)\theta)^2d\theta} \notag \\
&= \frac{1}{N} \underset{i = 1}{\stackrel{N}{\sum}} \hspace{1mm}  \left( F_{i}^2 + (F_{i+1} - F_i)F_i + \frac{1}{3}(F_{i+1} - F_i)^2 \right) \notag \\
&= \frac{1}{N}\langle A^{-1}x, x \rangle_X + \frac{1}{N} \underset{i = 1}{\stackrel{N}{\sum}} \hspace{1mm}  \left( (F_{i+1} - F_i)F_i + \frac{1}{3}(F_{i+1} - F_i)^2 \right). 
\end{align}

We have the bounds on the second term on the right-hand side
$$-\frac{2}{3N}\underset{i = 1}{\stackrel{N}{\sum}} \hspace{1mm} F_i^2 \leq \frac{1}{N} \underset{i = 1}{\stackrel{N}{\sum}} \hspace{1mm}  \left( (F_{i+1} - F_i)F_i + \frac{1}{3}(F_{i+1} - F_i)^2 \right) \leq 0$$
so that 
$$\langle \bar{x}, \bar{x} \rangle_{H^{-1}} \leq \frac{1}{N}\langle A^{-1}x, x \rangle_X \leq 3\langle \bar{x}, \bar{x} \rangle_{H^{-1}}.$$
Moreover, if $\bar{x}$ is bounded in $L^2$, then 
$$\frac{1}{N} \underset{i = 1}{\stackrel{N}{\sum}} \hspace{1mm} x_i^2 = \underset{i = 1}{\stackrel{N}{\sum}} \hspace{1mm} N(F_{i+1} - F_i)^2 \leq C \hspace{3mm} \text{and  } \frac{1}{N} \underset{i = 1}{\stackrel{N}{\sum}} \hspace{1mm} F_i^2 \leq C,$$
so we get from (\ref{bounds_norms}) 
\begin{align}
&\left|\langle \bar{x}, \bar{x} \rangle_{H^{-1}} - \frac{1}{N}\langle A^{-1}x, x \rangle_X \right| \notag \\
&= \frac{1}{N} \left| \underset{i = 1}{\stackrel{N}{\sum}} \hspace{1mm} \left( (F_{i+1} - F_i)F_i + \frac{1}{3}(F_{i+1} - F_i)^2 \right) \right| \notag \\
&\leq \frac{1}{N} \left(\underset{i = 1}{\stackrel{N}{\sum}} \hspace{1mm}  (F_{i+1} - F_i)^2\right)^{1/2} \left(\underset{i = 1}{\stackrel{N}{\sum}} \hspace{1mm} F_i^2\right)^{1/2} + \frac{1}{3N}\underset{i = 1}{\stackrel{N}{\sum}} \hspace{1mm} (F_{i+1} - F_i)^2 \notag \\
&\leq \frac{C}{N} \notag
\end{align}

\end{proof}

We are now ready to give the proof of our main result.

\begin{proof} [Proof of Theorem \ref{thm_lim_hydro}]
First, we have 
$$\underset{ \ell \uparrow \infty}{\lim} \hspace{1mm} \underset{t \in [0,T]}{\sup} \hspace{1mm} \int_{X_N}{||\bar{x} - \zeta(t, \cdot)||_{H^{-1}}^2 f(t,x)\mu_{N,a}(dx)} $$

$$\leq 2 \hspace{1mm} \underset{ \ell \uparrow \infty}{\lim} \hspace{1mm} \underset{t \in [0,T]}{\sup} \hspace{1mm} \int_{X_N}{||\bar{x} - \bar{\eta}_{\ell}(t)||_{H^{-1}}^2 f(t,x)\mu_{N,a}(dx)} + 2 \hspace{1mm} \underset{ \ell \uparrow \infty}{\lim} \hspace{1mm} \underset{t \in [0,T]}{\sup} \hspace{1mm} ||\bar{\eta}_{\ell}(t) - \zeta(t,\cdot)||_{H^{-1}}^2.$$

Applying Proposition \ref{convergence_macro_hydro}, we immediately see that the second term on the right-hand side of this equation goes to zero almost surely. For the first term, using bound (i) of Lemma \ref{compare_norms}, we see that, for any realization of the random field,
$$ \int{\frac{1}{N}\langle (x - NP^t\eta_{0,a}), A^{-1}(x - NP^t\eta_{0,a}) \rangle F_N(dx)}$$ 
$$\leq C\int_{X_N}{||\bar{x} - \bar{\eta}_{\ell,0}||_{H^{-1}}^2 f_{0,a}\mu_{N,a}(dx)} \longrightarrow 0$$
so that an application of Theorem \ref{thm_hydro_bounds} yields that 
$$\underset{ \ell \uparrow \infty}{\lim} \hspace{1mm} \underset{t \in [0,T]}{\sup} \int{\frac{1}{N}\langle (x - NP^t\eta_{0,a}), A^{-1}(x - NP^t\eta_{\ell,a}(t)) \rangle f(t,x)\mu_{N,a}(dx)} = 0$$
for any realization of the random field. Another application of bound (i) in Lemma \ref{compare_norms} then yields
$$\underset{ \ell \uparrow \infty}{\lim} \hspace{1mm} \underset{t \in [0,T]}{\sup} \hspace{1mm} \int_{X_N}{||\bar{x} - \bar{\eta}_{\ell}(t)||_{H^{-1}}^2 f(t,x)\mu_{N,a}(dx)} = 0$$
also for any realization of the random field, which concludes the proof.

The remainder of this section is devoted to the proof of Propoposition \ref{convergence_macro_hydro}

\end{proof}

\subsection{Auxiliary result: Polynomial bounds on the energy}\label{s_aux_1_poly_bounds_energy}

In this section, our aim will be to obtain polynomial bounds on the macroscopic and hydrodynamic free energies $\psi_{N,K,i,a}$, $\varphi_{N,K,i,a}$ and $\tilde{\varphi}$ (for the definitions see~\eqref{d_psi_K}, \eqref{d_varphi_K}, and \eqref{average_hydro_potential} respectively). The main tool will be the following Lemma, due to Caputo (see~[Cap, Lemma 2.4]).

\begin{lem} \label{lem_caputo}
Let 
$$s(\sigma)^2 := \int{(x - m)^2\mu_{\sigma(dx)}}$$
where $\mu_{\sigma(dx)} = Z^{-1}\exp(\sigma x - \psi(x))dx$ and $m = \int{x \mu_{\sigma}(dx)}$. Then 
$$\frac{1}{C \psi_c''(m)} \leq s(\sigma)^2 \leq \frac{C}{\psi_c''(m)}.$$
\end{lem}
For a proof of this statement we refer to [Cap, Lemma 2.4].

We define 
\begin{equation}\label{d_psi_K}
\psi_{N,K,i,a}(m) := -\frac{1}{K}\log \int_{X_{K,m}}{\exp\left(-\underset{j = (i-1)K + 1}{\stackrel{iK}{\sum}} \hspace{1mm} a_{j/N} x_j + \psi(x_j)\right)dx}
\end{equation}

and 

\begin{equation}\label{d_varphi_K}
\varphi_{N,K,i,a}(m) := \underset{\sigma \in \mathbb{R}}{\sup} \hspace{1mm} \left(\sigma m - \frac{1}{K} \underset{j = (i-1)K + 1 }{\stackrel{iK}{\sum}} \hspace{1mm} \log \int_{\mathbb{R}}{\exp((\sigma - a_{j/N})x - \psi(x))dx} \right)
\end{equation}

Theorem~\ref{lct} from the appendix tells us that these two functions behave in the same manner. We will deduce bounds on $\varphi_{N,K,i,a}$ from Lemma \ref{lem_caputo}, which will then carry on to $\psi_{N,K,i,a}$. Since these bounds will only depend on $K$ and $\sup |a_i|$, we simplify notations by writing $\psi_K$ instead of $\psi_{N,K,i,a}$ and $\varphi_K$ instead of $\varphi_{N,K,i,a}$ in this section.

\begin{prop} \label{lp_bounds_energy}
There exists $K_0$ and $C > 0$ which only depend on $\psi$ and $\sup a_i$ such that, for any $m \in \R$ and $K \geq K_0$\newline
(i) $\frac{1}{C}(1 + |m|^{p-2}) \leq \varphi_K''(m) \leq C(1 + |m|^{p-2})$;

(ii) $  \frac{1}{C}(|m|^{p-1} - 1) \leq |\varphi_K'(m)| \leq C(1 + |m|^{p-1})$;

(iii) $\frac{1}{C}(|m|^p - 1) \leq \varphi_K(m) \leq C(1 + |m|^p)$;

(iv) $ \frac{1}{C}(|m|^p - 1) \leq \psi_K(m) \leq C(|m|^p - 1)$.

Moreover, bounds (i), (ii) and (iii) are also valid for $\tilde{\varphi}$. 

\end{prop}

 \begin{proof}
We start with proving these bounds for $\varphi_K$. We only need to prove (i), since (ii) and (iii) directly follow by integrating (i) and using some uniform bounds on the value at $m = 0$, which directly follow from the boundedness of the random field.

We start with showing that 
\begin{align} \label{e_varphi_second_deriv}
\varphi_K''(m) = \left( \frac{1}{K} \hspace{1mm} s(\sigma - a_{j/N})^2 \right)^{-1},
\end{align}
where we used $\sigma = \varphi_K'(m)$ and
$$\frac{1}{K}  \underset{j = (i-1)K + 1 }{\stackrel{iK}{\sum}} \int{x \mu_{\sigma - a_{j/N}} (dx)} = m.$$

Indeed, because~$\varphi_K$ is the Legendre transform of 
\begin{align*}
 \frac{1}{K} \underset{j = (i-1)K + 1 }{\stackrel{iK}{\sum}} \hspace{1mm} \log \int_{\mathbb{R}}\exp((\sigma - a_{j/N})x - \psi(x))dx  
\end{align*}
we get that by fundamental properties of the Legendre transform 
\begin{align*}
 \varphi_K''(m) &= \left( \frac{d^2}{ds^2} \frac{1}{K} \underset{j = (i-1)K + 1 }{\stackrel{iK}{\sum}} \hspace{1mm} \log \int_{\mathbb{R}}\exp((\sigma - a_{j/N})x - \psi(x))dx \right)^{-1}.
\end{align*}
Straightforward calculation yields
\begin{align*}
\frac{d^2}{ds^2}  \frac{1}{K} \underset{j = (i-1)K + 1 }{\stackrel{iK}{\sum}} \hspace{1mm} \log \int_{\mathbb{R}}{\exp((\sigma - a_{j/N})x - \psi(x))dx}   = \frac{1}{K} \underset{j = (i-1)K + 1 }{\stackrel{iK}{\sum}} \hspace{1mm} s(\sigma - a_{j/N})^2,
\end{align*}
which satisfies the desired identity~\eqref{e_varphi_second_deriv}. 

It follows from the boundedness of the random field that
$$\frac{1}{C}s(\sigma_0) \leq s(\sigma - a_{j/N}) \leq Cs(\sigma_0)$$
where $\sigma_0$ is such that $\sum \int{x \mu_{\sigma_0}(dx)} = m$ and $C$ only depends on $\psi$ and $L = \sup_j |a_{j/N}|$ (see the proof of (\ref{uniformity}) in the appendix). Therefore
$$\frac{1}{C s(\sigma_0)^2} \leq \varphi_K''(m) \leq \frac{C}{s(\sigma_0)^2},$$
which yields statement (i) by Lemma \ref{lem_caputo} and Assumption (\ref{poly_growth}). \smallskip

Let us now turn to the verification of (i), (ii), and (iii) for the function~$\tilde{\varphi}$. We recall the definition~\eqref{average_hydro_potential} of $\tilde{\varphi}$, namely
\begin{align*}
  \tilde{\varphi}(m) = \underset{\sigma \in \R}{\sup} \hspace{1mm} \left\{ \sigma m - \mathbb{E}^a\left[\log \int_{\R}{\exp \left((\sigma - a)x - \psi(x) \right)dx}\right] \right\}.
\end{align*}
Because~$\tilde{\varphi}$ is again a Legendre transform one can use the same argument as for~$\varphi_K$ to obtain similar bounds.\smallskip
 
Let us now turn to the verifiaction of (iv). By the local Cram\'er Theorem (cf.~Theorem~\ref{lct} in the appendix) one can directly transfer the the polynomial bound of $\varphi_K$ to the function~$\psi_K$.

 \end{proof}

\subsection{Auxiliary result: Convergence of the free energy}\label{s_aux_2_conv_free_energy}

In the previous section we derived polynomial bounds on the macroscopic and hydrodynamic free energies $\psi_{N,K,i,a}$, $\varphi_{N,K,i,a}$ and $\tilde{\varphi}$ (for the definitions see~\eqref{d_psi_K}, \eqref{d_varphi_K}  and \eqref{average_hydro_potential} respectively). In this section we will show that the macroscopic free energy~$\varphi_{N,K,i,a}$ converges to the hydrodynamic free energy~$\tilde{\varphi}$.

We recall that, following (\ref{equn_bar_h_psi})., we have

\begin{equation}
\bar{H}_{N,K,a}(y) = \frac{1}{M} \underset{i = 1}{\stackrel{M}{\sum}} \hspace{1mm} \psi_{N,K,a}(y_i) + \frac{1}{N} \log \bar{Z}
\end{equation}

By Theorem \ref{lct}, we know that, for any compact subset $E$ of $\R$, we have

\begin{equation} \label{appl_lct}
\sup \left\{ |\psi_{N,K,i,a}(m) - \varphi_{N,K,i,a}(m)| \hspace{1mm} ; \hspace{3mm} m \in E, \hspace{2mm} \sup |a_q| \leq L \right\} \longrightarrow 0 \qquad \mbox{ as } N, K \to \infty.
\end{equation}

This implies that, in order to study the asymptotic behavior of $\psi_{N,K,i,a}$, we can study the behavior of $\varphi_{N,K,i,a}$.

\begin{prop} \label{conv_free_energy}
For almost every realization of the random variables $a_q$, we have for every compact subset $E$ of $\R$, and for any sequence $(i_{\ell})$ with $i_\ell \in \{1,..,M_\ell \}$,
\begin{equation} \label{convergence_free_energy_as}
\varphi_{N,K,i,a}(m) \longrightarrow \tilde{\varphi}(m)
\end{equation}
uniformly in $m \in E$. Moreover, 
\begin{equation} \label{convergence_free_energy_exp}
\mathbb{E}[\varphi_{N,K,i,a}(m)] \underset{K \uparrow \infty}{\longrightarrow} \tilde{\varphi}(m)
\end{equation}
and
\begin{equation} \label{free_energy_limit_inf}
\tilde{\varphi}(m) = \underset{K}{\inf} \hspace{1mm} \mathbb{E}[\varphi_{N,K,i,a}(m)].
\end{equation}
\end{prop}

\begin{proof}
Due to the polynomial bounds on $\varphi_K$ and $\tilde{\varphi}$, we can show that the optimum in 
$$\varphi_{N,K,i,a}(m) = \underset{\sigma}{\sup} \hspace{1mm} \left( \sigma m - \varphi_{N,K,i,a}^*(\sigma) \right) $$
is reached for $\sigma = \varphi_{N,K,i,a}'(m)$ which satisfies the bound (see statement~(iii) of Proposition~\ref{lp_bounds_energy})
$$|\sigma| \leq C(1 + |m|^{p-1})$$ 
with a constant $C$ uniform in $K$, which yields
$$|\sigma|^q \leq C(1 + |m|^{p}),$$
after taking the $q = \frac{p}{p-1}$-th power. Therefore we can localize the variational formulation as 
$$\varphi_{N,K,i,a}(m) = \underset{|\sigma|^q \leq C(1 + |m|^p)}{\sup} \hspace{1mm} \left( \sigma m - \varphi_{N,K,i,a}^*(\sigma) \right) .$$

We have the same localization property for $\tilde{\varphi}$, and therefore
\begin{align}
& |\varphi_{N,K,i,a}(m) - \tilde{\varphi}(m)| \\
& \quad = \left|\underset{|\sigma|^q \leq C(1 + |m|^p)}{\sup} \hspace{1mm} \left( \sigma m - \varphi_{N,K,i,a}^*(\sigma) \right) - \underset{|\sigma|^q \leq C(1 + |m|^p)}{\sup} \hspace{1mm}  \left( \sigma m - \tilde{\varphi}^*(\sigma) \right) \right| \notag \\ 
\end{align}.
We will now show that 
\begin{align*}
&  \left|\underset{|\sigma|^q \leq C(1 + |m|^p)}{\sup} \hspace{1mm} \left( \sigma m - \varphi_{N,K,i,a}^*(\sigma) \right) - \underset{|\sigma|^q \leq C(1 + |m|^p)}{\sup} \hspace{1mm}  \left( \sigma m - \tilde{\varphi}^*(\sigma) \right) \right| \\
& \leq \underset{|\sigma|^q \leq C(1 + |m|^p)}{\sup} \hspace{1mm} |\varphi_{N,K,i,a}^*(\sigma) - \tilde{\varphi}^*(\sigma)|.
\end{align*}
Indeed, let us assume that (the other case works with the same argument)
\begin{align*}
&  \left|\underset{|\sigma|^q \leq C(1 + |m|^p)}{\sup} \hspace{1mm} \left( \sigma m - \varphi_{N,K,i,a}^*(\sigma) \right) - \underset{|\sigma|^q \leq C(1 + |m|^p)}{\sup} \hspace{1mm}  \left( \sigma m - \tilde{\varphi}^*(\sigma) \right) \right| \\
& = \underset{|\sigma|^q \leq C(1 + |m|^p)}{\sup} \hspace{1mm} \left( \sigma m - \varphi_{N,K,i,a}^*(\sigma) \right) - \underset{|\sigma|^q \leq C(1 + |m|^p)}{\sup} \hspace{1mm}  \left( \sigma m - \tilde{\varphi}^*(\sigma) \right) .
\end{align*}
By elementary properties of the Legendre transform, there is~$\tilde \sigma$ such that
\begin{align*}
  \underset{|\sigma|^q \leq C(1 + |m|^p)}{\sup} \hspace{1mm} \left( \sigma m - \varphi_{N,K,i,a}^*(\sigma) \right) = \left( \tilde \sigma m - \varphi_{N,K,i,a}^*( \tilde \sigma) \right).
\end{align*}
Then we get
\begin{align*}
&  \underset{|\sigma|^q \leq C(1 + |m|^p)}{\sup} \hspace{1mm} \left( \sigma m - \varphi_{N,K,i,a}^*(\sigma) \right) - \underset{|\sigma|^q \leq C(1 + |m|^p)}{\sup} \hspace{1mm}  \left( \sigma m - \tilde{\varphi}^*(\sigma) \right) \\
& \leq \left( \tilde \sigma m - \varphi_{N,K,i,a}^*( \tilde \sigma) \right) -  \left( \tilde\sigma m - \tilde{\varphi}^*(\tilde \sigma) \right) \\
 & = \tilde{\varphi}^*(\tilde \sigma) - \varphi_{N,K,i,a}^*( \tilde \sigma)  \\
 & \leq \underset{|\sigma|^q \leq C(1 + |m|^p)}{\sup} \hspace{1mm} |\varphi_{N,K,i,a}^*(\sigma) - \tilde{\varphi}^*(\sigma)|.
\end{align*}

Hence we get overall that
\begin{align*}
&|\varphi_{N,K,i,a}(m) - \tilde{\varphi}(m)| \\
& = \left|\underset{|\sigma|^q \leq C(1 + |m|^p)}{\sup} \hspace{1mm} \left( \sigma m - \varphi_{N,K,i,a}^*(\sigma) \right) - \underset{|\sigma|^q \leq C(1 + |m|^p)}{\sup} \hspace{1mm}  \left( \sigma m - \tilde{\varphi}^*(\sigma) \right) \right| \\
& \leq \underset{|\sigma|^q \leq C(1 + |m|^p)}{\sup} \hspace{1mm} |\varphi_{N,K,i,a}^*(\sigma) - \tilde{\varphi}^*(\sigma)| \\
&= \underset{|\sigma|^q \leq C(1 + |m|^p)}{\sup} \hspace{1mm} \left| \frac{1}{K}\underset{j = (i-1)K + 1}{\stackrel{iK}{\sum}} \hspace{1mm} \varphi^*(\sigma - a_{j/N}) - \mathbb{E}[\varphi^*(\sigma - a)] \right|
\end{align*}
Assume now that~$|m| \leq m_0$ for some positive number~$m_0$. Then we have that
\begin{align*}
&\sup_{|m| \leq m_0} |\varphi_{N,K,i,a}(m) - \tilde{\varphi}(m)| \\
& \leq \sup_{|m| \leq m_0} \underset{|\sigma|^q \leq C(1 + |m|^p)}{\sup} \hspace{1mm} \left| \frac{1}{K}\underset{j = (i-1)K + 1}{\stackrel{iK}{\sum}} \hspace{1mm} \varphi^*(\sigma - a_{j/N}) - \mathbb{E}[\varphi^*(\sigma - a)] \right| \\
& =  \underset{|\sigma|^q \leq C(1 + |m_0|^p)}{\sup} \hspace{1mm} \left| \frac{1}{K}\underset{j = (i-1)K + 1}{\stackrel{iK}{\sum}} \hspace{1mm} \varphi^*(\sigma - a_{j/N}) - \mathbb{E}[\varphi^*(\sigma - a)] \right|.
\end{align*}
If we denote by $\nu$ is the distribution of the random variable $a$, and $\nu^{N,K,i,a}$ the (random) probability measure given by
$$\nu^{N,K,i,a}(dx) := \frac{1}{K}\underset{j = (i-1)K + 1}{\stackrel{iK}{\sum}} \hspace{1mm} \delta_{a_i},$$
we have
\begin{align*}
& \frac{1}{K} \underset{j = (i-1)K + 1}{\stackrel{iK}{\sum}} \hspace{1mm} \varphi^*(\sigma - a_{j/N}) - \mathbb{E}[\varphi^*(\sigma - a)]   \\
& = \int{\varphi^*(\sigma - x)\nu^{N,K,i,a}(dx)} - \int{\varphi^*(\sigma - x)\nu(dx)} \\
& = \int \left(  \varphi^*(\sigma - x) - \varphi^*(\sigma - y) \right) \pi (dx , d y),
\end{align*}
where $\nu^{N,K,i,a}$ is the empirical measure associated to the $a_i$ and~$\pi*(da, d \tilde a)$ is an arbitrary coupling of the measures~$\nu^{N,K,i,a}(dx)$ and~$\nu(d \tilde y)$. \medskip

 Our main argument will be that this random measure converges almost surely to $\nu$ (this is a variant of the strong law of large numbers).

The function $a \longrightarrow \varphi^*(\sigma - a)$ is $C(L)(1 + |\sigma|^{q-1})$-Lipschitz on $[-L, L]$. This follows from the fact that $(\varphi^*)'(\sigma) \leq C(1 + |\sigma|^{q-1})$, which we deduce now. Indeed, by duality of the Legendre transform we have~$(\varphi^*)'(\sigma)=m$. By part (iii) of Propositio~\ref{lp_bounds_energy}, we have that
\begin{align*}
  |\sigma | \geq \frac{1}{C} \left(|m|^{p-1} -1 \right), 
\end{align*}
which yields
\begin{align*}
  \left|(\varphi^*)'(\sigma) \right| = |m| \leq C \left( |\sigma| +1\right)^{\frac{1}{p-1}} \leq C \left( |\sigma|^{\frac{1}{p-1}} +1\right),
\end{align*}
which yields the bound $(\varphi^*)'(\sigma) \leq C(1 + |\sigma|^{q-1})$ by observing that~$q -1 = \frac{p}{p-1} - \frac{p-1}{p-1} = {\frac{1}{p-1}}$.
By this fact we can use the Kantorovitch-Rubinstein duality formula to obtain
\begin{align}
\underset{|\sigma|^q \leq C(1 + |m|^p)}{\sup} &\hspace{1mm} \left| \frac{1}{K}\underset{j = (i-1)K + 1}{\stackrel{iK}{\sum}} \hspace{1mm} \varphi^*(\sigma - a_{j/N}) - \mathbb{E}[\varphi^*(\sigma - a)] \right| \notag \\
&\leq C(1 + |m|^p)^{(q-1)/q}W_1(\nu^{N,K,i,a}, \nu) \notag \\
&= C(1 + |m|)W_1(\nu^{N,K,i,a}, \nu),
\end{align}
where~$W_1(\nu^{N,K,i,a}, \nu)$ denotes the ~$L^1$-Wasserstein distance between the measures~$\nu^{N,K,i,a}$ and~$\nu$. By Varadarajan's Theorem (see section 11.4 in [Du]), almost surely $W_1(\nu^K, \nu) \longrightarrow 0$, and therefore almost surely, $\varphi_{N,K,i,a}(m)$ converges to $\tilde{\varphi}(m)$ uniformly on bounded sets. This proves (\ref{convergence_free_energy_as}). \medskip

Since the $a_q$ are bounded random variables, $\varphi_{N,K,i,a}(m)$ is also bounded, and an application of the Dominated Convergence Theorem yields (\ref{convergence_free_energy_exp}). \medskip

Finally, we prove~(\ref{free_energy_limit_inf}): since the expectation of $\varphi_{N,K,i,a}(m)$ only depends on $K$, we drop the subscript and consider
$$\varphi_K(m) = \underset{\sigma \in \R}{\sup} \left(\sigma m - \frac{1}{K} \underset{j = 1}{\stackrel{K}{\sum}} \varphi^*(\sigma - a_j) \right),$$
with the $a_i$ a sequence of iid random variables. We then have, for any $K_1, K_2 \in \mathbb{N}$
\begin{align}
(K_1 + K_2)\varphi_{K_1 + K_2}(m) &= (K_1 + K_2)\underset{\sigma \in \R}{\sup} \left(\sigma m - \frac{1}{K_1 + K_2} \underset{j = 1}{\stackrel{K_1 + K_2}{\sum}} \varphi^*(\sigma - a_j) \right) \notag \\
&= \underset{\sigma \in \R}{\sup} \left((K_1 + K_2)\sigma m - \underset{j = 1}{\stackrel{K_1 + K_2}{\sum}} \varphi^*(\sigma - a_j) \right) \notag \\
&\leq \underset{\sigma \in \R}{\sup} \left(K_1\sigma m - \underset{j = 1}{\stackrel{K_1}{\sum}} \varphi^*(\sigma - a_j) \right) \notag \\
& \hspace{4mm} + \underset{\sigma \in \R}{\sup} \left(K_2\sigma m - \underset{j = K_1 + 1}{\stackrel{K_1 + K_2}{\sum}} \varphi^*(\sigma - a_j) \right) \notag 
\end{align}
So that
$$(K_1 + K_2) \mathbb{E}[\varphi_{K_1 + K_2}(m)] \leq K_1\mathbb{E}[\varphi_{K_1}(m)] + K_2\mathbb{E}[\varphi_{K_2}(m)].$$
Therefore $K \hspace{1mm} \mathbb{E}[\varphi_{K}(m)]$ is a sub-additive sequence, and an application of the sub-additivity theorem yields
$$\mathbb{E}[\varphi_{K}(m)] \longrightarrow \underset{K}{\inf} \hspace{1mm} \mathbb{E}[\varphi_{K}(m)].$$
Since $\mathbb{E}[\varphi_{K}(m)]$ also converges to $\tilde{\varphi}(m)$, this implies (\ref{free_energy_limit_inf}).

\end{proof}

We will also need convergence of the free energy when taking the average along some $L^p$ function : 

\begin{lem} \label{conv_energy_lp} Let $c(K) := \mathbb{E}[W_1(\nu^{N,K,i,a}, \nu)^q]$, which only depends on $K$, and assume that
\begin{equation} \label{assumption_k_borell_cantelli}
\underset{\ell}{\sum} \hspace{1mm} c(K_{\ell}) \hspace{3mm} < \infty
\end{equation}
for some $\epsilon > 0$. Then, almost surely, for any sequence of adapted step functions $\bar{\zeta}_{\ell}$ that is bounded in $L^p$, we have
$$\underset{\ell}{\lim} \hspace{1mm} \frac{1}{M}\underset{i = 1}{\stackrel{M}{\sum}} \hspace{1mm} |\varphi_{N,K,i,a}(\zeta_{\ell, i}) - \tilde{\varphi}(\zeta_{\ell, i})| = 0.$$
(By adapted step function, we mean that $\bar{\zeta}_{\ell}$ is the step function associated to an element of $Y_{\ell}$, ie that its mesh size is $1/M_{\ell}$.)
\end{lem}

\begin{rmq} \label{remark_cK}
Note that, since $W_1(\nu^{N,K,i,a}, \nu)$ almost surely converges to zero and is a bounded random variable (since the $a_q$ are bounded), by the Dominated Convergence Theorem, $c(K) \longrightarrow 0$ as $K$ goes to infinity.
\end{rmq}

\begin{proof}
With the same arguments as in the proof of Proposition \ref{conv_free_energy}, we have
\begin{equation}
\frac{1}{M}\underset{i = 1}{\stackrel{M}{\sum}} \hspace{1mm} |\varphi_{N,K,i,a}(\zeta_{\ell, i}) - \tilde{\varphi}(\zeta_{\ell, i})| \leq \frac{1}{M}\underset{i = 1}{\stackrel{M}{\sum}} \hspace{1mm} C(1 + |\zeta_{\ell, i}|)W_1(\nu^{N,K,i,a}, \nu).
\end{equation}
Using Holder's inequality, this yields
\begin{align}
\frac{1}{M}\underset{i = 1}{\stackrel{M}{\sum}} &\hspace{1mm} |\varphi_{N,K,i,a}(\zeta_{\ell, i}) - \tilde{\varphi}(\zeta_{\ell, i})| \notag \\
&\leq \left(\frac{1}{M}\underset{i = 1}{\stackrel{M}{\sum}} \hspace{1mm} C(1 + |\zeta_{\ell, i}|)^p \right)^{1/p}\left(\frac{1}{M}\underset{i = 1}{\stackrel{M}{\sum}} \hspace{1mm} W_1(\nu^{N,K,i,a}, \nu)^q \right)^{1/q} \notag \\
&\leq C(1 + ||\bar{\zeta}_{\ell}||^p_p)\left(\frac{1}{M}\underset{i = 1}{\stackrel{M}{\sum}} \hspace{1mm} W_1(\nu^{N,K,i,a}, \nu)^q \right)^{1/q}.
\end{align}
Therefore, all we need to show is that $\frac{1}{M}\underset{i = 1}{\stackrel{M}{\sum}} \hspace{1mm} W_1(\nu^{N,K,i,a}, \nu)^q$ converes almost surely to zero. By Markov's inequality, for any $\epsilon > 0$, we have
\begin{align}
\mathbb{P}\left[\frac{1}{M} \underset{i = 1}{\stackrel{M}{\sum}} \hspace{1mm} W_1(\nu^{N,K,i,a}, \nu)^q > \epsilon \right] & \leq \frac{1}{\epsilon}\mathbb{E}\left[\frac{1}{M}\underset{i = 1}{\stackrel{M}{\sum}} \hspace{1mm} W_1(\nu^{N,K,i,a}, \nu)^q \right] \notag \\
&= \frac{c(K)}{\epsilon} \notag
\end{align}
Applying the Borell-Cantelli Lemma with assumption \eqref{assumption_k_borell_cantelli} then yields the result.
\end{proof}

\subsection{Proof of Proposition \ref{convergence_macro_hydro}}\label{s_proof_of_conv_ODE_PDE}

Our proof will follow the same structure as the one in [GOVW]. The main differences are the use of Proposition \ref{conv_free_energy}, which is needed because of the additional random linear term in the Hamiltonian, and the use of $L^p$ bounds on the data (rather than $L^2$). They mostly appear in the proof of Lemma 5.14. The proof will rely on the following five lemmas.

\begin{lem} \label{bounds_sol_macro}
Consider the sequence $(\eta^{\ell})$ of solutions of (\ref{eqn_macro_seq}), subject to (\ref{bound_eta}). There exists $C < \infty$ (independent of $\ell$) such that
\begin{equation} \label{bnd_macro_unif}
\underset{0 \leq t \leq T}{\sup} \hspace{1mm} ||\eta^{\ell}(t)||_{L^p} \leq C
\end{equation}

\begin{equation} \label{bound_der_macro}
\int_0^T{\left \langle \frac{d\eta^{\ell}}{dt}(t), \bar{A}^{-1}\frac{d\eta^{\ell}}{dt}(t) \right \rangle_Y \hspace{1mm} dt} \leq C
\end{equation}
\end{lem}

In particular, (\ref{bnd_macro_unif}) implies that the $\eta^{\ell}$ are uniformly bounded in $L^{\infty}(L^p)$, and therefore that, for the sequence of associated step functions $\bar{\eta}^{\ell}$, there is a subsequence such that
$$\bar{\eta}^{\ell} \longrightarrow \eta_* \hspace{3mm} \text{weak}-\ast \text{ in } \hspace{2mm} L^{\infty}(L^p) = (L^1(L^q))^*$$
for some limit $\eta_*$. We will now state several lemmas, which we will use to show that $\eta_*$ must be the unique weak solution of (\ref{hydro_equn}).

\begin{lem} \label{inclusions}
Any limit $\eta_*$ of $\eta^{\ell}$ for the weak-* topology on $L^{\infty}(L^p)$, we have
\begin{equation}
\eta_* \in L^{\infty}(L^p), \hspace{5mm} \frac{\partial \eta_*}{\partial t} \in L^2(H^{-1}), \hspace{5mm} \tilde{\varphi}'(\eta_*) \in L^{\infty}(L^q).
\end{equation}
\end{lem}

\begin{lem} [Inequality formulation for convex potential] 
\label{lem_diff_inequality_convex_potential}
Assume $\bar{H}_{N,K,a}$ is convex. Then $\eta$ is a solution of (\ref{eqn_macro_seq}) iff for all $\xi \in Y_M$ and smooth $\beta : [0,T] \rightarrow \mathbb{R}_+$ with $\beta(0) = \beta(T) =0$, we have 
\begin{equation} \label{ineq_macro}
\int_0^T{\bar{H}_{N,K,a}(\eta)\beta(t)dt} \leq \int_0^T{\bar{H}_{N,K,a}(\eta + \xi)\beta(t)dt} - \int_0^T{\langle \xi, (\bar{A})^{-1}\eta \rangle_Y\dot{\beta}(t)dt}.
\end{equation}

Similarly, $\zeta$ satisfies (\ref{hydro_equn}) if and only if 
\begin{align} \label{ineq_hydro}
&\int_0^T{\int_{\mathbb{T}}{\tilde{\varphi}(\zeta(t,\theta))\beta(t)d\theta}dt} \notag \\
&\leq \int_0^T{\int_{\mathbb{T}}{\tilde{\varphi}(\zeta(t,\theta) + \xi(\theta))\beta(t)d\theta}dt} - \int_0^T{\langle \xi(\cdot), \zeta(t,\cdot) \rangle_{\textsl{H}^{-1}}\dot{\beta}(t)dt}
\end{align}
for all $\xi \in \textsl{L}^p(\mathbb{T})$ and smooth $\beta : [0,T] \rightarrow [0, \infty)$ that takes value $0$ at $0$ and $T$.
\end{lem}

\begin{lem} \label{lem_convergence} 
Suppose that the sequence $\{\eta_{\ell}\}$ satisfies (\ref{hydro_equn}), (\ref{bnd_macro_unif}) and (\ref{bound_der_macro}) and consider a subsequence such that the associated step functions weak-* converge in $(\textsl{L}^1(\textsl{L}^q))^*$ to a limit $\eta_*$. Moreover, assume that (\ref{assumption_k_borell_cantelli}) holds.

Let $\xi_{\ell} := \pi_{\ell}(\xi + \eta_*) - \eta_{\ell}$, where $\xi$ is an arbitrary $\textsl{L}^p$ function and $\pi_{\ell}$ is the $\textsl{L}^p$ projection onto elements of $Y_{\ell}$. Then we have for almost every realization of the random field $a$

(i) $$\underset{\ell \uparrow \infty}{\liminf} \int_0^T{(\bar{H}_{N,K,a}(\eta_{\ell}(t)) - \frac{1}{N}\log \bar{Z}_{N,K,a})\beta(t)dt} \geq \int_0^T{\beta(t)\int_{\mathbb{T}}{\tilde{\varphi}(\eta_*(t,\theta))d\theta}dt};$$

(ii) $\underset{\ell \uparrow \infty}{\limsup} \int_0^T{(\bar{H}_{N,K,a}(\eta_{\ell}(t) + \xi_{\ell}(t)) - \frac{1}{N}\log \bar{Z}_{N,K,a})\beta(t)dt}$

$\hspace{3cm} \leq \int_0^T{\beta(t)\int_{\mathbb{T}}{\tilde{\varphi}(\eta_*(t,\theta) + \xi(\theta))d\theta}dt};$

(iii) \begin{equation}
\underset{\ell \uparrow \infty}{\lim} \int_0^T{\langle \xi_{\ell}(t), (\bar{A})^{-1}\eta_{\ell}(t)\rangle_Y \dot{\beta}(t)dt} = \int_0^T{\langle \xi(\cdot), \eta_*(t, \cdot)\rangle_{\textsl{H}^{-1}} \dot{\beta}(t)dt}.
\end{equation}
\end{lem}

\begin{lem} \label{unicity}
Equation (\ref{hydro_equn}) has at most one weak solution with initial condition~$\zeta_0$.
\end{lem}

Using these, we can prove Proposition \ref{convergence_macro_hydro}. The main difference with [GOVW] will be the proof of Lemma \ref{lem_convergence}.

\begin{proof}[Proof of Proposition \ref{convergence_macro_hydro}]
As stated earlier, by applying Lemma \ref{bounds_sol_macro}, we can extract a subsequence such that
$$\bar{\eta}^{\ell} \longrightarrow \eta_*  \text{   weak-* in   } L^{\infty}(L^p) = (L^1(L^q))^*.$$
We want to use Lemma \ref{lem_convergence}, so we need to check  (\ref{assumption_k_borell_cantelli}). But we already know that $c(K_{\ell}) \longrightarrow 0$ (see Remark \ref{remark_cK}), so we can further extract a subsequence such that (\ref{assumption_k_borell_cantelli}) is satisfied, without changing the limit. Since all we care about here is to identify the limit $\eta_*$, there is no loss of generality here.

According to Lemma \ref{lem_diff_inequality_convex_potential}, $\eta^{\ell}$ satisfies (\ref{ineq_macro}). By applying Lemma \ref{lem_convergence} to this inequality, we get that $\eta_*$ satisfies (\ref{ineq_hydro}), so that another application of Lemma \ref{lem_diff_inequality_convex_potential} and of Lemma \ref{inclusions} yields that $\eta_*$ is a solution of the hydrodynamic equation. The unicity result of Lemma \ref{unicity} then guarantees that the full sequence converges to the unique weak solution of (\ref{hydro_equn}).
\end{proof}

We will now prove the previous Lemmas.

\begin{proof} [Proof of Lemma \ref{bounds_sol_macro}]
The proof is exactly the same as in Lemma 34 of [GOVW]. \newline
Argument for~(\ref{bnd_macro_unif}): By (\ref{energy2}) it holds 
\begin{align}
  \bar{H}_K(\eta^{\ell}(t))  \leq \bar{H}_K(\eta^{\ell}(0)) \overset{\eqref{equn_bar_h_psi}}{=} \frac{1}{M}\underset{i = 1}{\stackrel{M}{\sum}} \hspace{1mm}  \psi_{K,a_{(i-1)K+1},..a_{iK}}(y_i) - \frac{1}{N}\log \bar{Z}. \label{e_calc_bnd_macro_unif} 
\end{align}
Now, the desired bound~(\ref{bnd_macro_unif}) follows from a combination of the polynomial bounds on $\psi_{K,a_{(i-1)K+1},..a_{iK}}$ of Proposition \ref{lp_bounds_energy} and~\eqref{bound_eta}. \medskip

Argument for~(\ref{bound_der_macro}): Using By (\ref{energy2}) yields
\begin{align*}
  \int_0^T{\left \langle \frac{d\eta^{\ell}}{dt}(t), \bar{A}^{-1}\frac{d\eta^{\ell}}{dt}(t) \right \rangle_Y \hspace{1mm} dt} =  \bar{H}_K(\eta^{\ell}(0)) - \bar{H}_K(\eta^{\ell}(t)).  
\end{align*}
The first term on the right hand side is bounded as in~\eqref{e_calc_bnd_macro_unif}. To bound the second term, observe that
\begin{align*}
  \inf_y \bar{H}_K(y) \geq -C,
\end{align*}
unifromly, which is a consequence of the uniform convergence of~$\psi_K$ to $\varphi_K$, which is a strictly convex function.
\end{proof}

\begin{proof}[Proof of Lemma \ref{inclusions}]
By weak lower semicontinuity, and recalling (\ref{bnd_macro_unif}) we have for all $t \in [0,T]$ that 
$$||\eta_*(t,\cdot)||_{L^p} \leq \liminf ||\eta^{\ell}(t)||_{L^p} \leq C.$$
The first estimate is therefore proved. The third one immediately follows, since $\tilde{\varphi}'$ satisfies the polynomial bound $\tilde{\varphi}'(m) \leq C(1 + |m|^{p-1})$ and $q = p/(p-1)$. For the second estimate, we use Lemma \ref{compare_norms}, and, recalling that $\overline{NP^t\eta^{\ell}} = \bar{\eta}^{\ell}$, we have
\begin{align}
\int_0^T{\langle \dot{\bar{\eta}}^{\ell}, \dot{\bar{\eta}}^{\ell} \rangle_{H^{-1}} \hspace{1mm} dt} &\leq \frac{C}{N}\int_0^T{\langle NP^t\dot{\eta}^{\ell}, A^{-1}NP^t\dot{\eta}^{\ell} \rangle_X \hspace{1mm} dt} \notag \\
&= C\int_0^T{\langle \dot{\eta}^{\ell}, \bar{A}^{-1}\dot{\eta}^{\ell} \rangle_Y \hspace{1mm} dt} \leq C \notag
\end{align}
where the last inequality used (\ref{bound_der_macro}). Once more, by weak lower semicontinuity, we then have
$$\int_0^T{\langle \dot{\eta_*}, \dot{\eta_*} \rangle_{H^{-1}} \hspace{1mm} dt}  \leq C.$$
\end{proof}

\begin{proof}[Proof of Lemma \ref{unicity}]
Once more, this is exactly the same as in Lemma 38 of [GOVW]. Consider two solutions $\zeta_1$ and $\zeta_2$ with same initial condition. We have the weak formulation
$$\langle \dot{\zeta}_i, \xi \rangle_{H^{-1}} = -\int_{\T}{\tilde{\varphi}'(\zeta_i)\xi d\theta} \text{ for all } \xi \in L^2, \text{ for a.e. } t \in [0,T],$$
so that
$$\langle (\dot{\zeta}_1 - \dot{\zeta}_2), \xi \rangle_{H^{-1}} = -\int_{\T}{(\tilde{\varphi}'(\zeta_1)-\tilde{\varphi}'(\zeta_2))\xi d\theta}.$$
Since both $\zeta_1$ and $\zeta_2$ lie in $L^{\infty}(L^p)$, we can take $\zeta_1 - \zeta_2$ as a test function, so that for a.e. $t \in [0,T]$ we have
$$\frac{d}{dt} \langle (\zeta_1 - \zeta_2), (\zeta_1 - \zeta_2) \rangle_{H^{-1}} = -2\int_{\T}{(\tilde{\varphi}'(\zeta_1)-\tilde{\varphi}'(\zeta_2))(\zeta_1 - \zeta_2) d\theta} \leq 0$$
by convexity of $\tilde{\varphi}$. This implies that $\zeta_1 = \zeta_2$.
\end{proof}

\begin{proof} [Proof of Lemma \ref{lem_diff_inequality_convex_potential}]
Both parts of this lemma are proved in the same way, so we shall only prove (\ref{ineq_hydro}). Let $\zeta$ be a bounded solution of (\ref{hydro_equn}). First let us rewrite the equation as 
\begin{equation} \label{equ_weak}
\int_0^T{\langle \xi(\cdot), \zeta(t,\cdot) \rangle_{H^{-1}} \dot{\beta}(t) dt} = \int_0^T{\int_{\T}{\varphi'(\zeta(t,\theta))\xi(\theta) d\theta} \beta(t)dt}
\end{equation}
for all $\xi \in L^2$ and smooth $\beta : [0,T] \rightarrow \R_+$. By density of $L^p$ in $L^2$, it is enough if (\ref{equ_weak}) holds for all $\xi \in L^p$. 

Let us show that (\ref{equ_weak}) implies (\ref{ineq_hydro}). By convexity of $\varphi$, we have
$$\xi(\theta)\varphi'(\eta_*(\theta)) \leq -\varphi(\eta_*(\theta)) + \varphi(\eta_*(\theta) + \xi(\theta)).$$
Inserting this inequality into (\ref{equ_weak}) gives
$$\int_0^T{\langle \xi(\cdot), \zeta(t,\cdot) \rangle_{H^{-1}} \dot{\beta}(t) dt} \leq \int_0^T{\int_{\T}{\varphi(\zeta(t,\theta) + \xi(\theta))d\theta} \beta(t)dt} - \int_0^T{\int_{\T}{\varphi(\zeta(t,\theta))d\theta} \beta(t)dt}$$
which, after rearranging terms, is (\ref{ineq_hydro}).

Let us now show that (\ref{ineq_hydro}) implies (\ref{equ_weak}). We substitute $\xi$ for $\epsilon \xi$ in (\ref{ineq_hydro}), for some $\epsilon > 0$. Dividing both sides by $\epsilon$ and rearranging then yields
$$\int_0^T{\langle \xi(\cdot), \zeta(t,\cdot) \rangle_{H^{-1}} \dot{\beta}(t) dt} \leq \int_0^T{ \int_{\T}{\frac{\varphi(\zeta + \epsilon \xi) - \varphi(\zeta)}{\epsilon} \beta(t) d\theta} dt}.$$
Taking the limit $\epsilon \rightarrow 0$ then gives us
$$\int_0^T{\langle \xi(\cdot), \zeta(t,\cdot) \rangle_{H^{-1}} \dot{\beta}(t) dt} \leq \int_0^T{\int_{\T}{\varphi'(\zeta(t,\theta))\xi(\theta) d\theta} \beta(t)dt}.$$
Repeating this process with $-\epsilon \xi$ instead of $\epsilon \xi$ returns
$$\int_0^T{\langle \xi(\cdot), \zeta(t,\cdot) \rangle_{H^{-1}} \dot{\beta}(t) dt} \geq \int_0^T{\int_{\T}{\varphi'(\zeta(t,\theta))\xi(\theta) d\theta} \beta(t)dt},$$
so that (\ref{equ_weak}) holds.
\end{proof}

\begin{proof} [Proof of Lemma \ref{lem_convergence}]

Since the proof of (iii) is exactly the same as the proof of (iii) in Lemma 37 of [GOVW], we will skip its proof, and concentrate on (i) and (ii).

In order to prove (i), it is enough to show that
\begin{equation} \label{407}
\liminf \frac{1}{M}\underset{i = 1}{\stackrel{M}{\sum}} \hspace{1mm} \varphi_{N,K,a,i}(\eta_i(t)) \geq \int_{\mathbb{T}}{\tilde{\varphi}(\eta_*(t,\theta))d\theta},
\end{equation}
because it already follows from Theorem \ref{lct} that, for a sequence that is bounded in $L^p$, we can replace $\psi_{N,K,a,i}$ with $\varphi_{N,K,a,i}$, and since an application of Fatou's Lemma would then immediately yield (i). Since we have
\begin{align}
\frac{1}{M}\sum \varphi_{N,K,i,a}(\eta_i(t)) &= \frac{1}{M}\underset{i = 1}{\stackrel{M}{\sum}} \underset{\sigma \in \mathbb{R}}{\sup} \left(\sigma \eta_i - \frac{1}{K}\underset{j = (i-1)K + 1}{\stackrel{iK}{\sum}} \varphi^*(\sigma - a_{i/N}) \right) \notag \\
&\geq \underset{\sigma \in \mathbb{R}}{\sup} \left(\sigma \frac{1}{M}\underset{i = 1}{\stackrel{M}{\sum}} \eta_i - \frac{1}{N}\underset{j = 1}{\stackrel{N}{\sum}} \varphi^*(\sigma - a_{i/N}) \right) \notag \\
&= \varphi_{N,a}\left(\frac{1}{M}\underset{i = 1}{\stackrel{M}{\sum}} \eta_i \right) 
 \longrightarrow \tilde{\varphi}\left(\int_{\mathbb{T}}{\eta_*(\theta)d\theta}\right)
\end{align}
almost surely, we have 
$$\liminf \frac{1}{M}\underset{i = 1}{\stackrel{M}{\sum}} \varphi_{N,K,i,a}(\eta_i(t))  \geq \tilde{\varphi}\left(\int_{\mathbb{T}}{\eta_*(t,\theta)d\theta}\right).$$
But, in exactly the same way, we almost surely have
$$\liminf \frac{1}{M/2}\underset{i = 1}{\stackrel{M/2}{\sum}} \varphi_{N,K,i,a}(\eta_i(t)) \geq \tilde{\varphi}\left(\int_0^{1/2}{\eta_*(t,\theta)d\theta}\right)$$
and 
$$\liminf \frac{1}{M/2}\underset{i = M/2}{\stackrel{M}{\sum}} \varphi_{N,K,a,i}(\eta_i(t)) \geq \tilde{\varphi}\left(\int_{1/2}^1{\eta_*(t,\theta)d\theta}\right),$$
so that
$$\liminf \frac{1}{M}\underset{i = 1}{\stackrel{M}{\sum}} \varphi_{N,K,i,a}(\eta_i(t))  \geq \frac{1}{2} \tilde{\varphi}\left(\int_0^{1/2}{\eta_*(t,\theta)d\theta}\right) + \frac{1}{2} \tilde{\varphi}\left(\int_{1/2}^1{\eta_*(t,\theta)d\theta}\right).$$
Iterating this argument gives us, for any $L \in \mathbb{N}$, 
\begin{equation}
\liminf \frac{1}{M}\underset{i = 1}{\stackrel{M}{\sum}} \varphi_{N,K,i,a}(\eta_i(t))  \geq \frac{1}{2^L} \underset{j = 1}{\stackrel{2^L}{\sum}} \tilde{\varphi}\left(\int_{(j-1)/2^L}^{j/2^L}{\eta_*(\theta)d\theta}\right)
\end{equation}
almost surely. Letting $L$ go to infinity, since $\eta_*$ belongs to $L^p$, we get (\ref{407}), and therefore (i).

Let us now prove (ii). We fix $\xi \in L^p$. 
We have
\begin{align} \label{eq101}
&\left| \int_0^T{(\bar{H}_{N,K,a}(\eta_{\ell}(t) + \xi_{\ell}(t)) - \frac{1}{N}\log \bar{Z}_{N,K,a})\beta(t)dt} - \int_0^T{\beta(t)\int_{\mathbb{T}}{\tilde{\varphi}(\eta_*(t,\theta) + \xi(\theta))d\theta}dt} \right| \notag \\
&= \left| \int_0^T{\left(\underset{i = 1}{\stackrel{M_{\ell}}{\sum}} \hspace{1mm} \psi_{N,K,i,a}(\eta^{\ell}_i + \xi^{\ell}_i) \right)\beta(t)dt} - \int_0^T{\beta(t)\int_{\mathbb{T}}{\tilde{\varphi}(\eta_*(t,\theta) + \xi(\theta))d\theta}dt} \right| \notag \\
&\leq \int_0^T{\beta(t)\left(\underset{i = 1}{\stackrel{M_{\ell}}{\sum}} \hspace{1mm} |\psi_{N,K,i,a}(\eta^{\ell}_i + \xi^{\ell}_i) - \varphi_{N,K,i,a}(\eta^{\ell}_i + \xi^{\ell}_i)| \right)dt} \notag \\
& \hspace{5mm} + \int_0^T{\beta(t)\left|\left(\underset{i = 1}{\stackrel{M_{\ell}}{\sum}} \hspace{1mm} \varphi_{N,K,i,a}(\eta^{\ell}_i + \xi^{\ell}_i) - \tilde{\varphi}(\eta^{\ell}_i + \xi^{\ell}_i) \right)\right|dt} \notag \\
& \hspace{1cm} + \left|\int_0^T{\int_{\T}{\tilde{\varphi}(\bar{\xi}^{\ell} + \bar{\eta}^{\ell})\beta(t)d\theta} dt} - \int_0^T{\int_{\T}{\tilde{\varphi}(\eta_* + \xi)\beta(t)d\theta}dt} \right|.
\end{align}

The first term on the right-hand side is controlled by Theorem~\ref{lct}, since the sequence $\eta^{\ell}_i + \xi^{\ell}_i$ is bounded in $L^p$. The second term goes to $0$ as a consequence of Lemma \ref{conv_energy_lp}, which can be used since we assumed (\ref{assumption_k_borell_cantelli}) holds.

Finally, since we have $\eta^{\ell} + \xi^{\ell} = \pi_{\ell}(\xi + \eta_*)$, $\bar{\xi}^{\ell} + \bar{\eta}^{\ell}$ converges to $\xi + \eta_*$ strongly in $L^p$. From the polynomial upper bound on $\tilde{\varphi}$, $\rho \longrightarrow \int{\tilde{\varphi}(\rho)d\theta}$ is continuous with respect to strong $L^p$ convergence, and therefore
\begin{equation}
\int_{\T}{\tilde{\varphi}(\bar{\xi}^{\ell} + \bar{\eta}^{\ell})d\theta} \longrightarrow \int_{\T}{\tilde{\varphi}(\eta_* + \xi)d\theta}.
\end{equation}
By the Dominated Convergence Theorem, it follows that
\begin{equation}
\int_0^T{\int_{\T}{\tilde{\varphi}(\bar{\xi}^{\ell} + \bar{\eta}^{\ell})\beta(t)d\theta} dt} \longrightarrow \int_0^T{\int_{\T}{\tilde{\varphi}(\eta_* + \xi)\beta(t)d\theta}dt},
\end{equation}
so the third term on the right-hand side of (\ref{eq101}) goes to zero.

\end{proof}


\appendix

\section{Proof of the local Cram\'er theorem}

In this section, we give the proof of the local Cram\'er theorem we used to obtain convexity of the coarse-grained Hamiltonian (cf.~Proposition~\ref{p_convexity_coarse_grained_Hamiltonian}). It is a combination of the proofs of the local Cram\'er theorem of [GOVW] and [MO], modified to take into account the additional inhomogeneous linear term in the Hamiltonian.

\begin{thm} \label{lct}
Let
$$\psi_K(m) := -\frac{1}{K}\log \int_{X_{K,m}}{\exp \left(-\sum a_i x_i + \psi(x_i) \right)dx}$$
and 
$$\varphi_{K}(m) := \underset{\sigma \in \mathbb{R}}{\sup} \hspace{1mm} \left(\sigma m - \frac{1}{K} \underset{i = 1}{\stackrel{K}{\sum}} \hspace{1mm} \log \int_{\mathbb{R}}{\exp((\sigma - a_i)x - \psi(x))dx} \right).$$
There exists $K_0 \in \N$, $C > 0$ (which only depends on $\psi$, $K_0$ and $\sup |a_i|$) such that for any $m$
$$|\psi_K(m) - \varphi_K(m)| \leq \frac{C}{K} | \varphi_K''(m)|^{\frac{1}{2}}.$$
In particular, for any $L > 0$ and any compact subset E of $\mathbb{R}$,
$$\underset{K \rightarrow \infty}{\lim} \hspace{1mm} \underset{a_1 .. a_K \in [-L, L]}{\sup}  \hspace{3mm} ||\psi_K - \varphi_K||_{\infty, E}  = 0.$$
Moreover, there exists $\lambda > 0$ and $K_0 \in \mathbb{N}$ such that, for any $K \geq K_0$ and any collection of real numbers $a_1$ ... $a_K$ $\in [-L, L]$, we have 
\begin{equation} \label{bnd_derivative_psi}
\psi_K'' \geq \lambda.
\end{equation}
\end{thm}

The starting point of the proof is the explicit representation 
\begin{equation}
g_{K,m}(0) = \exp(K\varphi_K(m) - K\psi_K(m)),
\end{equation}
where $g_{K,m}$ is the Lebesgue density of the random variable 
$$\frac{1}{\sqrt{K}} \underset{i = 1}{\stackrel{K}{\sum}} \hspace{1mm} X_i - m_i,$$
and the $X_i$ are independent random variables, distributed as 
$$\mu_{\sigma,i}(dx) := \exp(-\varphi^*(\sigma - a_i) + (\sigma - a_i)x - \psi(x))dx,$$
with $m_i$ the mean of $X_i$.

Since $\varphi_K$ is the Legendre transform of the strictly convex function 

$$\varphi_K^*(\sigma) := \frac{1}{K}\underset{i = 1}{\stackrel{K}{\sum}} \hspace{1mm} \log \int_{\R}{\exp((\sigma - a_i)x - \psi(x))dx},$$
there exists for any $m \in \R$ a unique $\sigma = \sigma(m)$ such that
$$\varphi_K(m) = \sigma m - \varphi^*_K(\sigma).$$

It is a basic property of the Legendre transform that $\sigma$ is determined by
\begin{equation}
\frac{d}{d\sigma}\varphi^*_K(\sigma) = \frac{1}{K}\underset{i = 1}{\stackrel{K}{\sum}} \hspace{1mm} \frac{\int{x\mu_{\sigma,i}(dx)}}{\int{\mu_{\sigma,i}(dx)}} = \frac{1}{K}\underset{i = 1}{\stackrel{K}{\sum}} \hspace{1mm} m_i = m.
\end{equation}

Since the derivative of the Legendre transform of a convex function is the inverse of the derivative of the function, we also have

\begin{equation}
\sigma(m) = \varphi_K'(m).
\end{equation}

\begin{prop} \label{bounds_lct}
Let $\psi : \mathbb{R} \rightarrow \mathbb{R}$ be a bounded perturbation of a uniformly convex function. Let us define
\begin{align}
& \mu_{\sigma}(dx) := \exp(-\varphi^*(\sigma) + \sigma x - \psi(x))dx ; \label{e_d_mu_sigma}  \\ 
&  m(\sigma) := \int_{\mathbb{R}}{x\mu_{\sigma}(dx)} \hspace{1mm} ; \hspace{2mm} s(\sigma)^2 :=  \int_{\mathbb{R}}{(x - m)^2\mu_{\sigma}(dx)}. \label{e_d_m_sigma}
\end{align}
We assume that the following bounds hold uniformly in $\sigma$ and $j$ : 
\begin{equation} \label{bor}
\int_{\mathbb{R}}{|x - m(\sigma)|^k  \mu_{\sigma}(dx)} \leq C s^k \hspace{1mm} \text{for} \hspace{1mm} k= 1..5 ;
\end{equation}
\begin{equation} \label{f1}
\left|\int_{\mathbb{R}}{\exp(i(x - m)\xi)\mu_{\sigma}(dx)}\right| \leq \frac{C}{s|\xi|}  \hspace{1mm} \text{for all} \hspace{1mm} \xi \in \mathbb{R}.
\end{equation}
Moreover, we assume that 
\begin{equation} \label{uniformity}
\underset{\sigma \in \R}{\sup} \underset{a, b \in [-L,L]}{\sup} \frac{s(\sigma - a)}{s(\sigma - b)} \lesssim 1.
\end{equation}
Then we have
\begin{equation} \label{lct_bnd1}
\left|g_{K,m}(0) - \left(\frac{2\pi}{K}\sum s(\sigma - \lambda_i)^2 \right)^{-1/2} \right| \lesssim \frac{1}{\sqrt{K}}\left(\frac{1}{K}\sum s(\sigma - \lambda_i)^2 \right)^{-1/2};
\end{equation}
\begin{equation} \label{lct_bnd3}
\left| \frac{d}{d\sigma}g_{K,m}(0) \right| \lesssim 1
\end{equation}
\begin{equation} \label{lct_bnd2}
\left| \frac{d^2}{d\sigma^2}g_{K,m}(0) \right| \lesssim \left(\frac{1}{K}\underset{i = 1}{\stackrel{K}{\sum}} s(\sigma - a_i)^2 \right)^{1/2}.
\end{equation}
\end{prop}

Using this result, we can prove our local Cramer theorem.

\begin{proof}[Proof of Theorem \ref{lct}]

Let us first check that the Assumptions of Proposition \ref{bounds_lct} are verified. It has been shown in [MO, Lemma 3.2] that (\ref{bor}) and (\ref{f1}) hold when $\psi$ is a bounded perturbation of a uniformly convex potential. We refer to [MO] for a proof, which is based on simple general bounds on log-concave probability densities.

To show that (\ref{uniformity}) is satisfied, we will show that $\sigma \rightarrow \log s(\sigma)$ is Lipschitz continuous, by showing that its derivative is bounded. We have
$$\frac{d}{d\sigma} \log s(\sigma) = \frac{1}{s}\frac{d}{d\sigma}s =\frac{1}{2s^2}\frac{d}{d\sigma}s^2, $$
and we will show later (see (\ref{derivative_s2})) that 
$$\frac{d}{d\sigma}s(\sigma)^2 = \int{(x - m)^3\mu_{\sigma}(dx)},$$
so that, by (\ref{bor}), we have
$$\left| \frac{d}{d\sigma} \log s(\sigma) \right| \leq Cs(\sigma).$$
Since $\sigma \rightarrow s(\sigma)$ is bounded above (see Lemma \ref{lem2}), this quantity is bounded independently of $\sigma$. Therefore $\log s$ is Lipschitz-continuous, and (\ref{uniformity}) immediately follows.

We will now apply Proposition \ref{bounds_lct} to prove our theorem. First, we have 
$$\varphi_{K}(m) - \psi_K(m) = \frac{1}{K}\log g_{K,m}(0).$$
The first part of the theorem then immediately follows from Proposition \ref{bounds_lct} and the fact that (cf.~\eqref{e_varphi_second_deriv})
$$\varphi_K''(m) = \left( \frac{1}{K} \underset{i = 1}{\stackrel{K}{\sum}} \hspace{1mm} s(\sigma - a_i)^2 \right)^{-1}.$$
Since the variance $s$ is a continuous, positive function, the bound (\ref{lct_bnd1}) implies that $\log g_{K,m}(0)$ is uniformly bounded when $\sigma$ varies in a compact subset of $\R$ and the real numbers are bounded by some constant $L$. As an immediate consequence, we get the second part of Theorem \ref{lct}.

For the last part, we have
\begin{align} \label{truc}
\varphi_{K}''(m) - \psi_K''(m) &= \frac{1}{K}\frac{d^2}{dm^2}\log g_{K,m}(0) \notag \\
&= \frac{1}{K} \frac{1}{g_{K,m}(0)}\frac{d^2}{dm^2} g_{K,m}(0) - \frac{1}{K}\left(\frac{1}{g_{K,m}(0)} \frac{d}{dm} g_{K,m}(0) \right)^2 \notag \\
&= \frac{1}{K} \frac{1}{g_{K,m}(0)}\frac{d^2}{dm^2}\sigma \frac{d}{d\sigma} g_{K,m}(0) + \frac{1}{K}\frac{1}{g_{K,m}(0)} \left(\frac{d}{dm}\sigma\right)^2 \frac{d^2}{d\sigma^2}g_{K,m}(0) \notag \\
& \hspace{5mm} - \frac{1}{K}\left(\frac{1}{g_{K,m}(0)} \frac{d\sigma}{dm} \frac{d}{d\sigma} g_{K,m}(0) \right)^2 \notag \\
&\leq \frac{1}{K} \frac{1}{g_{K,m}(0)}\frac{d^2}{dm^2}\sigma \frac{d}{d\sigma} g_{K,m}(0) + \frac{1}{K}\frac{1}{g_{K,m}(0)} \left(\frac{d}{dm}\sigma\right)^2 \frac{d^2}{d\sigma^2}g_{K,m}(0) 
\end{align}

We can explicitly compute the derivatives of $\sigma$ with respect to $m$. Direct calculation (see \ref{derivative_m} below) yields
$$\frac{d}{d\sigma} m = \frac{1}{K} \underset{i = 1}{\stackrel{K}{\sum}} \hspace{1mm} \frac{d}{d\sigma} m_i = \frac{1}{K} \underset{i = 1}{\stackrel{K}{\sum}} \hspace{1mm} s(\sigma - a_i)^2,$$
so that
\begin{equation} \label{bnd_sigma_1}
\frac{d^2}{dm^2} \varphi_K =\frac{d}{dm} \sigma = \left( \frac{1}{K} \underset{i = 1}{\stackrel{K}{\sum}} \hspace{1mm} s(\sigma - a_i)^2 \right)^{-1}.
\end{equation}

Another direct calculation yields
\begin{align}
\frac{d^2}{dm^2}\sigma &= \frac{d}{dm} \left( \frac{1}{K} \underset{i = 1}{\stackrel{K}{\sum}} \hspace{1mm} s(\sigma - a_i)^2 \right)^{-1} \notag \\
&= -\left( \frac{1}{K} \underset{i = 1}{\stackrel{K}{\sum}} \hspace{1mm} \frac{d}{dm}s(\sigma - a_i)^2 \right)\left( \frac{1}{K} \underset{i = 1}{\stackrel{K}{\sum}} \hspace{1mm} s(\sigma - a_i)^2 \right)^{-2} \notag \\
&= -\left( \frac{1}{K} \underset{i = 1}{\stackrel{K}{\sum}} \hspace{1mm} \int{(x - m_i)^3\mu_{\sigma, i}(dx)} \right)\left( \frac{1}{K} \underset{i = 1}{\stackrel{K}{\sum}} \hspace{1mm} s(\sigma - a_i)^2 \right)^{-3}. \notag
\end{align}

This equality, (\ref{bor}) and (\ref{uniformity}) immediately imply
\begin{align} \label{bnd_sigma_2}
\left|\frac{d^2}{dm^2}\sigma \right| &\lesssim \left( \frac{1}{K} \underset{i = 1}{\stackrel{K}{\sum}} \hspace{1mm} s(\sigma - a_i)^2 \right)^{-3/2} \frac{\underset{i}{\sup} \hspace{1mm} s(\sigma - a_i)^3}{\underset{i}{\inf} \hspace{1mm} s(\sigma - a_i)^3} \notag \\
&\lesssim \left( \frac{1}{K} \underset{i = 1}{\stackrel{K}{\sum}} \hspace{1mm} s(\sigma - a_i)^2 \right)^{-3/2}.
\end{align}

Plugging the bounds (\ref{lct_bnd1}), (\ref{lct_bnd2}), (\ref{lct_bnd3}), (\ref{bnd_sigma_1}) and (\ref{bnd_sigma_2}) into (\ref{truc}) yields
$$\varphi_K''(m) - \psi_K''(m) \leq \frac{C}{K} \left( \frac{1}{K} \underset{i = 1}{\stackrel{K}{\sum}} \hspace{1mm} s(\sigma - a_i)^2 \right)^{-1}.$$
But since we have (\ref{bnd_sigma_1}), which states that
$$\varphi_K''(m) = \left( \frac{1}{K} \underset{i = 1}{\stackrel{K}{\sum}} \hspace{1mm} s(\sigma - a_i)^2 \right)^{-1},$$
we get 
$$\psi_K''(m) \geq \left( \frac{1}{K} \underset{i = 1}{\stackrel{K}{\sum}} \hspace{1mm} s(\sigma - a_i)^2 \right)^{-1} - \frac{C}{K}\left( \frac{1}{K} \underset{i = 1}{\stackrel{K}{\sum}} \hspace{1mm} s(\sigma - a_i)^2 \right)^{-1}$$

$$\geq \frac{1}{2}\left( \frac{1}{K} \underset{i = 1}{\stackrel{K}{\sum}} \hspace{1mm} s(\sigma - a_i)^2 \right)^{-1}$$
for $K$ large enough. Since $\sigma \rightarrow s(\sigma)$ is bounded above (see Lemma \ref{lem2} ), we immediately obtain (\ref{bnd_derivative_psi})

\end{proof}

Now all that is left is to prove Proposition \ref{bounds_lct}.

\begin{nota}
We denote by $h(\sigma, \xi)$ the Fourier transform of the law of $X - m$, with $X$ distributed according to $\mu_{\sigma}$, and $m$ its mean, that is
$$h(\sigma, \xi) := \int_{\mathbb{R}}{\exp(i(x - m)\xi)\mu_{\sigma}(dx)}.$$
\end{nota}

By using the inverse Fourier transform, we have the explicit formula for $g_{K,m}(0)$ :
\begin{equation} \label{cra1}
g_{K,m}(0) = \frac{1}{2\pi}\int_{\mathbb{R}}{\underset{i = 1}{\stackrel{K}{\prod}} h(\sigma - \lambda_i, \frac{\xi}{\sqrt{K}}) d\xi}.
\end{equation}

To control this quantity, we will split this integral in two : one integral over large values of $\xi$, and one over small values of $\xi$.

\begin{lem} \label{lem1}
Let $A(C_0)$ be the set of probability measures $\nu$ on $\mathbb{R}$ such that 
$$\int_{\mathbb{R}}{|x|\nu(dx)} \leq C_0$$
and
$$\left|\int_{\mathbb{R}}{\exp(ix\xi)\nu(dx)}\right| \leq \frac{C_0}{|\xi|}  \hspace{1mm} \text{for all} \hspace{1mm} \xi \in \mathbb{R}$$
for some $C_0 < \infty$. Then, for any $\delta > 0$, there exists $\lambda < 1$ such that 
$$\left|\int_{\mathbb{R}}{\exp(ix\xi)\nu(dx)}\right| \leq \lambda  \hspace{1mm} \text{for all} \hspace{1mm} \xi \geq \delta$$
for all $\nu$ in $A(C_0)$.
\end{lem}

\begin{proof} [Proof of Lemma \ref{lem1}]
We can assume without loss of generality that $\delta < 1$. 
Because of assumption (\ref{f1}), we have for any $|\xi| \geq 2C_0$ that $\left|\int_{\mathbb{R}}{\exp(ix\xi)\nu(dx)}\right| \leq \frac{1}{2}$, so it is enough to show that
$$\left|\int_{\mathbb{R}}{\exp(ix\xi)\nu(dx)}\right| \leq \lambda  \hspace{1mm} \text{for all} \hspace{1mm} \xi \in [\delta, 2C_0 ].$$

Lets assume that this result doesn't hold. Then there exists a sequence $(\nu_n)$ of probability measures in $A(C_0)$ and a sequence $(\xi_n) \in [\delta, 2C_0 ]^{\mathbb{N}}$ such that 
\begin{equation} \label{c1}
\underset{n \uparrow \infty}{\liminf} \hspace{1mm} \left|\int_{\mathbb{R}}{\exp(ix\xi_n)\nu_n(dx)}\right| \geq 1.
\end{equation}
Since for all $n$ we have $\int_{\mathbb{R}}{|x|\nu_n(dx)} \leq C_0$, we have compactness of the sequence $(\nu_n)$ for weak convergence, and we can extract a subsequence such that
$$\int{f(x)\nu_n(dx)} \longrightarrow \int{f(x)\nu_{\infty}(dx)}$$
for all bounded continuous function $f$, and 
$$\xi_n \longrightarrow \xi_{\infty}.$$
Therefore, since $|\exp(ix\xi_n) - \exp(ix\xi_{\infty})| \leq |x||\xi_n - \xi_{\infty}|$, we also have
$$\int_{\mathbb{R}}{\exp(ix\xi_n)\nu_n(dx)} \longrightarrow \int_{\mathbb{R}}{\exp(ix\xi_{\infty})\nu_{\infty}(dx)},$$
so that (\ref{c1}) saturates to 
$$\int_{\mathbb{R}}{\exp(ix\xi_{\infty})\nu_{\infty}(dx)} = 1.$$
The classical result on the equality cases in the triangle inequality then tells us that $x \rightarrow \exp(ix\xi_{\infty})$ is $\nu_{\infty}$-almost surely a constant $\zeta \in S^1$, and therefore 
$$\int_{\mathbb{R}}{\exp(ikx\xi_{\infty})\nu_{\infty}(dx)} = \zeta^k \text{ for all } k \in \mathbb{N}.$$
But since $\xi_{\infty} \neq 0$, $|k\xi_{\infty}|$ goes to infinity when $k$ goes to infinity, and since $\nu_{\infty} \in A(C_0)$, $\int_{\mathbb{R}}{\exp(ikx\xi_{\infty})\nu_{\infty}(dx)}$ should go to zero. We therefore have a contradiction, and (\ref{c1}) cannot hold.

\end{proof}

\begin{lem} \label{lem2}
There exists a constant $C > 0$ such that for all $\sigma \in \mathbb{R}$ we have
$$s(\sigma) = \int_{\mathbb{R}}{(x - m(\sigma))^2 \mu_{\sigma}(dx)} \leq C.$$
\end{lem}

Since, by this lemma, the variances $s(\sigma)$ are bounded above independently of $\sigma$, we can apply Lemma \ref{lem1}, so for any $\delta > 0$, there exists $\lambda < 1$ such that if $|\xi| > \delta$, then for any $\sigma$ we have $|h(\sigma, \xi)| \leq \lambda$.

\begin{proof} [Proof of Lemma \ref{lem2}]
It is easy to show (with a combination of the Bakry-Emery and Holley-Stroock criterions) that the measures $\mu_{\sigma}$ satisfy a spectral gap inequality with a uniform constant, yielding the desired uniform upper bound on the variances. 
\end{proof}

\begin{lem} \label{lem?}
Under the assumptions of Proposition \ref{bounds_lct}, there exists a family of complex-valued functions $v_{\sigma}$ such that
\begin{equation} \label{as1}
\int_{\mathbb{R}}{\exp(i(x - m(\sigma))\xi)\mu_{\sigma}(dx)} = \exp(-v_{\sigma}(\xi))
\end{equation}
and 
\begin{equation} \label{bnd1}
|v_{\sigma}(\xi) - \frac{s(\sigma)^2}{2}\xi^2| \leq C s(\sigma)^3|\xi|^3
\end{equation}
for all $\xi$ small enough, independently of $\sigma$.
\end{lem}

\begin{proof}
We can use (\ref{as1}) as a definition of the functions $v_{\sigma}$, so we just have to prove (\ref{bnd1}). We have 
$$\frac{d^k}{d\xi^k}\int_{\mathbb{R}}{\exp(i(x - m(\sigma))\xi)\mu_{\sigma}(dx)} = i^k\int_{\mathbb{R}}{(x - m)^k\exp(i(x - m)\xi)\mu_{\sigma}(dx)},$$
and, since we assumed \ref{bor}, we get
$$\left|\frac{d^k}{d\xi^k}\int_{\mathbb{R}}{\exp(i(x - m(\sigma))\xi)\mu_{\sigma}(dx)}\right| \leq Cs^k,$$
and a Taylor expansion of the function $v_{\sigma}$ then gives us the desired result.
\end{proof}

Let us now fix $\delta > 0$ small enough, so that the bound of Lemma \ref{lem?} holds for all $|\xi| \leq \delta$. We split (\ref{cra1}) in two integrals, depending on the value of $|\xi|$ : 
\begin{align}
g_{K,m}(0) &= \frac{1}{2\pi}\int_{\mathbb{R}}{\underset{i = 1}{\stackrel{K}{\prod}} h(\sigma - \lambda_i, \frac{\xi}{\sqrt{K}}) d\xi} \notag \\
&= \frac{1}{2\pi}\int_{\frac{1}{\sqrt{K}}|\xi| \leq \delta}{\underset{i = 1}{\stackrel{K}{\prod}} h(\sigma - \lambda_i, \frac{\xi}{\sqrt{K}}) d\xi} \notag \\
& \hspace{1cm} + \frac{1}{2\pi}\int_{\frac{1}{\sqrt{K}}|\xi| > \delta}{\underset{i = 1}{\stackrel{K}{\prod}} h(\sigma - \lambda_i, \frac{\xi}{\sqrt{K}}) d\xi} 
\end{align}

We will show that the first term is of leading order, and converges to 1. Let us first take care of the second term. Since we have a uniform bound on the variances of the probability measures $\mu_{\sigma}$, we can apply Lemma \ref{lem2}, and obtain the bound
$$h(\sigma, \xi) \leq \lambda$$
for some $\lambda < 1$, for all $|\xi| \leq \delta$, independently of $\sigma$. We use this estimate on $K - 2$ of the $K$ elements of the product, and estimate (\ref{f1}) on the last two, to obtain
\begin{align}
\left|\underset{i = 1}{\stackrel{K}{\prod}} h(\sigma - \lambda_i, \frac{\xi}{\sqrt{K}})\right| &\leq C \lambda^{K-2} \left(\frac{1}{1 + |\xi|/\sqrt{K}}\right)^2 \notag \\
&\leq C \lambda^{K-2} \frac{K}{K + \xi^2} \notag \\
&\leq C \lambda^{K-2} \frac{K}{1 + \xi^2} 
\end{align}

We therefore have
\begin{align}
\left|\frac{1}{2\pi}\int_{\frac{1}{\sqrt{K}}|\xi| > \delta}{\underset{i = 1}{\stackrel{K}{\prod}} h(\sigma - \lambda_i, \frac{\xi}{\sqrt{K}}) d\xi}\right| &\leq C K \lambda^{K-2} \int_{\mathbb{R}}{\frac{1}{1 + \xi^2}d\xi} \notag \\
&\leq \frac{C}{K} 
\end{align}

where the last estimate is because we have $\lambda < 1$, so this term is exponentially small in $K$, and therefore negligible when compared to $1/K$. All that is now left is to take care of the integral over small values of $\xi$. Recalling definition (\ref{as1}), we have

$$\frac{1}{2\pi}\int_{\frac{1}{\sqrt{K}}|\xi| \leq \delta}{\underset{i = 1}{\stackrel{K}{\prod}} h(\sigma - \lambda_i, \frac{\xi}{\sqrt{K}}) d\xi} = \frac{1}{2\pi}\int_{\frac{1}{\sqrt{K}}|\xi| \leq \delta}{\exp\left(-\underset{i = 1}{\stackrel{K}{\sum}}v_{\sigma - \lambda_i}(\xi/\sqrt{K})\right)d\xi} $$

Since we assumed $\delta$ small enough, according to Lemma \ref{lem?}, we have, independently of $\sigma$ and for all $\xi$ such that $|\xi|/\sqrt{K} \leq \delta$
$$|v_{\sigma}(\xi) - \frac{s(\sigma)^2}{2}\xi^2| \leq Cs(\sigma)^3|\xi|^3.$$
This bound implies that
\begin{align}
\text{Re}(v_{\sigma}(\xi)) &\geq \frac{s(\sigma)^2}{2}\xi^2 - Cs(\sigma)^3|\xi|^3 \notag \\
&\geq \frac{s(\sigma)^2}{4}\xi^2 \notag
\end{align}
independently of $\sigma$, as long as we have chosen $\delta$ small enough (since $s(\sigma)$ is bounded above). But since $y \rightarrow \exp(y)$ is Lipschitz continuous on $\{ \text{Re}(y) \leq -c\xi^2 \}$, with constant $\exp(-c\xi^2)$, we have
\begin{align} 
&\left|\exp\left(-\underset{i = 1}{\stackrel{K}{\sum}}v_{\sigma - \lambda_i}(\xi/\sqrt{K})\right) - \exp \left(-\frac{\sum s(\sigma - \lambda_i)^2}{2K}\xi^2\right)\right| \notag \\
&\hspace{5mm}\leq C\left(\sum s(\sigma - \lambda_i)^3 \right)\frac{|\xi|^3}{K^{3/2}} \exp\left(-\frac{\sum s(\sigma - \lambda_i)^2}{4K}\xi^2\right) \notag \\
&\hspace{1cm}\leq C\left(\frac{1}{K}\sum s(\sigma - \lambda_i)^3 \right)\frac{|\xi|^3}{\sqrt{K}}\exp\left(-\frac{\sum s(\sigma - \lambda_i)^2}{4K}\xi^2\right)
\end{align}
Consequently, we have
\begin{align} 
&\left|\frac{1}{2\pi}\int_{\frac{1}{\sqrt{K}}|\xi| \leq \delta}{\underset{i = 1}{\stackrel{K}{\prod}} h(\sigma - \lambda_i, \frac{\xi}{\sqrt{K}}) d\xi} - \frac{1}{2\pi}\int_{\frac{1}{\sqrt{K}}|\xi| \leq \delta}{\exp \left(-\frac{\sum s(\sigma - \lambda_i)^2}{2K}\xi^2\right)d\xi} \right| \notag \\
&\hspace{5mm} \leq C\left(\frac{1}{K}\sum s(\sigma - \lambda_i)^3 \right)\frac{1}{\sqrt{K}}\int_{\frac{1}{\sqrt{K}}|\xi| \leq \delta}{|\xi|^3\exp\left(-\frac{\sum s(\sigma - \lambda_i)^2}{4K}\xi^2\right)d\xi} \notag \\
&\hspace{5mm}\leq C\left(\frac{1}{K}\sum s(\sigma - \lambda_i)^3 \right)\frac{1}{\sqrt{K}}\int_{\mathbb{R}}{|\xi|^3\exp\left(-\frac{\sum s(\sigma - \lambda_i)^2}{4K}\xi^2\right)d\xi} \notag \\
&\hspace{5mm}\leq \frac{C}{\sqrt{K}}\frac{\left(\frac{1}{K}\sum s(\sigma - \lambda_i)^3 \right)}{\left(\frac{1}{K}\sum s(\sigma - \lambda_i)^2 \right)^{2}} \notag \\
&\hspace{5mm}\leq \frac{C}{\sqrt{K}} \left(\frac{1}{K}\sum s(\sigma - a_i)^2 \right)^{-1/2} \frac{\sup s(\sigma - a_i)^3}{\inf s(\sigma - a_i)^3} \notag \\
&\hspace{5mm}\leq \frac{C}{\sqrt{K}} \left(\frac{1}{K}\sum s(\sigma - a_i)^2 \right)^{-1/2}.
\end{align}

Since $\int_{\frac{1}{\sqrt{K}}|\xi| > \delta}{\exp \left(-\frac{\sum s(\sigma - \lambda_i)^2}{2K}\xi^2\right)d\xi}$ is exponentially small in $K$, we finally obtain
$$\left|\frac{1}{2\pi}\int_{\frac{1}{\sqrt{K}}|\xi| \leq \delta}{\underset{i = 1}{\stackrel{K}{\prod}} h(\sigma - \lambda_i, \frac{\xi}{\sqrt{K}}) d\xi} - \frac{1}{2\pi}\int_{\mathbb{R}}{\exp \left(-\frac{\sum s(\sigma - \lambda_i)^2}{2K}\xi^2\right)d\xi} \right|$$
$$\leq \frac{C}{\sqrt{K}}\left(\frac{1}{K}\sum s(\sigma - \lambda_i)^2 \right)^{-1/2}.$$
In the end, what we obtain is
\begin{equation}
\left|g_{K,m}(0) - \left(\frac{2\pi}{K}\sum s(\sigma - \lambda_i)^2 \right)^{-1/2} \right| \leq \frac{C}{\sqrt{K}}\left(\frac{1}{K}\sum s(\sigma - \lambda_i)^2 \right)^{-1/2}.
\end{equation}

To bound the derivatives of $g_{K,m}(0)$ with respect to $\sigma$, we will need the following estimates : 
\begin{lem}
Under the assumptions of Proposition \ref{bounds_lct} we have, uniformly in $\sigma$
\begin{equation} \label{es1}
\left| \frac{d}{d\sigma}h(\sigma, \xi) \right| \leq Cs^3|\xi|^2;
\end{equation}
\begin{equation} \label{es2}
\left| \frac{d^2}{d\sigma^2}h(\sigma, \xi) \right| \leq Cs^4(1 + s^2|\xi|^2)|\xi|^2.
\end{equation}
\end{lem}

\begin{proof}
To prove this lemma, we will rely on the following identity :
\begin{equation}
\frac{d}{d\sigma}\int{f(x)\mu_{\sigma}(dx)} = \int{(x - m)f(x)\mu_{\sigma}(dx)}.
\end{equation}
which can be directly computed easily.
Using this identity, we can show that
\begin{equation} \label{derivative_m}
\frac{d}{d\sigma}m = \frac{d}{d\sigma}\int{x \mu_{\sigma}(dx)} = \int{x(x - m)\mu_{\sigma}(dx)} = s^2,
\end{equation}
\begin{equation} \label{derivative_s2}
\frac{d}{d\sigma}s^2 = \frac{d}{d\sigma}\int{(x - m)^2\mu_{\sigma}(dx)} = \int{(x - m)^3\mu_{\sigma}(dx)}.
\end{equation}

To bound the derivatives in $\sigma$ of the Fourier transform, we will use a Taylor expansion in $0$. We have
\begin{equation}
\frac{d^k}{d\xi^k}\frac{d}{d\sigma}h(\sigma, \xi) = \frac{d}{d\sigma}\frac{d^k}{d\xi^k}h(\sigma, \xi) = i^k\frac{d}{d\sigma}\int{(x-m)^k\exp(i(x-m)\xi)\mu_{\sigma}(dx)}.
\end{equation}
In particular, $\frac{d}{d\sigma}h(\sigma, \xi)$ vanishes to the first order in $0$ in $\xi$, and, by assumption (\ref{bor}),
$$\left| \frac{d^2}{d\xi^2}\frac{d}{d\sigma}h(\sigma, \xi) \right| = \left| \frac{d}{d\sigma}s^2 \right| \leq Cs^3,$$
so that 
\begin{equation} 
\left| \frac{d}{d\sigma}h(\sigma, \xi) \right| \leq Cs^3|\xi|^2.
\end{equation}
In the same way,
\begin{equation}
\frac{d^k}{d\xi^k}\frac{d^2}{d\sigma^2}h(\sigma, \xi) = \frac{d^2}{d\sigma^2}\frac{d^k}{d\xi^k}h(\sigma, \xi) = i^k\frac{d^2}{d\sigma^2}\int{(x-m)^k\exp(i(x-m)\xi)\mu_{\sigma}(dx)},
\end{equation}
so that $\frac{d^2}{d\sigma^2}h(\sigma, \xi)$ also vanishes to the first order in $0$ in $\xi$, and
\begin{align}
& \frac{d^2}{d\xi^2}\frac{d^2}{d\sigma^2}h(\sigma, \xi) = -\frac{d^2}{d\sigma^2}\int{(x-m)^2\exp(i(x-m)\xi)\mu_{\sigma}(dx)} \notag \\
&= \frac{d}{d\sigma}\Big[ -\int{(x-m)^3\exp(i(x-m)\xi)\mu_{\sigma}(dx)} \\
& \qquad  + s^2\int{(2(x-m)  + i\xi(x-m)^2)\exp(i(x-m)\xi)\mu_{\sigma}(dx)} \Big] \notag \\
&= - \int{(x-m)^4\exp(i(x-m)\xi)\mu_{\sigma}(dx)} \\
& \hspace{3mm} + s^2\int{(3(x-m)^2  + i\xi(x-m)^3)\exp(i(x-m)\xi)\mu_{\sigma}(dx)} \notag \\
& \hspace{3mm} + \left(\int{(x-m)^3\mu_{\sigma}(dx)}\right)\left(\int{(2(x-m)  + i\xi(x-m)^2)\exp(i(x-m)\xi)\mu_{\sigma}(dx)}\right) \notag \\
& \hspace{3mm} -2s^4h(\sigma, \xi) -4i\xi s^4\int{(x-m)\exp(i(x-m)\xi)\mu_{\sigma}(dx)} \notag \\
& \hspace{3mm} +s^4\xi^2\int{(x-m)^2\exp(i(x-m)\xi)\mu_{\sigma}(dx)}, \notag 
\end{align}
so that, in the end, using assumption \ref{bor}, we get
\begin{equation} 
\left| \frac{d^2}{d\sigma^2}h(\sigma, \xi) \right| \leq Cs^4(1 + s^2|\xi|^2)|\xi|^2.
\end{equation}
\end{proof}

Let us now compute the derivatives with respect to $\sigma$ of $g_{K,\sigma}$

\begin{align} \label{first_derivative_g}
\frac{d}{d\sigma} g_{K,\sigma}(0) &= \frac{d}{d\sigma} \frac{1}{2\pi}\int_{\mathbb{R}}{\underset{i = 1}{\stackrel{K}{\prod}} h(\sigma - \lambda_i, \frac{\xi}{\sqrt{K}}) d\xi} \notag \\
&= \frac{1}{2\pi}\int_{\mathbb{R}}{\underset{i=1}{\stackrel{K}{\sum}} \hspace{1mm} \frac{d}{d\sigma}h(\sigma - \lambda_i, \frac{\xi}{\sqrt{K}}) \underset{j \neq i}{\prod} \hspace{1mm} h(\sigma - \lambda_j, \frac{\xi}{\sqrt{K}}) d\xi}; \notag 
\end{align}

\begin{align}
& \frac{d^2}{d\sigma^2} g_{K,\sigma}(0) = \frac{1}{2\pi}\int_\mathbb{R}{\underset{i=1}{\stackrel{K}{\sum}} \hspace{1mm} \frac{d^2}{d\sigma^2}h(\sigma - \lambda_i, \frac{\xi}{\sqrt{K}}) \underset{j \neq i}{\prod} \hspace{1mm} h(\sigma - \lambda_j, \frac{\xi}{\sqrt{K}}) d\xi} \notag \\
&+ \frac{1}{2\pi}\int_\mathbb{R}{\underset{i,j=1, i \neq j}{\stackrel{K}{\sum}} \hspace{1mm} \frac{d}{d\sigma}h(\sigma - \lambda_i, \frac{\xi}{\sqrt{K}}) \frac{d}{d\sigma}h(\sigma - \lambda_j, \frac{\xi}{\sqrt{K}}) \underset{k \neq i,j}{\prod} \hspace{1mm} h(\sigma - \lambda_k, \frac{\xi}{\sqrt{K}}) d\xi}. 
\end{align}

We are seeking to bound $\frac{d^2}{d\sigma^2} g_{K,\sigma}(0)$. Once again, we separate the integrals into inner and outer integrals. Once more, the leading order term will be the inner integrals. For the outer integrals, we have, by using Lemma \ref{lem1} on $K-6 $ of the $K-1$ factors, estimate (\ref{f1}) on the remaining 6 and estimate (\ref{es2}) on the second derivative of $h$ to obtain (still remembering that $s(\sigma)$ is bounded above independently of $\sigma$)
\begin{align}
\frac{1}{2\pi}&\left|\int_{|\xi|/\sqrt{K} > \delta}{\underset{i=1}{\stackrel{K}{\sum}} \hspace{1mm} \frac{d^2}{d\sigma^2}h(\sigma - \lambda_i, \frac{\xi}{\sqrt{K}}) \underset{j \neq i}{\prod} \hspace{1mm} h(\sigma - \lambda_j, \frac{\xi}{\sqrt{K}}) d\xi}\right|  \notag \\
&\leq C \underset{i=1}{\stackrel{K}{\sum}} \left|\int_{|\xi|/\sqrt{K} > \delta}{s(\sigma - \lambda_i)^4(1 + s(\sigma - \lambda_i)^2\frac{|\xi|^2}{K})\frac{|\xi|^2}{K} \left(\frac{K}{K + |\xi|^2}\right)^6 \lambda^{K-6} d\xi}\right| \notag \\
&\leq \lambda^{K-6}\frac{C}{K}\underset{i=1}{\stackrel{K}{\sum}}s(\sigma - \lambda_i)^4 \int_{\mathbb{R}}{\frac{1}{1 + \xi^2} d\xi} \notag \\
&\leq \frac{C}{K}.
\end{align}

The same strategy (using estimate (\ref{es1})) allows us to show that 
\begin{align*}
&\frac{1}{2\pi}\left|\int_{|\xi|/\sqrt{K} > \delta}{\underset{i,j=1, i \neq j}{\stackrel{K}{\sum}} \hspace{1mm} \frac{d}{d\sigma}h(\sigma - \lambda_i, \frac{\xi}{\sqrt{K}}) \frac{d}{d\sigma}h(\sigma - \lambda_j, \frac{\xi}{\sqrt{K}}) \underset{k \neq i,j}{\prod} \hspace{1mm} h(\sigma - \lambda_k, \frac{\xi}{\sqrt{K}}) d\xi} \right| \\
& \qquad  \leq \frac{C}{K}.
\end{align*}

We will now show that the inner integrals are bounded. 

\begin{align}
& \frac{1}{2\pi}\left|\int_{|\xi|/\sqrt{K} \leq \delta}{\underset{i=1}{\stackrel{K}{\sum}} \hspace{1mm} \frac{d^2}{d\sigma^2}h(\sigma - \lambda_i, \frac{\xi}{\sqrt{K}}) \underset{j \neq i}{\prod} \hspace{1mm} h(\sigma - \lambda_j, \frac{\xi}{\sqrt{K}}) d\xi}\right|  \notag \\
 &\leq C\underset{i=1}{\stackrel{K}{\sum}} \hspace{1mm} \Big| \int_{|\xi|/\sqrt{K} \leq \delta} s(\sigma - \lambda_i)^4\left(1 + s(\sigma - a_i)^2\frac{|\xi|^2}{K} \right)\frac{|\xi|^2}{K} \\
   & \qquad \qquad  \times \exp \left(-\underset{j \neq i}{\sum} \hspace{1mm} v_{\sigma-a_j}(\xi/\sqrt{K}) \right) d\xi \Big| \notag \\
 &\leq C\underset{i=1}{\stackrel{K}{\sum}} \hspace{1mm} \int_{|\xi|/\sqrt{K} \leq \delta}s(\sigma - a_i)^4(1 + s(\sigma - a_i)^2\delta^2) \frac{|\xi|^2}{K} \\
 & \qquad \qquad \times  \exp \left(-\underset{j \neq i}{\sum} \hspace{1mm} \frac{s(\sigma - a_j)^2}{2K}|\xi|^2 \right) d\xi  \notag \\
&\hspace{3mm} + C\underset{i=1}{\stackrel{K}{\sum}} \hspace{1mm} \int_{|\xi|/\sqrt{K} \leq \delta}s(\sigma - a_i)^4(1 + s(\sigma - a_i)^2\delta^2) \frac{|\xi|^5}{K^{5/2}}\left( \underset{j \neq i}{\sum} \hspace{1mm} s(\sigma - a_j)^3 \right) \\
& \qquad \qquad \times \exp \left(-\underset{j \neq i}{\sum} \hspace{1mm} \frac{s(\sigma - a_j)^2}{4K}|\xi|^2 \right) d\xi  \notag \\
&\leq \frac{C}{K} \underset{i=1}{\stackrel{K}{\sum}} \hspace{1mm} s(\sigma - a_i)^4 \int_{\R}{|\xi|^2\exp \left(-\underset{j \neq i}{\sum} \hspace{1mm} \frac{s(\sigma - \lambda_j)^2}{2K}|\xi|^2 \right) d\xi} \notag \\
& \hspace{3mm} + \frac{C}{K^{3/2}} \underset{i=1}{\stackrel{K}{\sum}} \hspace{1mm} s(\sigma - a_i)^4\left( \frac{1}{K}\underset{j \neq i}{\sum} \hspace{1mm} s(\sigma - a_j)^3 \right)  \int_{\R}{|\xi|^5 \exp \left(-\underset{j \neq i}{\sum} \hspace{1mm} \frac{s(\sigma - a_j)^2}{4K}|\xi|^2 \right) d\xi}  \notag \\
&\leq \frac{C}{K} \underset{i=1}{\stackrel{K}{\sum}} \hspace{1mm} \frac{s(\sigma - a_i)^4}{\left(\underset{j \neq i}{\sum} \hspace{1mm} \frac{s(\sigma - a_j)^2}{2K} \right)^{3/2}}\int_{\R}{u^2\exp(-u^2)du} \notag \\
& \hspace{3mm} + \frac{C}{K^{3/2}} \underset{i=1}{\stackrel{K}{\sum}} \hspace{1mm} s(\sigma - a_i)^4 \left( \frac{1}{K}\underset{j \neq i}{\sum} \hspace{1mm} s(\sigma - a_j)^3 \right)\left(\underset{j \neq i}{\sum} \hspace{1mm} \frac{s(\sigma - a_j)^2}{4K} \right)^{-3} \\
& \qquad \qquad \times \int_{\R}{|u|^5\exp(-u^2)du}  \notag \\
&\leq \frac{C}{K} \left(\frac{1}{K}\underset{i = 1}{\stackrel{K}{\sum}} s(\sigma - a_i)^2 \right)^{1/2} \underset{i = 1}{\stackrel{K}{\sum}} \frac{\left(\underset{j \neq i}{\sum} s(\sigma - a_j)^2 \right)^{1/2}}{\left(\underset{k = 1}{\stackrel{K}{\sum}} s(\sigma - a_k)^2 \right)^{1/2}} \frac{\underset{k}{\sup} \hspace{1mm} s(\sigma - a_k)^4}{\underset{k}{\inf} \hspace{1mm} s(\sigma - a_k)^4} \notag \\
&\hspace{3mm} + \frac{C}{K^{3/2}} \left(\frac{1}{K}\underset{i = 1}{\stackrel{K}{\sum}} s(\sigma - a_i)^2 \right)^{1/2} \underset{i = 1}{\stackrel{K}{\sum}}\frac{\left(\underset{j \neq i}{\sum} s(\sigma - a_j)^2 \right)^{1/2}}{\left(\underset{k = 1}{\stackrel{K}{\sum}} s(\sigma - a_k)^2 \right)^{1/2}} \frac{\underset{k}{\sup} \hspace{1mm} s(\sigma - a_k)^7}{\underset{k}{\inf} \hspace{1mm} s(\sigma - a_k)^7} \notag \\
&\leq C   \left(\frac{1}{K}\underset{i = 1}{\stackrel{K}{\sum}} s(\sigma - a_i)^2 \right)^{1/2}.
\end{align}

We can show in the same way that 

\begin{align}
\frac{1}{2\pi}&\Big|\int_{|\xi|/\sqrt{K} \leq \delta}\underset{i,j=1, i \neq j}{\stackrel{K}{\sum}} \hspace{1mm} \frac{d}{d\sigma}h(\sigma - \lambda_i, \frac{\xi}{\sqrt{K}}) \frac{d}{d\sigma}h(\sigma - \lambda_j, \frac{\xi}{\sqrt{K}}) \\ 
& \qquad \times \underset{k \neq i,j}{\prod} \hspace{1mm} h(\sigma - \lambda_k, \frac{\xi}{\sqrt{K}}) d\xi \Big| \notag \\
&\leq C \hspace{1mm} \underset{i,j=1, i \neq j}{\stackrel{K}{\sum}} \hspace{1mm} \int_{|\xi|/\sqrt{K} \leq \delta}{s(\sigma - \lambda_i)^3s(\sigma - \lambda_j)^3\frac{|\xi|^4}{K^2}\exp \left( -\underset{k \neq i,j}{\sum} \hspace{1mm} v_{\sigma-\lambda_k}(\xi/\sqrt{K}) \right) d\xi} \notag \\
&\leq C   \left(\frac{1}{K}\underset{i = 1}{\stackrel{K}{\sum}} s(\sigma - a_i)^2 \right)^{1/2}.
\end{align}

The same technique applied to (\ref{first_derivative_g}) also yields (\ref{lct_bnd3}).

Combining the bounds on the inner and outer integrals yields (\ref{lct_bnd2}).


\section{Some classical criteria for the LSI}

Here we recall the three main criteria that are commonly used to obtain logarithmic Sobolev inequalities.

\vspace{3mm}

\textbf{Criterion I} (Tensorization principle)

If $\mu_1 \in \textsl{P}(X_1)$ and $\mu_2 \in \textsl{P}(X_2)$ satisfy LSI($\rho_1$) and LSI($\rho_2$) respectively, then $\mu_1 \otimes \mu_2$ satisfies LSI($\min \{ \rho_1, \rho_2 \}$).

\vspace{3mm}

\textbf{Criterion II} (Holley-Stroock perturbation lemma)

Let $\mu \in \textsl{P}(X)$ satisfy LSI($\rho$) and let $\delta \psi : X \rightarrow \mathbb{R}$ be a bounded function. Let $\tilde{\mu} \in \textsl{P}(X)$ be defined through 
$$\frac{d\tilde{\mu}}{d\mu}(x) = \frac{1}{Z}\exp(-\delta\psi(x)).$$
Then $\tilde{\mu}$ satisfies LSI($\tilde{\rho}$), where 
$$\tilde{\rho} = \rho\exp(-\osc(\delta\psi)).$$

\vspace{3mm}

\textbf{Criterion III} (Bakry-Emery theorem)

Let $X$ be a $K-$ dimensional Riemannian manifold, let $H \in C^2(X)$ and let $\mu \in \textsl{P}(X)$ be defined by
$$\frac{d\mu}{d\textsl{H}^K}(x) = \frac{1}{Z}\exp(-H(x)).$$
If there is $\rho > 0$ such that $\Hess H (x) \geq \rho$ for all $x \in X$, then $\mu$ satisfies LSI($\rho$).

\vspace{1cm}

\textbf{Bibliography}

\begin{itemize}

\item
\label{BM}[BM]
	F. Barthe and E. Milman,
	Transference Principles for Log-Sobolev and Spectral-Gap with Applications to Conservative Spin Systems
	\textit{Comm. Math. Physics} \textbf{323}, 575-625, (2013).
	
\item
\label{Cap}[Cap]
	P. Caputo,
	Uniform Poincar\'e inequalities for unbounded conservative spin systems: the non-interacting case.
	\textit{Stochastic Processes and Applications} \textbf{106}, No. 2, 223-244 (2003).
	
\item
\label{CCL}[CCL]
	E. Carlen, D. Cordero-Erausquin and E. Lieb,
	Asymmetric covariance estimates of Brascamp-Lieb type and related inequalities for log-concave measures
	\textit{to appear in Ann. Inst. H. Poincar\'e Probab. Statist.}
	
\item
\label{Ch}[Ch]
	D. Chafai,
	Glauber versus Kawasaki for spectral gap and logarithmic Sobolev inequalities of some unbounded conservative spin systems 
	\textit{Markov Processes and Related Fields} \textbf{9}, 3 (2003) 341-362 
	
\item
\label{Du}[Du]
	R. Dudley,
	Real Analysis and Probability.
	
\item
\label{Gr}[Gr]
	L. Gross,
	Logarithmic Sobolev inequalities,
	\textit{Amer. J. Math.} \textbf{97}, 1061-1083 (1975).

\item
\label{Go}[Go]
	N. Gozlan, 
	A Characterization of Dimension-Free Concentration in Terms of Transportation Inequalities
	\textit{Ann. Probab.} \textbf{37}, Number 6 (2009), 2480-2498.
	
\item
\label{GOVW}[GOVW]
	N. Grunewald, F. Otto, C. Villani and M. G. Westdickenberg,
	A two-scale approach to logarithmic Sobolev inequalities and the hydrodynamic limit.
	\textit{Ann. Inst. H. Poincar\'e Probab. Statist}. 
	45 (2009), 2, 302--351.

\item
\label{GPV}[GPV]
	M.Z. Guo, G.C. Papanicolaou and S.R.S. Varadhan,
	Nonlinear Diffusion Limit for a System with Nearest Neighbor Interactions,
	\textit{Commun. Math. Phys.} 118, 31-59 (1988)
	
\item
\label{L}[L]
	M. Ledoux, 
	The Concentration of Measure Phenomenon,
	AMS, Math. Surveys and Monographs, \textbf{89}, Providence, Rhode Island, 2001.
	
\item
\label{LN}[LN]
	C. Landim and J. Noronha Neto, 
	Poincar\'e and logarithmic Sobolev inequality for Ginzburg-landau processes in random environment,
	\textit{Probab. Theory Relat. Fields} 131, 229-260 (2005)
	
\item
\label{LPY}[LPY]
	C. Landim, G. Panizo and H. T. Yau,
	Spectral gap and logarithmic Sobolev inequality for unbounded conservative spin systems. 
	\textit{Ann. Inst. H. Poincar\'e} 38, 739-777 (2002)
	
\item
\label{M}[M]
	G. Menz,
	LSI for Kawasaki dynamics with weak interaction,
	\textit{Commun. Math. Phys.} 307, 817-860 (2011)
	
\item
\label{MO}[MO]
	G. Menz and F. Otto, 
	Uniform logarithmic Sobolev inequalities for conservative spin systems with super-quadratic single-site potential.
	\textit{Ann. Probab.} \textbf{41}, Number 3B (2013), 2182-2224.
		
\item
\label{OV}[OV]
	F. Otto and C. Villani,
	Generalization of an Inequality by Talagrand and Links with the Logarithmic Sobolev Inequality,
	\textit{J. Funct. Analysis}, \textbf{243} (2001), pp. 121-157.
	
\item
\label{V1}[V1]
	C. Villani,
	Optimal Transport, Old and New.
	\textit{Grundlehren der mathematischen Wissenschaften},
	Vol. 338, Springer-Verlag, 2009.
	
\item
\label{Y}[Y]
	H.T. Yau,
	Relative Entropy and Hydrodynamics of Ginzburg-Landau Models, 
	\textit{Lett. Math. Phys.}, \textbf{22} (1991), 63-80.
\end{itemize}

\end{document}